\documentclass[11pt,reqno,a4paper]{amsart}

\usepackage{amsmath,amsfonts,amssymb,amsthm}
\usepackage{mathtools}
\usepackage{xparse}
\usepackage{tikz}
\usepackage{tikz-cd}
\usepackage{multirow,enumitem,array}
\usepackage{placeins,xspace,fix-cm}
\usepackage[utf8]{inputenc}
\usepackage[english]{babel}
\usepackage{graphicx}
\usepackage[nomessages]{fp}
\usepackage[hidelinks,linktoc=page]{hyperref}
\usepackage{bookmark,xurl}
\usepackage[all,2cell,cmtip]{xy}
\UseAllTwocells

\usepackage{geometry}
\newcolumntype{F}{>{$}c<{\hspace{-0.9ex}$}}
\newcolumntype{:}{>{$}m{0.8ex}<{$}}
\newcolumntype{R}{>{$}r<{$}}
\newcolumntype{C}{>{$}c<{$}}
\newcolumntype{L}{>{$}l<{$}}
\newcolumntype{N}{@{}>{$}l<{$}}
\newcommand{\linesep}[1]{\renewcommand{\arraystretch}{#1}}
\newlength\horspace
\newcommand{\h}[1][1.0]{\hspace*{#1\horspace}}
\newlength\verspace
\newcommand{\vsep}[1][1.0]{\vspace*{#1\verspace}\xspace}
\tikzset{iso/.style={draw=none,every to/.append style={edge node={node [sloped, allow upside down, auto=false]{$\cong$}}}}}
\tikzset{adjunction/.style={draw=none,every to/.append style={edge node={node [sloped, allow upside down, auto=false]{$\dashv$}}}}}
\tikzset{simeq/.style={draw=none,every to/.append style={edge node={node [sloped, allow upside down, auto=false]{$\simeq$}}}}}
\tikzset{simeqS/.style={draw=none,every to/.append style={edge node={node [sloped, allow upside down, auto=false]{$\raisebox{0.8em}{$\simeq$}$}}}}}
\tikzset{aiso/.style={simeqS,preaction={draw,->}}}
\tikzset{proarrowS/.style={draw=none,every to/.append style={edge node={node [sloped, allow upside down, auto=false]{\raisebox{1.4pt}{\small$\shortmid$}}}}}}
\tikzset{proarrow/.style={proarrowS,preaction={draw,->}}}
\tikzset{nullS/.style={draw=none,every to/.append style={edge node={node [sloped, allow upside down, auto=false]{\raisebox{-1.16 ex}{$\circ$}}}}}}
\tikzset{null/.style={nullS,preaction={draw,->}}}
\tikzset{dotdot/.style={dash pattern=on 0.25ex off 0.2ex, dash phase=0ex}}
\tikzset{RightA/.style={double distance=3.5pt,>={Implies},->},
	triple/.style={-,preaction={draw,RightA}},
	quadruple/.style={preaction={draw,RightA,shorten >=0pt},shorten >=1pt,-,double,double distance=0.2pt}}
\tikzset{Right/.style={double distance=1.7pt,>={Implies},->}}
\tikzset{simeqSRight/.style={draw=none,every to/.append style={edge node={node [sloped, allow upside down, auto=false]{$\raisebox{-1em}{\rotatebox{180}{$\simeq$}}$}}}}}
\tikzset{twoiso/.style={simeqSRight,preaction={draw,Right}}}

\newtheorem{theorem}{Theorem}[section]
\newtheorem{lemma}[theorem]{Lemma}
\newtheorem{proposition}[theorem]{Proposition}
\newtheorem{corollary}[theorem]{Corollary}

\theoremstyle{definition}
\newtheorem{definition}[theorem]{Definition}
\newtheorem{example}[theorem]{Example}
\newtheorem{construction}[theorem]{Construction}

\theoremstyle{remark}
\newtheorem{remark}[theorem]{Remark}

\def\nameit#1{\textrm{#1}~}
\def\thex{\nameit{Theorem}}
\def\prox{\nameit{Proposition}}
\def\corx{\nameit{Corollary}}
\def\lemx{\nameit{Lemma}}
\def\defx{\nameit{Definition}}
\def\remx{\nameit{Remark}}

\def\dfn#1{{\itshape #1}}

\newcommand{\refs}[1]{\textup{(}\ref{#1}\textup{)}}

\DeclareMathOperator{\coker}{coker}
\NewDocumentEnvironment{cd}{s O{6} O{6} b}{
	\IfBooleanF{#1}{\begin{equation*}}\begin{tikzcd}[row sep=#2ex,column sep=#3ex,ampersand replacement=\&]
			#4
		\end{tikzcd}\IfBooleanF{#1}{\end{equation*}}\ignorespacesafterend}{}

\newenvironment{enum}{\begin{enumerate}[label=$($\hspace{0.12ex}\roman*\hspace{0.075ex}$)$]}{\end{enumerate}}

\newenvironment{eqD*}{\begin{equation*}}{\end{equation*}\ignorespacesafterend}
\newcommand{\A}{\mathcal{A}}
\newcommand{\C}{\mathcal{C}}
\newcommand{\D}{\mathcal{D}}
\newcommand{\E}{\mathcal{E}}
\newcommand{\M}{\mathcal{M}}

\newcommand{\N}{\mathcal{N}}
\newcommand{\T}{\mathcal{T}}
\renewcommand{\S}{\mathcal{S}}
\newcommand{\U}{\mathcal{U}}
\newcommand{\V}{\mathcal{V}}
\newcommand{\F}{\mathcal{F}}
\newcommand{\0}{\mathbf{0}}

\newcommand{\I}{\mathcal{I}}

\newcommand{\Cat}{\mathbf{Cat}}
\newcommand{\Triang}{\mathbf{Triang}}
\newcommand{\AbCat}{\mathbf{AbCat}}
\renewcommand{\mod}{\mathbf{mod}}
\newcommand{\Ab}{\mathbf{Ab}}
\def\:{\colon}

\def\c{\circ}
\newcommand{\iso}{\cong}
\def\phi{\varphi}

\newcommand{\tom}{\rightarrowtail}
\newcommand{\toe}{\twoheadrightarrow}

\def\set#1#2{\left\{{#1}\left.\right|\,{#2}\right\}}

\newcommand{\cont}{\subseteq}
\newcommand{\contain}{\supseteq}

\newcommand{\HomC}[3]{{#1}\left({#2},\h[1]{#3}\right)}

\newcommand{\id}[1]{\operatorname{id}_{#1}}
\newcommand{\Id}[1]{\operatorname{Id}_{#1}}

\newcommand{\aar}[2][]{\xrightarrow[#1]{#2}}
\makeatletter
\newcommand{\aR}[2][]{
	\ext@arrow 0055{\Rightarrowfill@}{#1}{#2}}
\def\xLongrightarrowfill@{\arrowfill@\Relbar\Relbar\Longrightarrow}
\newcommand{\am}[2][]{
	\ext@arrow 0395\xmapstofill@{#1}{#2}}
\def\xlongmapstofill@{\arrowfill@\relbar\relbar\longmapsto}
\def\xlongrightarrowfill@{\arrowfill@\relbar\relbar\longrightarrow}
\newcommand{\aarr}[2][]{
	\ext@arrow 0099\xlongrightarrowfill@{#1}{#2}}
\newcommand{\eqq}{\DOTSB\protect\Relbar\protect\joinrel\Relbar}
\def\xeqqfill@{\arrowfill@\Relbar\Relbar\eqq}
\newcommand{\aeqq}[2][]{
	\ext@arrow 0099\xeqqfill@{#1}{#2}}
\def\xRrightarrowfill@{\arrowfill@\equiv\equiv\Rrightarrow}
\newcommand{\aM}[2][]{\ext@arrow 0359\xRrightarrowfill@{#1}{#2}}

\def\nullrightarrowfill@{\arrowfill@{\relbar{\circ}}\relbar\longrightarrow}
\newcommand{\anull}[2][]{
	\ext@arrow0099\nullrightarrowfill@{#1}{#2}}
\makeatother
\newcommand{\aiso}[1]{\overset{#1}{\iso}}

\newcommand{\PB}[1]{\arrow[#1,phantom,"\scalebox{1.6}{\color{black}$\lrcorner$}",very near start]}
\newcommand{\scaleu}[2][1.2]{{\scalebox{#1}{$#2$}}}
\NewDocumentCommand{\fib}{O{n} O{2.3} mmm}{

	\begin{cd}*[#2][5]
		{#3}\ifx#1n{\arrow[d,"{\,\scaleu{#4}}"]}\else{\ifx#1i{\arrow[d,hookrightarrow,"{\,\scaleu{#4}}"]}\else{\ifx#1e{\arrow[d,equal,"{\,\scaleu{#4}}"]}\else{\ifx#1R{\arrow[d,Rightarrow,"{\,\scaleu{#4}}"]}\fi}\fi}\fi}\fi\\
		{#5}\ifx#1o{\arrow[u,"{\,\scaleu{#4}}"']}\fi
	\end{cd}\xspace
}
\newcommand{\Ar}[4][]{\arrow[#2,"{#3}"{#1},""{name=#4, anchor=center}]}
\newcommand{\Ars}[4][]{\arrow[#2,"{#3}"'{#1},""{name=#4, anchor=center}]}
\newcommand{\Arb}[6][]{\arrow[#2,"{#3}"{#1},from=#4,to=#5,shorten <= #6 em, shorten >= #6 em]}
\newcommand{\Arbs}[6][]{\arrow[#2,"{#3}"'{#1},from=#4,to=#5,shorten <= #6 em, shorten >= #6 em]}
		\NewDocumentCommand{\twosquare}{s O{n} O{6} O{6} O{} O{2.7} O{2.2} O{0.5} O{n}}{

			\def\foosq##1##2##3##4##5##6##7##8{
				\IfBooleanTF{#1}{\begin{cd}*}{\begin{cd}}[#3][#4]
						{##1}\ifx#2p{\PB{rd}}\fi\arrow[r,"{##5}"]\ifx#9l{\arrow[d,equal,"{##6}"']}\else{\arrow[d,"{##6}"']}\fi\&{##2}\ifx#9r{\arrow[d,equal,"{##7}"]}\else{\arrow[d,"{##7}"]}\fi\ifx#2l{\arrow[ld,Rightarrow,shorten <=#6ex,shorten >=#7ex,"{#5}"{pos=#8}]}\fi \ifx#2i{\arrow[ld,twoiso,shorten <=#6ex,shorten >=#7ex,"{#5}"{pos=#8}]}\fi\\
						{##3}\ifx#9d{\arrow[r,equal,"{##8}"']}\else{\arrow[r,"{##8}"']}\fi\ifx#2o{\arrow[ur,Rightarrow,shorten <=#6ex,shorten >=#7ex,"{#5}"{pos=#8}]}\fi\&{##4}
				\end{cd}}
				\foosq}
			\NewDocumentCommand{\tr}{s O{4.5} O{6.5} O{0} O{0} O{n} O{0} O{} O{0}}{

				\def\footr##1##2##3##4##5##6{
					\IfBooleanTF{#1}{\begin{cd}*}{\begin{cd}}[#3][#2]
							{##1}\arrow[rr,"{##4}"]
							\Ars[inner sep =0.2ex]{dr}{##5}{A}\&[#4ex]\&[#5ex]{##2}\Ar[inner sep =0.2ex]{ld}{##6}{B}\\
							\&{##3}
							\ifx#6l{\Arb{Rightarrow,shift right=#7em}{#8}{A}{B}{#9}}\else{\ifx#6o{\Arbs{Rightarrow,shift right=#7em}{#8}{B}{A}{#9}}\else{\ifx#6i{\Arbs[inner sep=0.9ex]{iso,shift right=#7em}{#8}{A}{B}{#9}}\else{\ifx#6e{\Arb{equal,shift right=#7em}{#8}{A}{B}{#9}}\else{}\fi}\fi}\fi}\fi
					\end{cd}}
					\footr}
\newcommand{\pExact}{\mathbf{pEx}}

\usepackage[T1]{fontenc}
\usepackage{lmodern,microtype,booktabs,tabularx}
\setlist[enumerate]{itemsep=3pt,topsep=5pt}
\newtheorem*{maintheorem}{Main results}
\newtheorem{conjecture}[theorem]{Conjecture}
\newcommand{\Pt}{\mathbf{Cat}_{0}}
\newcommand{\Add}{\mathbf{Add}}
\newcommand{\Tri}{\Triang}
\newcommand{\TriAll}{\mathbf{Triang}_{\mathrm{all}}}
\newcommand{\HomCat}{\mathbf{HomCat}}
\newcommand{\DiHom}{\mathbf{DiHom}}
\newcommand{\Lcat}{\mathfrak L}
\newcommand{\B}{\mathcal B}
\newcommand{\K}{\mathcal K}

\newcommand{\Pcal}{\mathcal P}

\newcommand{\Scl}{\mathcal S}
\newcommand{\X}{\mathcal X}
\newcommand{\Zcat}{\mathbf 0}
\newcommand{\twoker}{\operatorname{2ker}}
\newcommand{\twocoker}{\operatorname{2coker}}

\newcommand{\pq}[2]{#1/_{\!0}#2}
\newcommand{\aq}[2]{#1/[#2]}
\newcommand{\sq}[2]{#1/_{\!\mathrm S}#2}
\newcommand{\vq}[2]{#1/_{\!\mathrm V}#2}
\DeclareMathOperator{\Ob}{Ob}
\DeclareMathOperator{\Mor}{Mor}
\DeclareMathOperator{\Ker}{Ker}
\DeclareMathOperator{\Coker}{Coker}
\DeclareMathOperator{\im}{im}
\DeclareMathOperator{\ret}{ret}
\DeclareMathOperator{\add}{add}
\DeclareMathOperator{\Ser}{Ser}
\DeclareMathOperator{\thick}{thick}
\DeclareMathOperator{\cl}{cl}
\DeclareMathOperator{\Nat}{Nat}
\DeclareMathOperator{\Ann}{Ann}
\newcommand{\Nim}{\operatorname{Nim}}
\newcommand{\Ncm}{\operatorname{Ncm}}
\newcommand{\Sden}{\mathcal S}
\newcommand{\Eden}{\mathcal E}
\newcommand{\Loc}{\mathcal L}
\newcommand{\EPos}{\mathbf{EPos}}
\DeclareMathOperator{\Ex}{Ex}
\DeclareMathOperator{\SatTh}{STh}
\DeclareMathOperator{\TFrac}{TFrac}

\hypersetup{pdftitle={Exactness of the 2-categories of abelian and triangulated categories},pdfauthor={Elena Caviglia, Zurab Janelidze, Luca Mesiti, Ulo Reimaa}}
\title[Exactness of the 2-categories of abelian and triangulated categories]{Exactness of the 2-categories of abelian and triangulated categories}
\author[E. Caviglia]{Elena Caviglia}
\address[E. Caviglia]{Mathematics Division, Department of Mathematical Sciences, Stellenbosch University, Private Bag X1 Matieland, 7602, South Africa; National Institute for Theoretical and Computational Sciences (NITheCS), Stellenbosch, South Africa.}
\email{elena.caviglia@outlook.com}
\author[Z. Janelidze]{Zurab Janelidze}
\address[Z. Janelidze]{Mathematics Division, Department of Mathematical Sciences, Stellenbosch University, Private Bag X1 Matieland, 7602, South Africa; National Institute for Theoretical and Computational Sciences (NITheCS), Stellenbosch, South Africa.}
\email{zurab@sun.ac.za}
\author[L. Mesiti]{Luca Mesiti}
\address[L. Mesiti]{Mathematics Division, Department of Mathematical Sciences, Stellenbosch University, Private Bag X1 Matieland, 7602, South Africa.}
\email{luca.mesiti@outlook.com}
\author[\"U. Reimaa]{\"Ulo Reimaa}
\address[\"U. Reimaa]{Institute of Mathematics and Statistics, 
University of Tartu, 
Narva mnt 18, 
51009 Tartu, 
Estonia}
\email{ulo.reimaa@ut.ee}
\date{5 September 2026}

\thanks{The fourth author was supported in part by the Estonian Research Council grants  STP26 and PRG1204. The second author is grateful for the invitation and kind hospitality offered by the fourth author during his visit to University of Tartu. The second author used ChatGPT-6 Astra (OpenAI) for research assistance
and for writing and organizing the exposition. The fourth author used Claude Fable 5.1 to review the paper and some of its suggestions were incorporated. All authors
remain responsible for the content of the paper, including
the correctness of its results and the accuracy of its references.}

\begin{document}
\begin{abstract}
We introduce a notion of $2$-homological category modelled on pointed homological categories in the sense of Grandis, using bizero objects whose null $1$-cells are zero objects in the hom-categories for formulating $2$-dimensional pointedness. We first prove directly that the $2$-category of di-exact homological categories, which generalize Puppe exact categories, functors preserving all kernels and cokernels, and arbitrary natural transformations is $2$-homological. Its normal subcategories are saturated thick subcategories, and its exact quotients are constructed by a complete three-arrow fraction calculus similar to the one known for Puppe exact categories. We then adapt this proof to prove that the $2$-category of triangulated categories is also $2$-homological; here, instead of the ternary fractions we use the well-known Verdier fractions. Abstracting the common quotient structure of these proofs yields a general criterion, using which we further establish that the $2$-categories of pointed, additive and abelian categories are also $2$-homological. The criterion also applies to categories enriched in semimodules over a fixed commutative rig, with a zero object, and to their full sub-$2$-category with finite biproducts. Their quotients are linear congruence quotients. Combining enrichment with Puppe exactness gives further examples whose quotients are exact linear localizations, including linear abelian categories as the finite-biproduct case. In the abelian and triangulated cases the normal subcategories and quotients are, respectively, Serre subcategories and Serre quotients, and thick triangulated subcategories and Verdier localizations.
\end{abstract}
\maketitle

\section*{Introduction}

Abelian categories and triangulated categories are fundamental concepts in homological algebra and cohomology theory, used extensively in algebraic geometry, algebraic topology and representation theory. In this paper we establish $2$-homological properties of the $2$-categories of these categories. Both of these categories are additive, and in some sense, an additive category can be viewed as a categorical analogue of a commutative monoid. The category of commutative monoids is a homological category in the sense of Grandis \cite{Grandis13}, so one may naturally wonder whether the $2$-categories of abelian and triangulated categories will have the corresponding $2$-dimensional property. Grandis established the homological structure of the category of abelian categories and exact functors, with zero-valued functors as null morphisms \cite[Section~1.9, p.~146]{Grandis92}; see also \cite[Section~5.6]{Grandis13}. We use its pointed formulation obtained by identifying naturally isomorphic functors, namely the $1$-truncation of the ordinary $2$-category of abelian categories. We show that those $1$-cells that in the $1$-truncation end up being kernels and cokernels (they are given by thick subcategories and Serre quotients, well known in algebraic geometry) are actually kernels and cokernels in a very natural $2$-dimensional sense. It then follows that the $2$-category of abelian categories is a $2$-homological category in a suitable sense. It turns out with triangulated categories there is a similar story, where Serre quotients are replaced with Verdier quotients: the $2$-category of triangulated categories is $2$-homological as well.

A central ingredient of our notion of a $2$-homological category is a seemingly new
$2$-dimensional form of pointedness. Recall that the ordinary 
$1$-dimensional pointedness can be described by first equipping each
hom-set with a distinguished element, with composition preserving these
elements in each variable, and then requiring the existence of an initial
object. Such an object is automatically terminal, hence a zero object,
and the distinguished elements are precisely the morphisms factoring
through it. Conversely, the existence of a zero object determines this
enrichment in pointed sets. It is therefore natural, in dimension $2$,
to replace pointed sets by pointed categories: we require each
hom-category to have a zero object, pre- and postcomposition by
$1$-cells to preserve these zero objects, and a bi-initial object to
exist. Under these local pointedness assumptions, a bi-initial object
is also bi-terminal, hence a \emph{bizero object}, and the $1$-cells
factoring through it are zero objects in their respective hom-categories.
Unlike in dimension $1$, however, the existence of a bizero object
alone does not guarantee this compatible local pointedness; see Example~\ref{ex:bizero-not-strong}.
Accordingly, given a bizero object $0$, we call a $1$-cell \emph{null}
if it is isomorphic to one factoring through $0$, and call $0$
\emph{strong}, or a \emph{$2$-zero object}, when every null $1$-cell
is a zero object of its hom-category. Explicitly, for every
$1$-cell $f:A\to B$ and every parallel null $1$-cell $z:A\to B$,
there must be exactly one $2$-cell $f\Rightarrow z$ and exactly one
$2$-cell $z\Rightarrow f$.
A $2$-category admitting a strong bizero object is called
\emph{$2$-pointed}. This recovers the locally pointed description
above: the null $1$-cells are precisely the zero objects of the
hom-categories, and pre- and postcomposition preserve them.

In a $2$-pointed $2$-category, we define $2$-kernels and $2$-cokernels as bipullbacks and bipushouts over the $2$-zero object (and also provide alternative equivalent descriptions). A pointed category is homological in Grandis's sense when it has kernels and cokernels, the two classes are separately closed under composition, and a kernel $k:B\to C$ followed by a cokernel $c:C\to D$ factors as a cokernel followed by a kernel whenever $\Ker(c)$ factors through $k$. Replacing these notions by their $2$-dimensional counterparts gives the definition of a $2$-homological category. For a $2$-pointed $2$-category with all $2$-kernels and $2$-cokernels, the remaining homological axioms are equivalent to homologicity of the $1$-truncation, which identifies invertibly isomorphic $1$-cells and discards the $2$-cells.

The analogous question also arises for additive and pointed categories, and we prove their $2$-homologicity as well. The additive and abelian results admit further extensions over a fixed commutative rig $R$: we treat $R$-linear categories with a zero object and $R$-linear Puppe exact categories, in each case both with and without finite biproducts. For $R=\mathbb Z$, the finite-biproduct linear case recovers additive categories; for $R=\mathbb N$, the finite-biproduct Puppe exact linear case recovers abelian categories. We begin, however, with a direct proof for the $2$-category of pointed \emph{di-exact homological categories} (where the term ``di-exact'' is taken from \cite{PeschkeLinden}), namely pointed homological categories whose normal morphisms (i.e., morphisms that factor as a cokernel followed by kernel) are closed under composition. The terminology \emph{di-exact} is borrowed from Peschke and Van der Linden
\cite[Definition~2.1.8]{PeschkeLinden}; see also the subsequent
study by Afsa \cite{Afsa}. A pointed category
with kernels and cokernels is di-exact when every composite of a kernel 
followed by a cokernel is a normal morphism. Our expression
\emph{di-exact homological category} therefore means a pointed
homological category in Grandis's sense which is additionally
di-exact in their sense. Under existence of kernels and cokernels this
is equivalent to requiring that normal morphisms be closed under
composition. 

That the $2$-category of Puppe exact categories is also $2$-homological is a consequence of the result due to Grandis that the $1$-truncation of the $2$-category of Puppe exact categories is homological. But we can also derive it from our result that the broader $2$-category of di-exact homological categories is $2$-homological. In fact our proof of the latter result extends the technique for proving that $1$-truncation of the $2$-category of Puppe exact categories is homological: in particular, we extend the ternary fraction calculus from Puppe exact to di-exact homological categories to describe the $2$-cokernels.

The tables in Section~\ref{sec:framework} summarize what are objects, $1$-cells, and $2$-cells, as well as $2$-kernels and $2$-cokernels in the $2$-homological categories mentioned above. After introducing the basic $2$-categorical conceptual ingredients of the paper in Section~\ref{sec:framework}, in
Section~\ref{sec:di-homological} we prove that the $2$-category of di-exact homological categories is $2$-homological.
Section~\ref{sec:tri-direct} gives a proof of the similar result for triangulated
categories. Then, in Section~\ref{sec:criterion}, we formulate and prove
a general theorem that abstracts essential features of the two proofs in the previous two sections.
We then turn to pointed categories in Section~\ref{sec:pointed},
followed by additive categories in Section~\ref{sec:additive} and
abelian categories in Section~\ref{sec:abelian}, where we use the general theorem to prove that the $2$-categories of these categories are $2$-homological. Section~\ref{sec:rig-linear} gives the more general results for categories enriched in semimodules over a fixed commutative rig $R$, with a zero object, and for the full sub-$2$-category with finite biproducts. Its final subsection treats Puppe exact $R$-linear categories and exact $R$-linear functors, whose $2$-cokernels are exact linear localizations. The finite-biproduct case consists of $R^{\mathrm{gp}}$-linear abelian categories, where $R^{\mathrm{gp}}$ is the universal ring obtained by additive group completion. 
Section~\ref{sec:tri} discusses a shorter proof for the triangulated case, making use of the general theorem.
Section~\ref{sec:di-revisited} explains how the same general theorem compresses
the final argument of the proof for the di-exact homological case. Finally, Section~\ref{sec:comparison} compares exactness of the $2$-categories considered in this paper. It also proposes a conjecture that the $2$-category of (pointed) homological categories is $2$-homological and makes a small step towards proving it.

The theory developed in this paper generalizes to dimension $2$ the
\emph{pointed} homological categories of Grandis; extending it to the
non-pointed setting is a further direction of interest. In that setting,
null morphisms are specified by an ideal of morphisms rather than determined by the zero
object, and kernels and cokernels are defined relative to that ideal.
A two-dimensional framework for such an extension is provided in \cite{CJM}, where $2$-ideals of null
$1$-cells and null $2$-cells, together with their relative $2$-kernels
and $2$-cokernels, are introduced without assuming pointedness. It would be interesting
to develop the corresponding non-pointed notion of $2$-homological
category, replacing the strong bizero used here by suitable $2$-ideal
data and formulating an ideal-relative version of the conditional
homology axiom. We also expect there to be several intermediate useful concepts of a $2$-homological category between ours in our $2$-pointed context and a non-pointed one given by a $2$-categorical ideal: for example, one may relax the strongness assumption on bizero objects.

Our notion of $2$-dimensional exactness is different from the approaches found in the work of Vitale
and Dupont, who are motivated by the study of symmetric categorical groups as
$2$-dimensional analogues of abelian groups. This includes the
factorization systems investigated by Kasangian and Vitale
\cite{KasangianVitale} and leads to a groupoid-enriched setting; Dupont
develops this perspective systematically through abelian and
$2$-abelian categories enriched in pointed groupoids \cite{Dupont}.
Vitale \cite{Vitale2024} further develops this programme by
constructing, from any abelian category, a $2$-abelian bicategory
obtained by localizing its $2$-category of arrows at weak equivalences.
This groupoid-enriched homological programme should be distinguished
from the broader theory of proper factorization systems in general
$2$-categories developed jointly by Dupont and Vitale
\cite{DupontVitale}, which is not restricted to invertible $2$-cells
and provides part of the framework used in \cite{CJM}. In the
groupoid-enriched homological frameworks, all $2$-cells are
invertible, and the zero $1$-cell is a distinguished object of each
hom-groupoid, not required to be a zero object in the sense of being
both initial and terminal. By contrast, our hom-categories may
contain noninvertible $2$-cells, and our stronger pointedness
requires every null $1$-cell to be a zero object of its hom-category;
consequently, an invertible nullhomotopy, when it exists, is unique.
This distinction is essential: a groupoid possessing a zero object
is equivalent to the terminal category, so imposing our local
pointedness condition while requiring all $2$-cells to be invertible
would exclude the nontrivial hom-groupoids of the categorical-group
examples. Moreover, whereas Dupont's $2$-Puppe-exact and
$2$-abelian frameworks impose abelian-type normality and
factorization conditions, we require only the two-dimensional
analogue of Grandis's conditional exactness (i.e., the homology axiom), together with
existence and separate composition closure of $2$-kernels and
$2$-cokernels. The approaches are therefore complementary rather
than related by a straightforward inclusion of axiom systems:
our combination of stronger local pointedness and conditional
exactness is designed to accommodate the categories,
structure-preserving functors, and noninvertible natural
transformations studied in this paper.

Homological categories are useful in that in them one can develop the homological algebra machinery, as shown in \cite{Grandis13}.
We hope that our paper will lead to an extension of this machinery to $2$-categories and that such extension will be useful in the development of the subject of $2$-homological algebra.

\setcounter{tocdepth}{1}
\tableofcontents

\section{The framework and the precise statements}\label{sec:framework}\label{sectionkernelscokernels}

\subsection{Objects, 1-cells, 2-cells, and size}

Here \emph{pointed category} means a category with a zero object. In the
linear cases, $R$ denotes a fixed commutative rig; the enrichment is
specified in Section~\ref{sec:rig-linear}. We use the following
$2$-categories:
\begin{center}
\small
\begin{tabularx}{\textwidth}{@{}l>{\raggedright\arraybackslash}X>{\raggedright\arraybackslash}X>{\raggedright\arraybackslash}X@{}}
\toprule
Notation & Objects & 1-cells & 2-cells\\
\midrule
$\Pt$ & Pointed categories & Zero-object-preserving functors & All natural transformations\\[3pt]
$\Add$ & Additive categories & Additive functors & All natural transformations\\[3pt]
$\pExact$ & Puppe exact categories & Exact functors & All natural transformations\\[3pt]
$\AbCat$ & Abelian categories & Exact functors & All natural transformations\\[3pt]
$\mathbf{Lin}_{R,0}$ & $R$-linear categories with a zero object & $R$-linear functors & All natural transformations\\[3pt]
$\mathbf{Lin}_{R}^{\oplus}$ & $R$-linear categories with finite biproducts & $R$-linear functors & All natural transformations\\[3pt]
$\mathbf{PEx}_{R}$ & $R$-linear Puppe exact categories & Exact $R$-linear functors & All natural transformations\\[3pt]
$\mathbf{PEx}_{R}^{\oplus}$ & $R$-linear Puppe exact categories with finite biproducts & Exact $R$-linear functors & All natural transformations\\[3pt]
$\Tri$ & Triangulated categories & Exact functors with their shift isomorphisms & Shift-compatible natural transformations\\[3pt]
$\TriAll$ & The same & The same & All natural transformations\\[3pt]
$\DiHom$ & Di-exact homological categories & Functors preserving zero objects and all kernels and cokernels & All natural transformations\\
\bottomrule
\end{tabularx}
\end{center}
Zero objects and other preserved structure need not be preserved as chosen
objects on the nose. In the Puppe exact, abelian and di-exact settings, exact functors preserve
zero objects and all kernels and cokernels; in the triangulated settings,
they preserve distinguished triangles with the specified shift isomorphism. Linear functors preserve zero
objects and every existing finite biproduct. The superscript $\oplus$
denotes the full sub-$2$-category on objects with finite biproducts.
As explained in Remark~\ref{rem:puppe-linear-abelian},
$\mathbf{PEx}_{R}^{\oplus}$ identifies with the $2$-category of
$R^{\mathrm{gp}}$-linear abelian categories and exact linear functors,
where $R^{\mathrm{gp}}$ is the universal ring of $R$. We also write $\HomCat$ for pointed Grandis
homological categories, functors preserving zero objects and all kernels
and cokernels, and arbitrary natural transformations. Thus $\DiHom$
is its full sub-$2$-category consisting of the di-exact homological categories.

Fix a universe and take the categories in the table to be small in that
universe. The statements also apply to essentially small categories, with
the usual equivalence-invariant interpretation. This convention avoids
local-smallness problems for quotients. For larger locally small categories,
the same proofs apply whenever the indicated quotient categories exist
within the adopted size convention; no assertion about the local smallness
of arbitrary large Verdier quotients is implicit.

\subsection{Strong bizero objects and invertible nullhomotopies}

\begin{definition}\label{def:2pointed}
An object $0$ of a $2$-category $\Lcat$ is a \emph{bizero object} if both
$\Lcat(0,A)$ and $\Lcat(A,0)$ are equivalent to the terminal category
for every object $A$. Relative to $0$, a $1$-cell is \emph{null} if it
is isomorphic to a $1$-cell factoring through $0$. The bizero object $0$ is \emph{strong}, or a \emph{$2$-zero object}, if,
for every pair $A,B$, every null $1$-cell $z:A\to B$ is a zero object
of the hom-category $\Lcat(A,B)$. Equivalently, for every $1$-cell
$f:A\to B$ and every parallel null $1$-cell $z:A\to B$, there is
exactly one $2$-cell in each of the two directions
\[
 f\Rightarrow z,
 \qquad z\Rightarrow f.
\]
A $2$-category admitting a strong bizero object is called
\emph{$2$-pointed}.
\end{definition}

\begin{example}\label{ex:bizero-not-strong}
A bizero object need not be strong. Consider the strict
$2$-category $\mathcal E$ with objects $0,A$, with
$\mathcal E(0,0)$, $\mathcal E(0,A)$ and $\mathcal E(A,0)$
terminal, and with $\mathcal E(A,A)$ the two-element poset
$1_A<z$. If $i:0\to A$ and $p:A\to0$ are the unique
$1$-cells, composition is given by
\[
 pi=1_0,\qquad ip=z,\qquad z^2=z,
 \qquad zi=i,\qquad pz=p,
\]
together with the identity laws. These rules are associative
and monotone on the hom-posets, so determine horizontal
composition. The object $0$ is a bizero and $z=ip$ is null.
However, $z$ is terminal but not initial in
$\mathcal E(A,A)$, since there is no $2$-cell $z\Rightarrow1_A$.
Thus $0$ is not strong.
\end{example}

\begin{lemma}[Local zero objects and whiskering]\label{lem:localpointedness}
In a $2$-pointed $2$-category, every hom-category is pointed and its
zero objects are precisely the null $1$-cells. The unique $2$-cell
between parallel null $1$-cells is invertible. Pre- and postcomposition
preserve both the zero objects and the zero morphisms of the
hom-categories. In particular, the condition of
Definition~\ref{def:2pointed} is invariant under reversing $1$-cells
or reversing $2$-cells.
\end{lemma}
\begin{proof}
For every $A,B$, the bizero property supplies $1$-cells $A\to0$ and
$0\to B$. Their composite is null, and hence a zero object of
$\Lcat(A,B)$. All other zero objects of that hom-category are
isomorphic to this composite, so they are null. Conversely, every
null $1$-cell is a zero object by definition. The unique arrows
between two zero objects are mutually inverse, proving the
invertibility assertion.

Whiskering a null $1$-cell on either side again gives a null
$1$-cell: a factorization through $0$ is preserved by composition,
and a factorization up to isomorphism remains such after whiskering.
Thus both whiskering functors preserve zero objects. The zero
$2$-cell from $f$ to $g$ in a hom-category is the composite
\[
 f\longrightarrow z\longrightarrow g
\]
through any null $1$-cell $z$ in that hom-category. Whiskering sends
this factorization to one through a null $1$-cell, so preserves the
zero $2$-cell. Finally, reversing $1$-cells interchanges the
bi-initial and bi-terminal properties, while reversing $2$-cells
interchanges the two uniqueness requirements. Both operations
therefore preserve the definition.
\end{proof}

The definition could equivalently require the local zero-object
property only for $1$-cells factoring through $0$ on the nose, since
being a zero object is invariant under isomorphism. We use the
isomorphism-invariant formulation throughout. If $z$ is null and
$f$ is parallel to $z$, the unique $2$-cells $f\Rightarrow z$ and
$z\Rightarrow f$ are invertible if and only if $f$ is null.

For $f:A\to B$, a \emph{2-kernel} is a 1-cell $k:K\to A$ with an
invertible nullhomotopy $fk\iso0$, universally so in the sense that
\begin{equation}\label{eq:kernelUP}
 \Lcat(X,K)\simeq
 \{u\in\Lcat(X,A)\mid fu\text{ is null}\}.
\end{equation}
The right-hand side is a full subcategory. The equivalence is induced by
postcomposition with $k$ and is pseudonatural in $X$. Dually, a
\emph{2-cokernel} $q:B\to Q$ has $qf\iso0$ and satisfies
\begin{equation}\label{eq:cokernelUP}
 \Lcat(Q,X)\simeq
 \{v\in\Lcat(B,X)\mid vf\text{ is null}\}.
\end{equation}
Equivalently, one may retain the invertible nullhomotopies as data in
these displays. When the required composite is null, its nullhomotopy
is the unique isomorphism to a parallel null $1$-cell. The compatibility
condition on $2$-cells is automatic by the local zero-object property.
The full subcategories on the right are pointed: they contain all null
$1$-cells, and their zero objects are those of the ambient hom-category.
The displayed equivalences preserve these zero objects. For general
$2$-ideal formulations see \cite[Section 2]{CJM}.

\begin{lemma}[Kernels as bipullbacks]\label{lem:kernel-bipullback}
In a $2$-pointed $2$-category, a $2$-kernel of $f:A\to B$ is
precisely the projection to $A$ of a bipullback of $f$ and a
$1$-cell $z_B:0\to B$. Dually, $2$-cokernels are the
corresponding bipushouts over $0$.
\end{lemma}
\begin{proof}
For each $X$, a pseudocone consists of $u:X\to A$, $w:X\to0$,
and an invertible $2$-cell $\phi:fu\Rightarrow z_Bw$.
Since $\Lcat(X,0)\simeq\mathbf1$, choices of $w$ are uniquely
isomorphic. The composite $z_Bw$ is a zero object of
$\Lcat(X,B)$, so the $2$-cell $fu\Rightarrow z_Bw$ is unique
and is invertible exactly when $fu$ is null. Between two such
cones, the component $w\Rightarrow w'$ is forced, and every
$2$-cell $u\Rightarrow u'$ satisfies the cone compatibility:
both sides of that equation have zero codomain $z_Bw'$.
Thus forgetting $w$ and $\phi$ gives an equivalence from the
pseudocone category to the full subcategory on the right of
\eqref{eq:kernelUP}. These equivalences commute with
precomposition, so the representing universal properties
coincide. The cokernel assertion follows by reversing $1$-cells.
\end{proof}

\begin{remark}\label{rem:notlax}
Arbitrary $2$-cells between $1$-cells do not mean arbitrary
nullhomotopies. In every $2$-pointed $2$-category, not only in the
concrete examples, there is a unique $2$-cell $fu\Rightarrow0$ and
a unique $2$-cell $0\Rightarrow fu$, whether or not $fu$ is null.
Allowing either merely as a noninvertible nullhomotopy would therefore
impose no annihilation condition. The same applies to $vf$ in the
$2$-cokernel universal property. Invertibility of the nullhomotopy
is essential throughout this paper.
\end{remark}

$2$-kernels and $2$-cokernels are invariant under invertible 2-cells and
equivalence of the representing object. A 1-cell is \emph{2-normal} if,
up to invertible 2-cell, it is a 2-cokernel followed by a 2-kernel.

\begin{definition}\label{def:homological}
A $2$-category $\Lcat$ is \emph{$2$-homological} if it is
$2$-pointed in the sense of Definition~\ref{def:2pointed} and satisfies
the following axioms:
\begin{enumerate}[label=\textup{(H\arabic*)},leftmargin=16mm]
\item every 1-cell has a 2-kernel and a 2-cokernel;
\item 2-kernels are closed under composition, and so are 2-cokernels;
\item if $m:M\to A$ is a 2-kernel and $e:A\to Q$ is a 2-cokernel,
and $\twoker(e)$ factors through $m$ up to invertible 2-cell, then $em$
is 2-normal.
\end{enumerate}
\end{definition}

\begin{remark}[Independence of the bizero object]\label{rem:bizero-independence}
Any two bizero objects $0$ and $0'$ are equivalent: choose
$1$-cells in both directions, and use the unique isomorphisms in
$\Lcat(0,0)\simeq\mathbf1$ and
$\Lcat(0',0')\simeq\mathbf1$ for their composites with the
identities. Transporting a factorization along this equivalence
shows that they determine the same null $1$-cells. Hence if one
is strong, so is the other, and the notions of $2$-kernel,
$2$-cokernel, and $2$-homologicity do not depend on this choice.
\end{remark}

Thus a $2$-homological $2$-category has a bizero object and pointed
hom-categories whose zero objects are precisely the null $1$-cells.
Condition (H3) is the conditional subquotient formulation of Grandis's
homology axiom: boundaries lie in cycles. In dimension one the hypothesis
can also be written $\Coker(m)\Ker(e)=0$; see
\cite[Section 1.3.6]{Grandis13} and \cite[Definition 2.16]{Fritz}.
Grandis's foundational development is \cite{Grandis92}.

\subsection{Universal properties, normal factorizations, and truncation}

\begin{remark}\label{remcanonical2ideal}
The null $1$-cells form a two-sided ideal. On each pair of objects, take
as null $2$-cells all $2$-cells between null $1$-cells. Strong pointedness
makes that full subcategory equivalent to the terminal category. Whiskering
preserves it, since a composite with a $1$-cell factoring through the bizero
again factors through the bizero. This is the canonical pointed instance
of the $2$-ideal formalism of \cite[Section 2]{CJM}. The term
\emph{null $2$-cell} in this $2$-ideal is distinct from the ordinary
\emph{zero morphism of a hom-category}: the latter exists between
any parallel $1$-cells by Lemma~\ref{lem:localpointedness}, whereas
the former has both boundary $1$-cells null.
\end{remark}

\begin{proposition}[Expanded universal properties]\label{prop:expanded-up}
A $1$-cell $k:K\to A$ is a $2$-kernel of $f:A\to B$ precisely when:
\begin{enumerate}[label=\textup{(K\arabic*)},leftmargin=16mm]
\item $fk$ is null;
\item every $u:X\to A$ with $fu$ null admits a factorization
$u\cong kv$;
\item for any $v,w:X\to K$, every $2$-cell $kv\Rightarrow kw$ is
$k\beta$ for a unique $\beta:v\Rightarrow w$.
\end{enumerate}
Dually, $q:B\to Q$ is a $2$-cokernel of $f$ precisely when $qf$ is null,
every $u:B\to X$ annihilating $f$ is isomorphic to $vq$, and every
$2$-cell $vq\Rightarrow wq$ is $\beta q$ for a unique $\beta:v\Rightarrow w$.
In particular $2$-kernels are representably fully faithful and
$2$-cokernels are representably co-fully faithful.
\end{proposition}
\begin{proof}
Conditions (K2) and (K3) say exactly that the functor in
\eqref{eq:kernelUP} is essentially surjective and fully faithful;
(K1) makes that functor well-defined. The nullhomotopies impose no
additional condition on the $2$-cells, since both their boundaries are
null. The dual argument gives the cokernel assertion. Whiskering
commutes with these functors, so the equivalences have the stated
pseudonaturality.
\end{proof}

\begin{proposition}\label{propkeriskerofitscoker}
In a $2$-pointed $2$-category, if $k$ is a $2$-kernel and
a $2$-cokernel $q$ of $k$ exists, then $k$ is a $2$-kernel
of $q$. Dually, if $q$ is a $2$-cokernel and a $2$-kernel
$k$ of $q$ exists, then $q$ is a $2$-cokernel of $k$.
A $2$-cokernel of a null $1$-cell is an equivalence, and dually a
$2$-kernel of a null $1$-cell is an equivalence.
\end{proposition}
\begin{proof}
Let $k$ be a $2$-kernel of $f$, and let $q$ be its $2$-cokernel.
Since $fk$ is null, write $f\cong hq$. If $qu$ is null then $fu$
is null, so $u$ factors through $k$; conversely $qk$ is null.
The full faithfulness needed for the universal property is already
that of $k$. Thus the same $k$ represents the $2$-kernel of $q$.
The dual proof gives the first assertion for $2$-cokernels.
For a null $f:A\to B$, the identity of $B$ satisfies its $2$-cokernel
universal property, because every $v:B\to X$ annihilates $f$.
Uniqueness of representing objects up to equivalence gives the assertion.
\end{proof}

\begin{proposition}[The $1$-truncation]\label{prop:truncation}
Let $\Lcat$ be $2$-pointed and have all $2$-kernels and $2$-cokernels.
Its $1$-truncation $\tau_1\Lcat$, whose arrows are invertible-isomorphism
classes of $1$-cells, is pointed and has all kernels and cokernels.
A $1$-cell is a $2$-kernel, respectively a $2$-cokernel, exactly when
its class is a kernel, respectively a cokernel. Consequently $\Lcat$
is $2$-homological exactly when $\tau_1\Lcat$ is homological in
Grandis's pointed sense.
\end{proposition}
\begin{proof}
The bizero becomes a zero object, since each hom-category to or from it
has one isomorphism class. A composite represents zero precisely when
it is null. Passing to isomorphism classes in \eqref{eq:kernelUP}
therefore gives the ordinary kernel universal property: full faithfulness
implies $kv\cong kw$ if and only if $v\cong w$. The dual argument applies
to cokernels.

Conversely, suppose $[u]$ is a kernel of $[f]$. Choose a $2$-kernel
$k$ of $f$. The two ordinary kernels are uniquely isomorphic over their
codomain in $\tau_1\Lcat$. Representing that isomorphism and its inverse
by $1$-cells gives mutually inverse equivalences up to invertible $2$-cells,
and the over-codomain equality gives $u\cong ka$. Thus $u$ is a
$2$-kernel. The dual assertion follows in the same way. Composition,
the factorization of $\Ker(e)$ through $m$, and cokernel--kernel
factorizations now correspond on the two sides, proving the last claim.
The assumption that the $2$-kernels and $2$-cokernels already exist is
essential to this reflection argument. The strengthened $2$-pointedness
is also an explicit premise: passing to $\tau_1\Lcat$ does not by
itself recover the zero-object properties of the hom-categories.
\end{proof}

\begin{definition}\label{def:2-diexact}
A $2$-pointed $2$-category with all $2$-kernels and
$2$-cokernels is \emph{$2$-di-exact} if every $2$-kernel
followed by a $2$-cokernel is $2$-normal. It is
\emph{$2$-Puppe exact} if every $1$-cell is $2$-normal.
\end{definition}

These definitions agree under $\tau_1$ with their one-dimensional
counterparts, by Proposition~\ref{prop:truncation}. In a
$2$-homological $2$-category, $2$-di-exactness is equivalent to
closure of $2$-normal $1$-cells under composition, by
Proposition~\ref{prop:diexact-equivalences} applied to the
$1$-truncation. More generally, that closure is equivalent to
$2$-di-exactness together with separate closure of $2$-kernels
and $2$-cokernels under composition. Thus the unconditional mixed
clause in Definition~\ref{def:2-diexact} does not replace the
composition clauses when $2$-homologicity is not assumed.

\begin{proposition}[Normal factorization]\label{prop:normal-factor}
Every $1$-cell $f:A\to B$ in a $2$-pointed $2$-category with all
$2$-kernels and $2$-cokernels has a factorization, up to invertible
$2$-cell,
\[
 A\xrightarrow{q_f}\operatorname{Coim}(f)
 \xrightarrow{\overline f}\operatorname{Im}(f)
 \xrightarrow{m_f}B,
 \qquad q_f=\twocoker(\twoker f),\quad
 m_f=\twoker(\twocoker f).
\]
The comparison $\overline f$ is determined up to invertible $2$-cell.
The $1$-cell $f$ is $2$-normal if and only if $\overline f$ is an equivalence.
\end{proposition}
\begin{proof}
Write $f\cong hq_f$ by the cokernel property, and let $r=\twocoker f$.
Since $rhq_f$ is null and precomposition with $q_f$ is fully faithful,
$rh$ is null: lift the isomorphism from $rhq_f$ to a zero composite
through that fully faithful functor. Hence $h$ factors through $m_f$,
giving $\overline f$. The two full-faithfulness assertions give uniqueness
up to invertible $2$-cell.

If $\overline f$ is an equivalence the displayed factorization is
$2$-normal. Conversely let $f\cong me$ with $e$ a $2$-cokernel and $m$
a $2$-kernel. A representably fully faithful $m$ reflects nullness:
compare $mv$ with $m0$ and lift their isomorphism. Thus $\twoker f$
is $\twoker e$, and dually $\twocoker f$ is $\twocoker m$.
Proposition~\ref{propkeriskerofitscoker} identifies $e$ with $q_f$ and
$m$ with $m_f$, up to equivalence. The comparison between the resulting
middle objects is consequently an equivalence.
\end{proof}

\begin{maintheorem}\label{thm:main}
Each of $\Pt$, $\Add$, $\pExact$, $\AbCat$, $\Tri$, $\TriAll$,
and $\DiHom$ is $2$-pointed and $2$-homological. For every commutative
rig $R$, the same holds for $\mathbf{Lin}_{R,0}$,
$\mathbf{Lin}_{R}^{\oplus}$, $\mathbf{PEx}_{R}$, and
$\mathbf{PEx}_{R}^{\oplus}$. Their $2$-kernels are, up to equivalence,
full annihilation subcategories. The normal subcategories and quotient
constructions are as follows:
\begin{center}
\small
\begin{tabularx}{\textwidth}{@{}>{\raggedright\arraybackslash}p{30mm}>{\raggedright\arraybackslash}X>{\raggedright\arraybackslash}X@{}}
\toprule
Setting & Normal subcategories & Normal quotients\\
\midrule
Pointed & Full replete pointed subcategories closed under retracts & Collapse morphisms factoring through the subcategory to zero\\[3pt]
Additive & Full replete additive subcategories closed under retracts & Additive ideal quotients\\[3pt]
Puppe exact & Thick subcategories (subquotients and extensions) & Exact localizations\\[3pt]
Abelian & Serre subcategories & Serre quotients\\[3pt]
$\mathbf{Lin}_{R,0}$, $\mathbf{Lin}_{R}^{\oplus}$ & Full replete subcategories satisfying Definition~\ref{def:rig-normal} & Linear ideal quotients by the congruence in~\eqref{eq:rig-congruence}\\[3pt]
$\mathbf{PEx}_{R}$ & Thick subcategories & Exact $R$-linear localizations\\[3pt]
$\mathbf{PEx}_{R}^{\oplus}$ & Serre subcategories & $R^{\mathrm{gp}}$-linear Serre quotients\\[3pt]
Triangulated & Thick triangulated subcategories & Verdier quotients\\[3pt]
Di-exact homological & Saturated thick subcategories & Exact localizations at the associated normal denominators\\
\bottomrule
\end{tabularx}
\end{center}
The linear assertions are proved in Theorems~\ref{thm:rig-linear}
and~\ref{thm:puppe-linear}. They recover the additive case at
$R=\mathbb Z$ and the abelian case at $R=\mathbb N$, respectively,
when finite biproducts are imposed. In the tables, subcategories are
full and replete, and quotient constructions are understood up to
equivalence.
For nested normal subcategories $\N\subseteq\M\subseteq\C$, the
homology axiom is realised by
\[
 \M\longrightarrow\M/\N\longrightarrow\C/\N,
\]
where the first arrow is a 2-cokernel and the second is the 2-kernel of
$\C/\N\to\C/\M$. Each quotient is taken in its indicated setting.
\end{maintheorem}

\section{Di-exact homological categories and their ternary fractions}\label{sectionabelian}\label{sec:di-homological}

We first prove the $2$-homologicity of $\DiHom$ directly.
After the one-dimensional normal calculus we introduce saturated
thick subcategories and stable normal denominators, develop epidd
reduction posets and the full three-arrow calculus, and prove the
saturated quotient theorem. The last step verifies the two composition
laws and Grandis's conditional homology axiom on exact-functor
categories, without a general quotient criterion. Puppe exact
categories then appear as the specialization in which saturation is
automatic. Throughout, exact functors preserve zero objects and the
kernels and cokernels of every morphism.

\subsection{One-dimensional normal factorizations}

Puppe exactness originates in \cite{Pup62,Mit65,BP69}; Grandis's
nonadditive development includes modular-lattice and projective-space
models \cite{Gra84,CaG96,Gra12}. 

We work with pointed categories. A \emph{normal monomorphism} is a
kernel, and a \emph{normal epimorphism} is a cokernel. A morphism is
\emph{normal} if it is a normal epimorphism followed by a normal
monomorphism. An \emph{exact functor} preserves zero objects,
all kernels, and all cokernels. Outside the Puppe-exact setting this
is stronger than preservation of short exact sequences alone;
see \cite[Section 1.3]{Grandis92}. A short exact sequence means a
kernel--cokernel pair, with zero objects adjoined at its ends.

\begin{definition}[Pointed homological categories]\label{def:grandis-homological}
A category $\C$ is \emph{homological in the pointed sense of Grandis}
if it satisfies the following conditions:
\begin{enumerate}[label=\textup{(\roman*)}]
\item $\C$ is pointed and every morphism has a kernel and a cokernel;
\item normal monomorphisms are closed under composition, and normal
      epimorphisms are closed under composition;
\item if $k:B\to C$ is a normal monomorphism and $c:C\to D$
      is a normal epimorphism such that $\ker c$ factors through
      $k$, then $ck$ is normal.
\end{enumerate}
Condition \textup{(iii)} is the \emph{homology axiom}.
\end{definition}

This is Grandis's pointed case, with zero morphisms as the null
ideal; see \cite[Sections~1.3 and~1.6]{Grandis92}.
The pullback and pushout stability laws are consequences of
\textup{(ii)} and \textup{(iii)}, proved in
Lemma~\ref{di:lem:basic}; they are not additional axioms here.
In our definitions and results, a homological category means one in
Definition~\ref{def:grandis-homological}.

A pointed category with kernels and cokernels has the canonical normal
factorization
\[
 X\xrightarrow{e_f}\Ncm(f)\xrightarrow{\overline f}\Nim(f)
 \xrightarrow{m_f}Y,
 \qquad e_f=\operatorname{coker}(\ker f),\quad
 m_f=\ker(\operatorname{coker}f).
\]
As in Proposition~\ref{prop:normal-factor}, $f$ is normal precisely when
$\overline f$ is invertible. We call $f$ a \emph{near-isomorphism} if
its kernel and cokernel have zero domain and zero codomain, respectively.
A normal near-isomorphism is an isomorphism: in a cokernel--kernel factorization
$f=me$, the kernel of $e$ and the cokernel of $m$ are zero, forcing
$e$ and $m$ to be isomorphisms.

\begin{definition}
A \emph{di-exact homological category} is a homological category
in the sense of Definition~\ref{def:grandis-homological} in which
normal morphisms are closed under composition.
A \emph{Puppe exact category} is a pointed category with kernels and
cokernels in which every morphism is normal. We use $\pExact$ for their
$2$-category, with exact functors and arbitrary natural transformations.
Here ``exact category'' always has this Puppe-Mitchell meaning when
used without a qualifier.
\end{definition}

This terminology keeps both the di-exactness and the homological
hypotheses explicit; compare \cite{Afsa,PeschkeLinden,Grandis13}.
The classical Puppe-exact theory is developed in
\cite{Pup62,Mit65,BP69}. Abelian categories are precisely the
Puppe exact categories with finite biproducts
\cite{Mit65,BP69}.

\begin{proposition}\label{prop:diexact-equivalences}
For a pointed category with kernels and cokernels, the following are
equivalent:
\begin{enumerate}[label=\textup{(\roman*)}]
\item normal morphisms are closed under composition;
\item kernels and cokernels are separately closed under composition,
and every kernel followed by a cokernel is normal;
\item the category is di-exact homological.
\end{enumerate}
In particular every Puppe exact category is di-exact homological.
\end{proposition}
\begin{proof}
Every kernel and every cokernel is normal, using an identity for the
other factor. Assume (i). A composite of kernels is normal and monic.
If a normal monomorphism is written $me$, its cokernel factor $e$ is
monic. A monic cokernel has zero kernel and is an isomorphism: it is
a cokernel of its own zero kernel. Hence the composite is a kernel.
The dual argument gives closure of cokernels. Applying (i) to a kernel
and then a cokernel gives the rest of (ii).

Assume (ii), and write $f_i=m_ie_i$ with the indicated normal factors.
Factor $e_2m_1=m'e'$ as a cokernel followed by a kernel. Then
$f_2f_1=(m_2m')(e'e_1)$ is normal, proving (i). Condition (ii) implies the homology axiom of
Definition~\ref{def:grandis-homological}\textup{(iii)} simply by
omitting its containment hypothesis; together with (i) and the
composition clauses, it therefore implies (iii).
Conversely (iii) includes (i).
The last assertion follows because all morphisms are normal.
\end{proof}

\begin{lemma}[Cancellation and inheritance]\label{lem:inheritance}
In a pointed category with kernels and cokernels, if $k=k_1k_2$ is a
kernel and $k_1$ is monic, then $k_2$ is a kernel. Dually, if $q=q_2q_1$
is a cokernel and $q_1$ is epic, then $q_2$ is a cokernel.
A nonempty full replete subcategory closed under ambient kernels and
cokernels of its morphisms inherits homologicity, di-exact
homologicity, or Puppe exactness from the ambient category.
\end{lemma}
\begin{proof}
If $k=\ker f$, then $k_2=\ker(fk_1)$: an arrow $u$ annihilated by
$fk_1$ gives a unique $v$ with $kv=k_1u$, and monicity of $k_1$ gives
$k_2v=u$. This also identifies the square with the pullback of $k$
along $k_1$. Dualize for cokernels.

For inheritance, the kernel of the identity of any object supplies a
zero object in the subcategory. Fullness and the closure assumption
make its kernels and cokernels the ambient ones. We verify the
remaining clauses of Definition~\ref{def:grandis-homological}
inside the subcategory. If two internal kernels are composed,
ambient homologicity makes their composite a kernel;
it is then the kernel of its own cokernel, which is an internal
morphism. Thus the composite is an internal kernel. Dualize. For the
homology axiom, form the normal coimage and normal image of the relevant
composite using successive kernels and cokernels. All these objects
are internal. Its ambient normality makes the comparison invertible,
and fullness and repleteness make this an internal isomorphism.
The same argument for every normal composite, or for every morphism,
proves the di-exact and Puppe-exact assertions.
\end{proof}

\begin{proposition}[The exact core]\label{prop:exact-core}
If $\C$ is di-exact homological, its objects and normal morphisms form
a Puppe exact category $\C_{\mathrm{ex}}$. Its kernels and cokernels
are the ambient ones, and its monomorphisms and epimorphisms are,
respectively, the ambient kernels and cokernels.
\end{proposition}
\begin{proof}
Identities are normal and composition is allowed by hypothesis. Let
$k:K\to X$ be the ambient kernel of a normal morphism $f$, and let
a normal $h:Z\to X$ satisfy $fh=0$. Write $h=m_he_h$ as a cokernel followed by a kernel.
Since $e_h$ is epic, $fm_h=0$, so $m_h=kv$ uniquely.
By Lemma~\ref{lem:inheritance}, $v$ is a kernel. Hence $ve_h$ is normal
and gives the required factorization in $\C_{\mathrm{ex}}$.
This proves the kernel universal property there, and the dual argument
proves the cokernel property. Every arrow of the core has its given
normal factorization inside the core, so the core is Puppe exact.
A monic arrow of the core has a cokernel--kernel factorization inside it whose
cokernel factor is an isomorphism, by the argument of
Proposition~\ref{prop:diexact-equivalences}; it is therefore an ambient
kernel. The converse and the dual statement follow immediately.
\end{proof}

This proposition transfers statements about \emph{objects} and normal
subquotients from the Puppe-exact theory to di-exact homological
categories. It does not by itself transfer localization of the whole
category, because the core omits nonnormal arrows. Saturation and
arbitrary base change in the fraction calculus below address precisely
those additional arrows.

\subsection{Discrete exact categories and full annihilation kernels}

Let $\mathbf J$ have objects $0,J$, their identity arrows, and the
zero arrows, with $1_J\ne0$. Every arrow is either an identity or
zero, and direct computation gives kernels, cokernels, and cokernel--kernel
factorizations. Thus $\mathbf J$ is Puppe exact. More generally,
a pointed category in which every nonzero arrow is an identity is
Puppe exact. Cartesian products of Puppe exact categories are Puppe
exact, with all structure computed coordinatewise. Pointed wedges
are formed by keeping the summands and their morphisms, identifying
chosen zero objects, and admitting only zero arrows between distinct
summands. Their kernels and cokernels are computed in the relevant
summand; a cross-summand arrow is zero and has the identity as kernel
and cokernel. Hence these wedges too are Puppe exact. With zero objects
preserved up to isomorphism, the resulting universal property is a
bicoproduct property, rather than an assertion about an unstrictified
ordinary coproduct.

For every object $A$ of a pointed category, the functor
$L_A:\mathbf J\to\C$ selecting $A$ is exact whenever $\C$ has kernels
and cokernels: its only kernel and cokernel diagrams are those of an
identity or a zero arrow. Moreover,
\[
 \Nat(L_A,L_B)\cong\C(A,B).
\]
The component at $0$ is forced and the component at $J$ is arbitrary.
Consequently an exact functor between our homological categories is
representably fully faithful precisely when its underlying functor is
fully faithful. One implication follows by testing on $\mathbf J$;
the other follows by lifting components uniquely and using faithfulness
to verify naturality.

\begin{proposition}\label{prop:puppe-kernels}
For an exact functor $F:\C\to\D$ between Puppe exact categories,
the full subcategory
\[
 \Ker(F)=\{X\in\C\mid F(X)\cong0\}
\]
is Puppe exact, its inclusion is exact, and this inclusion is a
$2$-kernel of $F$ in $\pExact$.
\end{proposition}
\begin{proof}
The subcategory contains zero objects. Exactness of $F$ implies that
the ambient kernel and cokernel of a morphism between annihilated
objects are annihilated. Lemma~\ref{lem:inheritance} applies.
Every exact functor whose composite with $F$ is null corestricts to
this full subcategory and is exact there, because its kernels and
cokernels are ambient. Natural transformations corestrict uniquely
by fullness. These are all parts of the hom-category universal
property of a $2$-kernel.
\end{proof}

The same annihilation-kernel argument works for di-exact homological
categories, and indeed for all pointed homological categories:
Lemma~\ref{lem:inheritance} supplies the induced structure, and
fullness supplies the lifted transformations. In each case the
one-object zero category is a strong bizero. Functors to and from
it preserve zero objects, kernels and cokernels. For an exact
$F:\C\to\D$ and a zero-valued exact $Z:\C\to\D$, the unique
maps $F(X)\to Z(X)$ and $Z(X)\to F(X)$ form natural
transformations in both directions. The naturality equations hold
because their two sides have a common zero codomain or zero domain,
respectively. Their components are forced, so the transformations
are unique. All of them are admitted $2$-cells.

\subsection{Normal subquotients, extensions, and thick closure}

A \emph{thick subcategory} of a Puppe exact category is a full replete
subcategory containing zero objects and closed under subobjects,
quotients, and extensions. In a di-exact homological category we use
\emph{normal} subobjects and normal quotients in this definition.
For abelian categories this is exactly a Serre subcategory. Thick
subcategories inherit the relevant structure by
Lemma~\ref{lem:inheritance}. The annihilation subcategory of an exact
functor is thick: a short exact sequence remains short exact, and
if its middle, or its two outside objects, vanish, the required
remaining objects vanish as well. Subobjects and quotients are the
special cases of this argument.

For a class $\mathcal B$ of objects, first adjoin a zero object and
close under isomorphisms. Write $\mathcal B^{\mathsf S}$ for its
normal subquotients, and $\mathcal B^{\mathsf E}$ for the objects
in short exact sequences whose two outside objects belong to
$\mathcal B$.

\begin{lemma}\label{LemD}
In a di-exact homological category,
\[
 \mathcal B^{\mathsf{SS}}=\mathcal B^{\mathsf S},\qquad
 \mathcal B^{\mathsf{ES}}\subseteq\mathcal B^{\mathsf{SE}},\qquad
 \mathcal B^{\mathsf{SE}^{\infty}}
   =\bigcup_{n\ge0}\mathcal B^{\mathsf{SE}^{n}}
\]
is the smallest thick subcategory containing $\mathcal B$.
Here $\mathsf E^0$ denotes no extension step, after the indicated
$\mathsf S$ step. In an abelian category the resulting subcategory
is additive.
\end{lemma}
\begin{proof}
By Proposition~\ref{prop:exact-core}, normal subobjects and
normal quotients in $\C$ are exactly the subobjects and quotients
in $\C_{\mathrm{ex}}$. The short exact sequences also agree:
their maps are normal, and the core's kernels and cokernels are
the ambient ones. Thus both operations $\mathsf S$ and
$\mathsf E$ are unchanged, and it suffices to reason in the
Puppe exact category $\C_{\mathrm{ex}}$. We use its elementary exact calculus: pullbacks of
epimorphisms along monomorphisms are epimorphisms, dually for pushouts,
and normal image factorizations are compatible with these squares
\cite{BP69,Mit65}. A subobject of a quotient of $B$ is a quotient of
its inverse image in $B$. This proves that an iterated subquotient
is a subquotient and gives the first equality.

For $0\to A\to X\to C\to0$, a subobject $Y\hookrightarrow X$
fits in the short exact sequence
\[
 0\longrightarrow A\cap Y\longrightarrow Y
 \longrightarrow\operatorname{im}(Y\to C)\longrightarrow0.
\]
Both outside terms are subobjects of the original outside terms.
Dually, a quotient of $X$ is an extension of quotients of $A$ and
$C$. Combining these two statements proves
$\mathcal B^{\mathsf{ES}}\subseteq\mathcal B^{\mathsf{SE}}$.
Since zero belongs to $\mathcal B$, the classes
$\mathcal B^{\mathsf{SE}^n}$ form an increasing chain.
Moving each $\mathsf S$ past the preceding $\mathsf E$ by the
inclusion just proved shows that their union is closed under
subquotients. Any two of its objects occur at a common finite stage,
so an extension of them occurs at the next stage. This proves
thickness. Every thick subcategory containing $\mathcal B$ contains
every stage, proving minimality. In the abelian case a finite
biproduct is a split extension of its two summands, so the closure
is additive.
\end{proof}

\subsection{Normal pullbacks, pushouts, and comparison morphisms}\label{di:sec:grandis}

We first recall the universal constructions and the normal-factorization facts used in the fraction calculus. Existence of a pullback or pushout is distinct from stability of a chosen denominator class under that operation.

\begin{lemma}\label{di:lem:basic}
In a pointed category with all kernels and cokernels, a normal monomorphism $m:M\to Y$ has a pullback along every $f:X\to Y$, and a normal epimorphism has a pushout along every morphism with the same domain. Explicitly,
\[
 f^*m=\ker((\coker m)f),\qquad
 f_*e=\coker(f(\ker e)).
\]
We also use normal cancellation as in Lemma~\ref{lem:inheritance}: if $ba$ is normal monic and $b$ is monic, then $a$ is normal monic, and dually for normal epimorphisms.
If the category is homological in the sense of
Definition~\ref{def:grandis-homological}, pulling back a normal
epimorphism along a normal monomorphism gives a normal epimorphism, and
dually pushing out a normal monomorphism along a normal
epimorphism gives a normal monomorphism.
\end{lemma}
\begin{proof}
Put $c=\coker m$, so $m=\ker c$ up to the canonical isomorphism. If $k=\ker(cf)$, the equation $cfk=0$ produces $v$ with $mv=fk$. Given $r:T\to X$ and $s:T\to M$ with $fr=ms$, one has $cfr=0$, so $r=ku$ uniquely. Monicity of $m$ then gives $vu=s$. This is precisely the pullback universal property. The pushout assertion is dual. Normal cancellation was proved in Lemma~\ref{lem:inheritance}.

For the homological stability assertion, let $e:X\to Y$ be normal
epic and $m:M\to Y$ normal monic. Put $c=\coker m$ and
let $l:P\to X$ be $\ker(ce)$. The pullback projection
$v:P\to M$ satisfies $el=mv$. By Definition~\ref{def:grandis-homological}\textup{(ii)},
normal epimorphisms compose, so $ce$ is normal epic and hence is
a cokernel of $l$. Since $\ker e$ factors through $l$, clause
\textup{(iii)} of that definition makes $el$ normal. Moreover, $c$ is a cokernel of $el$: if $ael=0$,
the cokernel property of $ce$ gives $ae=bce$ uniquely, and
epicity of $e$ gives $a=bc$; uniqueness also follows from
epicity of $c$. Thus the normal image of $el$ is $\ker c=m$.
Its normal factorization through $m$ has a normal-epic first
factor, necessarily $v$ by monicity of $m$. The corresponding
pushout law follows by duality.
\end{proof}

The existence and cancellation assertions also appear, in greater
generality with an ideal of null morphisms, in
\cite[Lemmas~2.6 and~2.9]{Fritz}. The homological pullback and
pushout laws are the equivalent forms \textup{(ex3a)} and
\textup{(ex3a*)} of the homology axiom in
\cite[Section~1.6, p.~144]{Grandis92}.

\begin{lemma}\label{di:lem:comparison}
In a homological category as in
Definition~\ref{def:grandis-homological}, the canonical comparison
\[
 \bar f:\Ncm f\longrightarrow\Nim f,
 \qquad f=m_f\bar f e_f,
\]
has zero kernel and zero cokernel. Here $e_f=\coker(\ker f)$ and $m_f=\ker(\coker f)$.
\end{lemma}
\begin{proof}
Let $k:K\to\Ncm f$ be the kernel of $\bar f$. Pull back the normal epimorphism $e_f$ along $k$, obtaining $l:P\to X$ and a normal epimorphism $v:P\to K$ with $e_fl=kv$. Then $fl=0$, so $l$ factors through $\ker f$ and $e_fl=0$. Monicity of $k$ gives $v=0$. Since $v$ is epic, $1_K=0$ and $K$ is zero. Apply the same argument in the opposite category for the cokernel.
\end{proof}

\begin{lemma}\label{di:lem:nested-normal}
For normal monomorphisms $A\aar{m}B\aar{n}C$ in a pointed homological category there is a short exact sequence
\[
 0\longrightarrow\Coker m\longrightarrow\Coker(nm)
   \longrightarrow\Coker n\longrightarrow0.
\]
Dually, for composable normal epimorphisms there is a short exact sequence of their kernel objects.
\end{lemma}
\begin{proof}
Set $c=\coker(nm)$. Since $\ker c=nm$ factors through $n$, the homology axiom makes $cn$ normal. Its kernel is $m$: this follows from monicity of $n$ and $\ker c=nm$. Hence its cokernel--kernel factorization is $cn=j\,\coker m$ with $j$ normal monic. The map $\coker n$ factors as $rc$ because it annihilates $nm$. The map $r$ is the cokernel of $cn$. Indeed, if $a cn=0$, then $ac=b\,\coker n=brc$ uniquely, and epicity of $c$ gives $a=br$. Since $\coker m$ is epic, $r$ is also the cokernel of $j$. This gives the asserted short exact sequence. Dualize for kernels.
\end{proof}

\subsection{Saturated thick subcategories}\label{di:sec:saturation}

\begin{definition}\label{di:def:saturated}
A full replete subcategory $\N$ of a pointed homological category $\C$ is \emph{thick} if it contains the zero objects and is closed under normal subobjects, normal quotients, and extensions. A thick subcategory is called \emph{saturated} if, for every morphism $f:X\to Y$ of $\C$, it satisfies
\begin{align}
 Y,\Ker f\in\N&\quad\Longrightarrow\quad X\in\N,\label{di:eq:satk}\\
 X,\Coker f\in\N&\quad\Longrightarrow\quad Y\in\N.\label{di:eq:satc}
\end{align}
Write $\SatTh_{\C}(\mathcal A)$ for the smallest saturated thick subcategory containing a given class of objects $\mathcal A$ of $\C$.
\end{definition}

In a Puppe exact category the two implications (\ref{di:eq:satk}-\ref{di:eq:satc}) follow from thickness and normality of $f$, so saturation introduces no extra restriction there.

\begin{lemma}\label{di:lem:saturated-properties}
Saturated thick subcategories have the following properties.
\begin{enumerate}[label=\textup{(\roman*)}]
\item They are closed under existing retracts and under ambient kernels and cokernels of their morphisms. They inherit homologicity, and inherit di-exact homologicity when the ambient category is di-exact. Their inclusions are exact.
\item Every full annihilation subcategory of an exact functor is saturated thick.
\item Arbitrary intersections and inverse images under exact functors are saturated thick. Saturation and thickness are transitive along full inclusions. If $\N\subseteq\M$ are saturated thick in $\C$, then $\N$ is saturated thick in $\M$.
\item If $\Nim f\in\N$, then $\Ncm f\in\N$. Dually, $\Ncm f\in\N$ implies $\Nim f\in\N$.
\end{enumerate}
\end{lemma}
\begin{proof}
For (i), a retraction $r:A\to B$ has zero cokernel. If $A\in\N$, implication~\eqref{di:eq:satc} gives $B\in\N$. Ambient kernels and cokernels of internal arrows are normal subobjects and normal quotients of internal objects. They therefore lie in $\N$ and are the internal kernels and cokernels by fullness. Normal compositions remain normal internally because a normal monomorphism is the kernel of its cokernel (and dually for normal epimorphisms), and those objects are internal. The normal coimage and normal image of any internal arrow are also internal. Consequently its ambient normality is its internal normality. This proves inheritance of the conditional homology axiom and of di-exactness.

For (ii), zero objects, repleteness, normal subquotients, and extensions are preserved or detected as required by applying the exact functor. For instance, a short exact sequence with zero outer terms has zero middle term. If $F(Y)=F(\Ker f)=0$, then $F(f)$ has zero codomain and its kernel has zero domain. The kernel of a map to zero is an identity, so $F(X)=0$. This proves~\eqref{di:eq:satk}; the dual proves~\eqref{di:eq:satc}.

Intersections in (iii) are immediate, and inverse images follow because an exact functor preserves all the diagrams and objects used in the definition. For transitivity, suppose $\N$ is saturated thick in $\M$, which is saturated thick in $\C$. Closure of $\N$ under normal subquotients and extensions in $\C$ follows first by putting the relevant objects into $\M$, then applying the corresponding internal closure of $\N$. If $Y,\Ker f\in\N$, saturation of $\M$ first puts $X$ into $\M$. The arrow $f$ is then internal, with the same kernel, and saturation of $\N$ puts $X$ into $\N$. Dualize. The restriction assertion follows because kernels, cokernels, and short exact sequences in $\M$ are computed in $\C$. Since $\C$ itself is saturated thick, the intersection defining $\SatTh$ exists.

Finally, apply~\eqref{di:eq:satk} or~\eqref{di:eq:satc} to the comparison $\bar f$ of Lemma~\ref{di:lem:comparison}, whose kernel and cokernel are zero.
\end{proof}

\subsection{Stable normal denominators}\label{di:sec:denominators}

Fix a di-exact homological category $\C$ and a saturated thick $\N\subseteq\C$. Put
\begin{align*}
 \Sden&=\{m\mid m\text{ is normal monic and }\Coker m\in\N\},\\
 \Eden&=\{e\mid e\text{ is normal epic and }\Ker e\in\N\}.
\end{align*}
These will be the two types of denominator.

\begin{lemma}\label{di:lem:denominators}$\;$
\begin{enumerate}[label=\textup{(\roman*)}]
\item Both classes contain all isomorphisms and are closed under composition.
\item $\Sden$ is stable under pullback along every morphism of $\C$, and $\Eden$ is stable under pushout along every morphism.
\item A pushout of an arrow of $\Sden$ along a normal epimorphism is again in $\Sden$. Dually a pullback of an arrow of $\Eden$ along a normal monomorphism is again in $\Eden$.
\item If $s\in\Sden$ and $e\in\Eden$ are composable, then
\[
 es=s'e',\qquad s'\in\Sden,\quad e'\in\Eden.
\]
\item If $e,ue\in\Eden$, then $u\in\Eden$. Dually, if $s,sv\in\Sden$, then $v\in\Sden$.
\end{enumerate}
\end{lemma}
\begin{proof}
For (i), normal monomorphisms and normal epimorphisms are separately closed under composition, and the short exact sequences in Lemma~\ref{di:lem:nested-normal} put the new cokernel or kernel in $\N$ by extension closure.

For (ii), pull back $s$ along $f$ and put $c=\coker s$. The resulting normal monomorphism is $k=\ker(cf)$. Since $\Nim(cf)$ is a normal subobject of $\Coker s\in\N$, Lemma~\ref{di:lem:saturated-properties}(iv) gives
\[
 \Coker k=\Ncm(cf)\in\N.
\]
Thus $k\in\Sden$. The opposite-category argument proves the assertion for $\Eden$. This is the precise use of saturation that allows arbitrary, rather than only normal, base-change maps.

For (iii), the pushout exists by Lemma~\ref{di:lem:basic}, since it is also the pushout of the normal epimorphism along the arrow of $\Sden$. The homological law in that lemma makes the pushed-out arrow normal monic; its cokernel is canonically the original cokernel, by the pushout universal property. The dual argument gives existence of the pullback and preservation of the kernel object.

For (iv), di-exactness makes $es$ normal. In its normal factorization, the kernel is a normal subobject of $\Ker e$, obtained by pulling back $\ker e$ along $s$. The cokernel is a normal quotient of $\Coker s$, obtained by pushing out $\coker s$ along $e$. Hence both belong to $\N$, giving the stated classes for the two normal factors.

For (v), $u$ is normal epic by Lemma~\ref{di:lem:basic}. Pulling back $e$ along $\ker u$ gives a normal epimorphism $\Ker(ue)\to\Ker u$. Its source belongs to $\N$, so its target does too. Therefore $u\in\Eden$. Dualize for $\Sden$.
\end{proof}

\subsection{Epidd posets and common reductions}\label{subsec:epidd}

The order-theoretic mechanism behind a common-reduction calculus can be
separated from the diagrams producing its reductions. Define an \emph{epidd poset} to be a poset in which every
principal ideal is downward directed. Explicitly, if $x,y\leqslant z$,
there is $w\leqslant x,y$. Put
\[
 x\sim y\quad\Longleftrightarrow\quad
 \text{$x$ and $y$ have a common lower bound}.
\]
The direction of the order is chosen so that a further reduction is
smaller than the diagram being reduced.

\begin{lemma}\label{lem:epidd}
On an epidd poset, $\sim$ is an equivalence relation. For a function
$f:P\to Q$ between epidd posets, the following conditions are equivalent:
\begin{enumerate}[label=\textup{(\roman*)}]
\item $f(x)\sim f(y)$ whenever $x\leqslant y$;
\item $f(x)\sim f(y)$ whenever $x\sim y$.
\end{enumerate}
If $R$ is a reflexive relation on $P$ with $R\circ R$ being equal to ``$\leqslant$'', these
conditions are also equivalent to $f(x)\sim f(y)$ whenever $xRy$.
\end{lemma}
\begin{proof}
Reflexivity and symmetry of $\sim$ are immediate. If $a\leqslant x,y$
and $b\leqslant y,z$, the epidd property below $y$ gives $c\leqslant a,b$.
Then $c\leqslant x,z$, proving transitivity.
Suppose (i) and let $u\leqslant x,y$. Then
$f(x)\sim f(u)\sim f(y)$, so (ii) follows. Conversely, $x\leqslant y$
implies $x\sim y$, proving (i) from (ii).
If $xRy$, reflexivity implies $x(R\circ R)y$, hence $x\leqslant y$.
Conversely, if $x\leqslant y$, choose $z$ with $xRzRy$. The condition
on $R$ gives $f(x)\sim f(z)\sim f(y)$. Transitivity in $Q$ gives
(i), completing the equivalence.
\end{proof}

\begin{definition}
An \emph{epidd map} is a function satisfying the conditions in Lemma~\ref{lem:epidd}.
Write $\EPos$ for the category of epidd posets and epidd maps.
An \emph{epidd category} is a category enriched in $\EPos$ with its
cartesian monoidal structure.
\end{definition}

\begin{proposition}\label{prop:epidd-products}
The category $\EPos$ has finite products, given by cartesian products
of posets. A function $P\times Q\to R$ is epidd precisely when it is
epidd in each variable separately. The assignment
$\pi_0P=P/{\sim}$ is left adjoint to the inclusion of sets as discrete
posets and preserves finite products.
\end{proposition}
\begin{proof}
Identities and composites preserve $\sim$, so $\EPos$ is a category.
In a product, common lower bounds are chosen coordinatewise, as are
lower bounds of two elements in a principal ideal. Thus the product
is epidd and its projections have the claimed universal property.
The one-element poset is terminal. If $b$ is separately epidd and
$x\sim x'$, $y\sim y'$, then
$b(x,y)\sim b(x',y)\sim b(x',y')$, so it is jointly epidd.
The converse follows by fixing a coordinate. A map from $P$ to a
discrete poset is epidd exactly when it is constant on the $\sim$
classes, which proves the adjunction. Finally, two pairs have a common
lower bound exactly when their two coordinate pairs do; hence
$\pi_0(P\times Q)=\pi_0P\times\pi_0Q$ canonically, and the terminal
object is also preserved.
\end{proof}

Thus an epidd category can equivalently be described as an ordinary
category whose hom-sets are epidd posets and whose composition is
an epidd map in each variable.

\begin{proposition}[The common-reduction quotient]\label{prop:epidd-congruence}
In an epidd category, the relation of having a common lower bound is
a congruence. The quotient has the same objects and hom-sets
$\pi_0P(X,Y)$.
More generally, start with epidd posets $P(X,Y)$, distinguished elements
$1_X\in P(X,X)$ and binary composition operations which are epidd in
each variable. It suffices for the unit and associativity equations
to hold \emph{modulo $\sim$}; then the same quotient is a category.
\end{proposition}
\begin{proof}
If $f\sim f'$ and $g\sim g'$, separate compatibility gives
$gf\sim gf'\sim g'f'$. Thus composition is well-defined on classes.
The equivalence-relation assertion is Lemma~\ref{lem:epidd}, and
the unit and associativity equations on classes follow from the
corresponding equations modulo $\sim$. For an epidd category those
equations already hold before taking classes.
\end{proof}

For ternary fractions the elementary relation $R$ will mean either
one source restriction or one target quotient. Every reduction is
one of each, so $R\circ R$ is the entire reduction order. The lemma
therefore makes precise why these two elementary operations suffice
for compatibility checks. Section~\ref{subsec:di-reductions} constructs
the epidd posets of ternary diagrams, and
Theorem~\ref{di:thm:composition} proves their composition law and
common-reduction quotient. The more general formulation of
Proposition~\ref{prop:epidd-congruence} is needed there: with arbitrary
nonnormal middle arrows, associativity before reduction need not hold,
as Example~\ref{di:ex:raw-associativity} shows. This does not affect the
associative category of localized fractions.

\subsection{Ternary diagrams and their epidd reduction posets}\label{subsec:di-reductions}

Fix the classes $\Sden,\Eden$ of Lemma~\ref{di:lem:denominators}.
The construction in this and the following two subsections uses only
properties (i)--(v) of that lemma, not the description of the classes
in terms of $\N$. Thus it gives a three-arrow calculus for any pair
of normal denominator classes with those properties.

\begin{definition}\label{di:def:ternary}
A \emph{ternary fraction diagram} from $X$ to $Y$ is a diagram
\begin{equation}\label{di:eq:ternary}
 X\xleftarrow{m}U\xrightarrow{f}V\xleftarrow{e}Y,
 \qquad m\in\Sden,\quad e\in\Eden.
\end{equation}
Its middle arrow $f$ is arbitrary. Diagrams are identified only when
their two intermediate objects are related by isomorphisms commuting
with all three arrows. Write $e\backslash f/m$ for the resulting
isomorphism class and $\TFrac(X,Y)$ for their set.
A \emph{reduction} of $\rho=e\backslash f/m$ is
\[
 \rho'=(be)\backslash(bfa)/(ma),
 \qquad a\in\Sden,\quad b\in\Eden.
\]
We write $\rho'\preceq\rho$. A source restriction uses $b=1$,
and a target quotient uses $a=1$; these are the elementary reductions.
Equivalently, between two admissible diagrams one may require only
$m'=ma$, $e'=be$ and $f'=bfa$: membership of the comparisons
$a\in\Sden$ and $b\in\Eden$ then follows from
Lemma~\ref{di:lem:denominators}(v).
\end{definition}

Thus the comparison diagram for a reduction has the form
\begin{equation}\label{di:eq:large-reduction}
\vcenter{\xymatrix@C=24pt@R=20pt{
 & U\ar@{>->}[dl]_-{m}\ar[rr]^-{f}
 && V\ar@{->>}[dd]^-{b} & \\
 X &&&& Y\ar@{->>}[ul]^-{e}\ar@{->>}[dl]_-{be} \\
 & U'\ar@{>->}[uu]^-{a}\ar@{>->}[ul]^-{ma}\ar[rr]_-{bfa}
 && V' &
}}
\end{equation}
whose three defining equations are
$m'=ma$, $e'=be$ and $f'=bfa$. No inverse of $a$ or $b$ is used
in a reduction.

\begin{lemma}\label{di:lem:reduction-posets}
Reduction is well-defined on isomorphism classes and is a partial
order on $\TFrac(X,Y)$. It is an epidd order. If $R$ denotes the
reflexive relation of an elementary reduction, then
$R\circ R$ is equal to ``$\preceq$''. Consequently, the relation of having a common reduction is an equivalence relation, denoted $\sim$.
\end{lemma}
\begin{proof}
Changing representatives conjugates $a$ and $b$ by the intermediate
isomorphisms, and the denominator classes contain isomorphisms and
are closed under composition. This proves independence of
representatives. Reflexivity uses identities; transitivity composes
the two source comparisons and the two target comparisons.
If $\rho'\preceq\rho\preceq\rho'$, monicity of $m,m'$ shows that
the two source comparisons are inverse isomorphisms. Epicity of
$e,e'$ gives the same conclusion for the target comparisons.
The numerator equations then identify the diagrams, proving
antisymmetry.

Let $\rho_i=(b_ie)\backslash(b_ifa_i)/(ma_i)\preceq\rho$ for
$i=1,2$. Pull $a_1$ and $a_2$ back over $U$, obtaining
$a=a_1r_1=a_2r_2$ with $r_i\in\Sden$. Push $b_1$ and $b_2$ out
under $V$, obtaining $b=t_1b_1=t_2b_2$ with $t_i\in\Eden$.
The two comparison constructions are the following pullback and
pushout, respectively:
\begin{equation}\label{di:eq:common-reduction-squares}
\begin{tikzcd}[column sep=large,row sep=large]
U_{12}\arrow[r,tail,"r_2"]\arrow[d,tail,"r_1"'] &
 U_2\arrow[d,tail,"a_2"] &
 V\arrow[r,two heads,"b_2"]\arrow[d,two heads,"b_1"'] &
 V_2\arrow[d,two heads,"t_2"]\\
 U_1\arrow[r,tail,"a_1"'] & U &
 V_1\arrow[r,two heads,"t_1"'] & V_{12}.
\end{tikzcd}
\end{equation}
All the asserted memberships follow from arbitrary base-change
stability and composition of denominators. Then
\[
 (be)\backslash(bfa)/(ma)
\]
is a reduction of both $\rho_i$, since
$t_i(b_ifa_i)r_i=bfa$. This proves the epidd condition.
Every reduction is a source restriction followed by a target
quotient, and the identity is elementary. Conversely, two elementary
reductions compose to a reduction. Thus $R\circ R=\preceq$.
Lemma~\ref{lem:epidd} proves the last assertion.
\end{proof}

\begin{lemma}[Common denominators and equality]\label{di:lem:common-denominators}
Any two ternary diagrams have reductions with the same source and
target denominators, but not necessarily the same numerator. If two
diagrams already have the same denominators $m,e$, with numerators
$f,g$, then they have a common reduction precisely when
\begin{equation}\label{di:eq:equality}
 bfa=bga
\end{equation}
for some $a\in\Sden$, $b\in\Eden$ of the appropriate types.
Moreover, the set of common-reduction classes is the filtered colimit
\begin{equation}\label{di:eq:epidd-colimit}
 \pi_0\TFrac(X,Y)\cong
 \varinjlim_{\substack{m:U\tom X\,\in\Sden\\
                       e:Y\toe V\,\in\Eden}}\C(U,V).
\end{equation}
\end{lemma}
\begin{proof}
Pull the two source denominators back over $X$ and push the two
target denominators out under $Y$. This gives the claimed common
denominators. For equal denominators, in any common reduction the
two source comparisons coincide by monicity of $m$, and the two
target comparisons coincide by epicity of $e$. The equality of the
reduced numerators is therefore exactly \eqref{di:eq:equality}.
The converse is the definition of reduction.

Choose representatives of normal subobjects and normal quotients.
Order the pairs of denominators by further restriction and further
quotient. The pullbacks and pushouts just used make this a filtered
index category. Comparisons are unique, and are in $\Sden$ and
$\Eden$ by Lemma~\ref{di:lem:denominators}(v). Its transition maps
send $f$ to $bfa$. Equality in the colimit is equality after a
common transition, which is precisely the common-reduction relation
already proved to be an equivalence. This proves the display.
\end{proof}

\subsection{The pullback--factorization--pushout composition}\label{subsec:di-composition}

Take composable diagrams
\[
 \rho=e\backslash f/m:X\longrightarrow Y,
 \qquad \sigma=e'\backslash g/m':Y\longrightarrow Z,
\]
where $m:U\to X$, $f:U\to V$, $e:Y\to V$, and
$m':U'\to Y$, $g:U'\to V'$, $e':Z\to V'$.
The middle composite $em'$ is a normal monomorphism followed by
a normal epimorphism; di-exactness and
Lemma~\ref{di:lem:denominators}(iv) give its normal factorization
\begin{equation}\label{di:eq:middle-factor}
 em'=np,\qquad p:U'\to I\text{ in }\Eden,
 \quad n:I\to V\text{ in }\Sden.
\end{equation}
The whole composition construction is displayed in the following
pasted diagram. The left parallelogram is a pullback, the central
diamond is the normal factorization \eqref{di:eq:middle-factor},
and the right parallelogram is a pushout. The dashed outer arrows
are the new denominators, and the dashed top arrow is the new
numerator.
\begin{equation}\label{di:eq:large-composition}
\vcenter{\xymatrix@C=18pt@R=26pt{
 && P\ar@{>->}[dl]_-{a}\ar[rr]^-{h}
      \ar@/^18pt/@{-->}[rrrr]^-{jh}
      \ar@/_16pt/@{>-->}[ddll]_-{ma}
 && I\ar@{>->}[dl]^-{n}\ar[rr]^-{j}
 && W && \\
 & U\ar@{>->}[dl]_-{m}\ar[rr]^-{f}
 && V && U'\ar@{>->}[ld]_-{m'}\ar@{->>}[ul]^-{p}\ar[rr]_-{g}
 && V'\ar@{->>}[ul]^-{b} & \\
 X &&&& Y\ar@{->>}[ul]^-{e}
 &&&& Z\ar@{->>}[ul]_-{e'}\ar@/_16pt/@{-->>}[uull]_-{be'}
}}
\end{equation}
In particular, the pullback gives $fa=nh$, and the pushout gives
$bg=jp$. Arbitrary base-change stability in
Lemma~\ref{di:lem:denominators}(ii) gives $a\in\Sden$ and
$b\in\Eden$. Thus $ma\in\Sden$ and $be'\in\Eden$.
This is the full pullback--factorization--pushout diagram of the
ternary calculus, with normal arrows retained explicitly. Its
central factorization is the place where di-exactness is used;
neither $f$ nor $g$ is assumed normal. Define the diagrammatic
composite by
\begin{equation}\label{di:eq:diagram-composition}
 \sigma\star\rho=(be')\backslash(jh)/(ma).
\end{equation}
All choices are immaterial up to isomorphism of the two intermediate
objects: normal factorizations are unique up to a unique commuting
isomorphism, as are the displayed pullbacks and pushouts.
Changing either input by an isomorphism changes the result by such
an isomorphism. Hence $\star$ is well-defined on ternary diagram
classes. The diagrams $1_X\backslash1_X/1_X$ are its two-sided
units, since the corresponding normal factorizations and squares
with identity maps can be taken to be identity factorizations and
squares.

For every functor $T$ inverting $\Sden\cup\Eden$, assign the value
\[
 T\langle e\backslash f/m\rangle=(Te)^{-1}(Tf)(Tm)^{-1}.
\]
Reductions preserve this value. The three equations
$em'=np$, $fa=nh$ and $bg=jp$ imply
\begin{equation}\label{di:eq:composition-value}
 T\langle\sigma\star\rho\rangle
   =T\langle\sigma\rangle T\langle\rho\rangle.
\end{equation}
Indeed, $(Tm')^{-1}(Te)^{-1}=(Tp)^{-1}(Tn)^{-1}$;
then $(Tn)^{-1}Tf=Th(Ta)^{-1}$ and
$Tg(Tp)^{-1}=(Tb)^{-1}Tj$, which give the equation by substitution.
The common-reduction criterion and associativity of the induced
composition are proved below.

\begin{lemma}[The three elementary fraction diagrams]\label{di:lem:three-factor-diagram}
For $\rho=e\backslash f/m:X\to Y$, put
\[
 \alpha=1_U\backslash1_U/m:X\to U,\qquad
 \beta=1_V\backslash f/1_U:U\to V,\qquad
 \gamma=e\backslash1_V/1_V:V\to Y.
\]
Both $(\gamma\star\beta)\star\alpha$ and
$\gamma\star(\beta\star\alpha)$ are isomorphic as ternary diagrams
to $\rho$. This particular decomposition holds before taking
common-reduction classes; it does not assert general associativity
at that stage.
\end{lemma}
\begin{proof}
The two ways of evaluating the three elementary diagrams are
assembled in the following pasted diagram. The three bottom zigzags
are $\alpha$, $\beta$, and $\gamma$. The successive identity
factorizations and pullback--pushout squares give the two
parenthesizations above them.
\begin{equation}\label{di:eq:large-decomposition}
\vcenter{\xymatrix@!=7pt{ & & & U\ar@{>->}[dl]_-{1_U}\ar[rr]^-{1_U} & & U\ar@{>->}[dl]_-{1_U}\ar[rrr]^-{f}\ar@{<<-}[drr]_-{1_U} & & & V\ar@{>->}[dll]^-{1_V}\ar[rr]^-{1_V}\ar@{<<-}[dr]^-{1_V} & & V\ar@{<<-}[dr]^-{1_V} \\ & & U\ar@{>->}[dl]_-{1_U}\ar[rr]^-{1_U} & & U\ar@{>->}[dl]_-{1_U}\ar[rr]_-{f}\ar@{<<-}[dr]_-{1_U} & & V\ar@{<<-}[drr]^-{1_V} & U\ar@{>->}[dll]_-{1_U}\ar[rr]_-{f} & & V\ar@{>->}[dl]^-{1_V}\ar[rr]^-{1_V}\ar@{<<-}[dr]^-{1_V} & & V\ar@{<<-}[dr]^-{1_V} \\ & U\ar@{>->}[dl]_-{m}\ar[rr]_-{1_U} & & U\ar@{<<-}[dr]_-{1_U}  & & U\ar@{>->}[dl]^-{1_U}\ar[rrr]_-{f} & & & V\ar@{<<-}[dr]_-{1_V} & & V\ar@{>->}[dl]^-{1_V}\ar[rr]_-{1_V} & & V\ar@{<<-}[dr]_-{e}\\ X & & & & U & & &  & & V & & & & Y }}
\end{equation}
For an explicit check, the central normal factorization in
$\beta\star\alpha$ is that of $1_U$, its source pullback is
an identity square, and its target pushout is the pushout of
$1_U$ along $f$. Thus $\beta\star\alpha=1_V\backslash f/m$.
Composing with $\gamma$ uses the factorization of $1_V$ and
identity squares and gives $e\backslash f/m$. Likewise
$\gamma\star\beta=e\backslash f/1_U$, and subsequent composition
with $\alpha$ gives $e\backslash f/m$. Uniqueness of the indicated
universal constructions yields the claimed isomorphisms of diagrams.
\end{proof}
\subsection{Identification with the ordinary localization}\label{di:sec:fractions}

Let
\[
 Q:\C\longrightarrow\Loc=\C[(\Sden\cup\Eden)^{-1}]
\]
be the ordinary localization. For completeness it can be constructed from the graph of $\C$ by adjoining inverse arrows for the denominators, taking finite paths, and imposing the composition, identity, and inverse relations. Since $\C$ is small, this gives a small category and its usual universal property. In particular associativity is already ensured by the construction. This is the elementary localization of \cite[Chapter~I]{GZ}; no fraction theorem for di-exact categories is being assumed.

\begin{theorem}\label{di:thm:fractions}
For every $X,Y\in\C$ there is a natural description
\begin{equation}\label{di:eq:colimit}
 \Loc(QX,QY)\cong
 \varinjlim_{\substack{m:U\tom X\ \in\Sden\\ e:Y\toe V\ \in\Eden}}
 \C(U,V).
\end{equation}
The index is ordered by further restriction of $m$ and further quotient of $e$. A representative $[m,f,e]$ corresponds to
\[
 (Qe)^{-1}(Qf)(Qm)^{-1}.
\]
Two representatives are equal precisely when their numerators become equal after a common such refinement. In particular, equality of representatives over the same $m,e$ is witnessed by $a\in\Sden$ and $b\in\Eden$ with
\[
 bfa=bga.
\]
\end{theorem}
\begin{proof}
By Lemmas~\ref{di:lem:reduction-posets} and
\ref{di:lem:common-denominators}, the right side is
$H_X(Y)=\pi_0\TFrac(X,Y)$. In this proof write $[m,f,e]$ for the
common-reduction class of $e\backslash f/m$. Its transitions are
$[m,f,e]\mapsto[ma,bfa,be]$. In particular equality is witnessed by
one common reduction. All zero
numerators give the same class after choosing common denominators;
this makes $H_X(Y)$ a pointed set.

\smallskip\noindent\emph{Functoriality in $Y$.}
For $g:Y\to Z$ and a representative $[m,f,e:Y\to V]$, take the pushout
\[
\begin{tikzcd}
Y\arrow[r,"e"]\arrow[d,"g"']&V\arrow[d,"g'"]\\
Z\arrow[r,"e'"']&W.
\end{tikzcd}
\]
The bottom map is in $\Eden$ by Lemma~\ref{di:lem:denominators}(ii). Set
\[
 H_X(g)[m,f,e]=[m,g'f,e'].
\]
A source refinement commutes with this formula. For a target refinement $b:V\to V'$ in $\Eden$, push out $b$ along $g'$ to obtain $d:W\to W'$ in $\Eden$. Pasting the two pushouts gives a pushout of $be$ along $g$, and both formulas agree after the target refinement $d$. Choices of pushout are uniquely isomorphic and thus give the same colimit class. Pushout pasting for two successive maps identifies their action with the action of their composite; the pushout along an identity supplies the identity action. Thus $H_X$ is a pointed-set-valued functor.

\smallskip\noindent\emph{Inverting $\Eden$.}
For $t:Y\to Z$ in $\Eden$, the inverse of $H_X(t)$ is
\[
 [m,f,e:Z\to V]\longmapsto[m,f,et:Y\to V].
\]
Indeed, the pushout of $et$ along $t$ has maps $1_V$ and $e$, because $t$ is epic. Conversely, in the pushout of $e$ along $t$, the other comparison $V\to W$ lies in $\Eden$, since it is a pushout of $t$. The composite with the displayed inverse is then a permitted target refinement. Both composites are identities.

\smallskip\noindent\emph{Surjectivity for $\Sden$.}
Let $s:Y\to Z$ lie in $\Sden$ and take $[m,f,e]\in H_X(Z)$, with $e:Z\to V$. By Lemma~\ref{di:lem:denominators}(iv), factor
\[
 es=np,\qquad p:Y\to I\text{ in }\Eden,
 \quad n:I\to V\text{ in }\Sden.
\]
Pull $n$ back along $f:U\to V$, obtaining $a:U'\to U$ in $\Sden$ and $h:U'\to I$ with $fa=nh$. Consider $\theta=[ma,h,p]\in H_X(Y)$. Push $p$ out along $s$, obtaining $d:Z\to W$ in $\Eden$ and $j:I\to W$. There is a comparison $u:W\to V$ with
\[
 ud=e,\qquad uj=n.
\]
As $d,e\in\Eden$, Lemma~\ref{di:lem:denominators}(v) gives $u\in\Eden$. Consequently
\[
 H_X(s)\theta=[ma,jh,d]=[ma,nh,e]=[ma,fa,e]=[m,f,e].
\]

\smallskip\noindent\emph{Injectivity for $\Sden$.}
Put two classes of $H_X(Y)$ over common $m:U\to X$ and $e:Y\to V$, with numerators $f,g$. Push $e$ out along $s$. Write $j:V\to W$ for the resulting arrow in $\Sden$ and $d:Z\to W$ for the arrow in $\Eden$; membership follows from Lemma~\ref{di:lem:denominators}(ii),(iii). If the image classes agree, the filtered equality relation gives $a\in\Sden$ and $b\in\Eden$ such that
\[
 bjfa=bjga.
\]
Factor $bj=np$ with $n\in\Sden$ and $p\in\Eden$. Monicity of $n$ gives $pfa=pga$, a permitted refinement proving equality of the original classes. Therefore $H_X(s)$ is bijective.

\smallskip\noindent\emph{Identification with the localization.}
We have proved that $H_X$ factors through the ordinary localization to $\overline H_X:\Loc\to\mathbf{Set}_*$. There is a well-defined map
\[
 \Phi_Y:H_X(Y)\longrightarrow\Loc(QX,QY),\qquad
 [m,f,e]\longmapsto(Qe)^{-1}(Qf)(Qm)^{-1}.
\]
It respects refinements by direct cancellation of the invertible denominator images. The pushout relation $Q(e')Q(g)=Q(g')Q(e)$ makes $\Phi$ natural for all original arrows. Since both functors invert the denominators, it is also natural for their inverses and hence for all arrows of $\Loc$.

Let $\eta_X=[1_X,1_X,1_X]$. An inverse to $\Phi_Y$ is
\[
 \Psi_Y(\alpha)=\overline H_X(\alpha)(\eta_X).
\]
Naturality and $\Phi_X(\eta_X)=1_{QX}$ give $\Phi_Y\Psi_Y=1$. To check the other composite on $[m,f,e]$, first note that $H_X(m)$ sends $[m,1_U,1_U]$ to $[m,m,1_X]=\eta_X$. Thus $\overline H_X((Qm)^{-1})\eta_X=[m,1_U,1_U]$. Applying $Qf$ gives $[m,f,1_V]$, and applying $(Qe)^{-1}$ gives $[m,f,e]$ by the inverse formula for $\Eden$. Hence $\Psi_Y\Phi_Y=1$, proving~\eqref{di:eq:colimit} and the asserted equality criterion.
\end{proof}

\begin{theorem}[The complete three-arrow calculus]\label{di:thm:composition}
For the normal denominator classes above, the quotient sets
$\pi_0\TFrac(X,Y)$ form a category under the composition induced by
\eqref{di:eq:diagram-composition}. This category is canonically
isomorphic, over $\C$, to
$\C[(\Sden\cup\Eden)^{-1}]$. More explicitly:
\begin{enumerate}[label=\textup{(\roman*)}]
\item every localized arrow has a representative $e\backslash f/m$;
\item two representatives define the same arrow if and only if they
have a common reduction, with equality tested by
\eqref{di:eq:equality} after choosing common denominators;
\item $\star$ is epidd in each variable, so common reduction is a
congruence; associativity holds modulo common reduction;
\item $f\mapsto[1\backslash f/1]$ defines the localization functor.
The inverses of $Qm$ and $Qe$ are represented, respectively, by
$1_U\backslash1_U/m$ and $e\backslash1_V/1_V$;
\item every functor $T$ inverting the denominators extends uniquely,
in the identity-on-objects model, by
\[
 \overline T[e\backslash f/m]=(Te)^{-1}(Tf)(Tm)^{-1}.
\]
Every natural transformation between such functors extends uniquely
with the same components.
\end{enumerate}
\end{theorem}
\begin{proof}
Theorem~\ref{di:thm:fractions} identifies each proposed hom-set
bijectively with the corresponding hom-set of the ordinary
localization. This proves (i) and (ii) independently of any
associativity assertion about raw diagrams.
Equation~\eqref{di:eq:composition-value}, applied to $T=Q$, shows
that the operation $\star$ realizes the ordinary composite under
these bijections. If either input is replaced by a reduction, its
localized value is unchanged. Therefore the two output values are
equal, and (ii) supplies a common reduction of the output diagrams.
This proves the epidd condition on comparable inputs. By
Lemma~\ref{lem:epidd}, it proves the condition on all related inputs;
by Proposition~\ref{prop:epidd-products}, it also proves joint
compatibility. Equivalently it is enough to check the two types of
elementary reductions, since their relation $R$ has
$R\circ R=\preceq$ by Lemma~\ref{di:lem:reduction-posets}.

The two bracketings of three diagrams have equal values by
\eqref{di:eq:composition-value} and associativity in the ordinary
localization. Part (ii) supplies a common reduction, proving
associativity modulo $\sim$. The unit diagrams are already units
before reduction. Proposition~\ref{prop:epidd-congruence} now applies
in its general form and proves the category assertion and (iii).
The bijections are compatible with the indicated composition and
units, hence define the asserted canonical isomorphism of categories.

The inverse representatives in (iv) have the claimed values by
cancellation, so their two inverse equations follow from (ii).
The assignment on original arrows respects composition and
identities, either by the construction with identity denominators
or by the canonical isomorphism. Every representative decomposes as
\begin{equation}\label{di:eq:three-factor}
 [e\backslash f/m]
 = [e\backslash1_V/1_V]\,
   [1_V\backslash f/1_U]\,
   [1_U\backslash1_U/m].
\end{equation}
This proves uniqueness in (v). Its proposed formula is unchanged
under reduction and preserves composition by
\eqref{di:eq:composition-value}; it therefore proves existence
without further identifications. For a transformation $\alpha:T\Rightarrow T'$,
use its components on the unchanged object set. Naturality for an
original denominator implies naturality for its inverse by
multiplication by inverse images, and hence naturality for
\eqref{di:eq:three-factor}. These components are forced, proving
existence and uniqueness for arbitrary transformations.
\end{proof}

\begin{example}[Raw ternary composition need not be associative]\label{di:ex:raw-associativity}
Let $\C$ be a small skeleton of finite pointed sets and put $\N=\C$.
A normal monomorphism is a pointed subset inclusion, and a normal
epimorphism collapses a pointed subset to the basepoint without other
identifications. Hence the normal maps are those injective away from
their zero fibre. They compose, so $\C$ is di-exact homological by
Proposition~\ref{prop:diexact-equivalences}. The denominator classes
here are all normal monomorphisms and all normal epimorphisms.

Set $J=\{*,1\}$ and $B=\{*,b,c\}$. Let $e:B\to J$ send $b$ to
$*$ and $c$ to $1$, and let $g:B\to J$ send both $b,c$ to $1$.
Write $i:0\to J$ and $q:J\to0$. Consider the three diagrams
\[
 \rho=e\backslash1_J/1_J:J\to B,\qquad
 \sigma=1_J\backslash g/1_B:B\to J,\qquad
 \tau=1_0\backslash1_0/i:J\to0.
\]
The pushout of $e$ and $g$ is $0$: $e$ identifies $b$ with the
basepoint, and $g(b)=g(c)$ then forces both elements of $J$ to the
basepoint. Thus
$\sigma\star\rho=q\backslash0_{J,0}/1_J$.
Formula~\eqref{di:eq:diagram-composition} gives
\[
 \tau\star(\sigma\star\rho)
       =1_0\backslash0_{J,0}/1_J.
\]
On the other hand $\Ker g=0$, so
$\tau\star\sigma=1_0\backslash1_0/(0\to B)$ and
\[
 (\tau\star\sigma)\star\rho
       =1_0\backslash1_0/i.
\]
These diagrams are not isomorphic: their source-denominator domains
are $J$ and $0$. The second is a source restriction of the first,
so they have a common reduction, as the theorem requires. In
particular one cannot in this generality call the raw diagrams an
epidd-enriched \emph{category} before imposing common reductions.
The category is obtained from the epidd composition calculus by
Proposition~\ref{prop:epidd-congruence}. This example involves the
nonnormal middle map $g$.
\end{example}

\subsection{Exactness and precise annihilation of the saturated quotient}\label{di:sec:quotient}

\begin{theorem}\label{di:thm:quotient}
For $\C$ di-exact homological and $\N\subseteq\C$ saturated thick, the localization
\[
 Q_{\N}:\C\longrightarrow\C/\N:=\C[(\Sden\cup\Eden)^{-1}]
\]
has the following properties.
\begin{enumerate}[label=\textup{(\roman*)}]
\item $\C/\N$ is pointed and $Q_{\N}X\cong0$ if and only if $X\in\N$.
\item $Q_{\N}$ preserves every kernel and every cokernel.
\item $\C/\N$ is di-exact homological.
\item For every pointed category $\D$ with kernels and cokernels, precomposition induces an equivalence
\begin{equation}\label{di:eq:2up}
 \Ex(\C/\N,\D)\simeq
 \{F\in\Ex(\C,\D)\mid F(X)\cong0\text{ for }X\in\N\},
\end{equation}
where the right side is full and both sides retain arbitrary natural transformations.
\end{enumerate}
\end{theorem}
\begin{proof}
Write $Q=Q_{\N}$ and $\Loc=\C/\N$.

\smallskip\noindent\emph{Pointedness and precise annihilation.}
The hom-set formula makes $Q0$ both initial and terminal: a normal subobject of $0$ is zero, and a normal quotient of $0$ is zero, so each relevant filtered colimit is a singleton. Thus $Q$ preserves zero morphisms. If $X\in\N$, the arrow $0\to X$ belongs to $\Sden$ and $QX\cong0$. Conversely, if $QX\cong0$, then $Q1_X=Q0_{X,X}$. The equality criterion supplies $m:U\to X$ in $\Sden$ and $e:X\to V$ in $\Eden$ with $em=0$. Hence $m$ factors through $\ker e$, and its induced map into $\Ker e$ is normal monic by Lemma~\ref{di:lem:basic}. Thus $U\in\N$. Also $\Coker m\in\N$, and the short exact sequence determined by $m$ gives $X\in\N$.

\smallskip\noindent\emph{Images of normal monomorphisms are monic.}
Let $k:K\to X$ be any normal monomorphism. Put two arrows $\alpha,\beta:QZ\to QK$ over common denominators $m:U\to Z$ and $e:K\to V$, with numerators $a,b$. Since $e$ is normal epic, Lemma~\ref{di:lem:basic} supplies its pushout along $k$:
\[
\begin{tikzcd}
K\arrow[r,"k"]\arrow[d,"e"']&X\arrow[d,"p"]\\
V\arrow[r,"k'"']&W.
\end{tikzcd}
\]
Here $p\in\Eden$ and $k'$ is normal monic, by denominator stability and Grandis's law. If $Qk\,\alpha=Qk\,\beta$, there are refinements $r\in\Sden$ and $t:W\to W'$ in $\Eden$ such that
\[
 tk'ar=tk'br.
\]
Di-exactness gives $tk'=nd$ with $n$ normal monic and $d$ normal epic. Its kernel is a normal subobject of $\Ker t\in\N$, so $d\in\Eden$. Cancel $n$ to obtain $dar=dbr$. This is a common refinement of the representatives of $\alpha,\beta$, proving $\alpha=\beta$. Thus $Qk$ is monic.

\smallskip\noindent\emph{Preservation of kernels of arbitrary arrows.}
Let $f:X\to Y$, and let $k:K\to X$ be its kernel. Suppose $\alpha:QZ\to QX$ satisfies $Qf\,\alpha=0$. Represent it as
\[
 \alpha=(Qe)^{-1}(Qa)(Qm)^{-1},\qquad
 m:U\to Z\text{ in }\Sden,\quad e:X\to V\text{ in }\Eden.
\]
Again Lemma~\ref{di:lem:basic} supplies the pushout of the normal epimorphism $e$ along $f$:
\[
\begin{tikzcd}
X\arrow[r,"f"]\arrow[d,"e"']&Y\arrow[d,"p_0"]\\
V\arrow[r,"f_0"']&Y_0.
\end{tikzcd}
\]
Then $p_0\in\Eden$. The equation $Qf\,\alpha=0$ says $[m,f_0a,p_0]=0$. After a source refinement and a target refinement $d:Y_0\to Y_1$ in $\Eden$, we may replace $m,a$ by their source refinements and assume
\[
 h a=0,\qquad h e=p f,
 \quad h=df_0:V\to Y_1,\quad p=dp_0:Y\to Y_1\text{ in }\Eden.
\]
Set $j=\ker h:K'\to V$ and $l=\ker(pf):L\to X$. There is a pullback square
\[
\begin{tikzcd}
L\arrow[r,"v"]\arrow[d,"l"']&K'\arrow[d,"j"]\\
X\arrow[r,"e"']&V.
\end{tikzcd}
\]
Grandis's law makes $v$ normal epic; its kernel object is canonically $\Ker e$, so $v\in\Eden$.

Factor $k=lu$. The map $u:K\to L$ is $\ker(fl)$, by monicity of $l$ and the universal property of $k$. Since $pfl=0$, the normal image of $fl$ is a normal subobject of $\Ker p\in\N$. Lemma~\ref{di:lem:saturated-properties}(iv) therefore gives
\[
 \Coker u=\Ncm(fl)\in\N,
\]
so $u\in\Sden$. Finally $a=ja'$ for a unique $a':U\to K'$. The arrow
\[
 \beta=\bigl(Q(vu)\bigr)^{-1}(Qa')(Qm)^{-1}:QZ\longrightarrow QK
\]
is defined, since both $Qu$ and $Qv$ are invertible. The equations $el=jv$ and $k=lu$ give $Qk\,\beta=\alpha$. Uniqueness follows from the already proved monicity of $Qk$. Therefore $Qk$ is a kernel of $Qf$.

All the hypotheses and the construction are self-dual: the opposite localization is the corresponding localization of $\C^{\mathrm{op}}$. Applying the preceding argument there proves preservation of every cokernel. This proves (ii).

\smallskip\noindent\emph{The quotient stays di-exact homological.}
For a general fraction
\[
 \phi=(Qe)^{-1}(Qf)(Qm)^{-1}:QX\longrightarrow QY,
\]
we have, up to their universal isomorphisms,
\begin{equation}\label{di:eq:fraction-kc}
 \ker\phi=Q(m\,\ker f),\qquad
 \coker\phi=Q((\coker f)e).
\end{equation}
Both inside composites are normal in $\C$. Hence every normal monomorphism in $\Loc$ is isomorphic over its codomain to the image of an ambient normal monomorphism, and dually for normal epimorphisms under their domains.

For two composable normal monomorphisms in $\Loc$, represent the outer one by an ambient normal monomorphism. Transport the inner one across the resulting isomorphism, and represent it in turn over the ambient middle object. Their composite is isomorphic to the image of a composite of normal monomorphisms in $\C$, so is normal. Dualize for normal epimorphisms. For a normal monomorphism followed by a normal epimorphism with common middle object $QB$, formula~\eqref{di:eq:fraction-kc} represents them, up to isomorphisms of the outside objects, by $Qm$ and $Qe$ for an ambient pair with middle object $B$. The composite $em$ is normal by di-exactness of $\C$; its image is normal by (ii). Thus every such composite in $\Loc$ is normal. This proves (iii), including the conditional homology axiom.

\smallskip\noindent\emph{The exact functor and natural-transformation universal property.}
An exact functor $F:\C\to\D$ annihilating $\N$ sends $s\in\Sden$ to a normal monomorphism with zero cokernel, and sends $e\in\Eden$ to a normal epimorphism with zero kernel. Such arrows are isomorphisms: use $\ker(\coker s)=s$ and its dual. Consequently there is a unique ordinary functor $\overline F:\Loc\to\D$ with $\overline FQ=F$. Formula~\eqref{di:eq:fraction-kc}, together with exactness of $F$ and invertibility of denominator images, proves that $\overline F$ preserves every kernel and cokernel. It preserves the zero object since $\overline F(Q0)=F0$. Thus it is exact. Conversely, an exact functor out of $\Loc$ restricts to one annihilating $\N$ by (i).

Let $R,S:\Loc\to\D$ be exact, and let $\alpha:RQ\Rightarrow SQ$ be any natural transformation. Since $Q$ is the identity on objects, its extension must have components $\overline\alpha_X=\alpha_X$. Naturality holds for $Qf$ by assumption. For a denominator $w$, naturality for $Qw$ implies naturality for $(Qw)^{-1}$ by composing with $R(Qw)^{-1}$ and $S(Qw)^{-1}$. Since these arrows generate $\Loc$, the components form a natural transformation, uniquely. No component of $\alpha$ was assumed invertible. This proves full faithfulness and essential surjectivity in~\eqref{di:eq:2up}.
\end{proof}

\begin{remark}
Equivalently one may localize at the class of \emph{normal} morphisms whose kernel and cokernel objects lie in $\N$: each such morphism has a factor in $\Eden$ followed by a factor in $\Sden$, and every arrow of $\Sden\cup\Eden$ belongs to that class.
\end{remark}

\subsection{An axiomatic version in terms of normal denominators}\label{subsec:di-exact-denominators}

The preceding construction also completes the formulation in terms
of a single class of normal morphisms. The unrestricted base-change
clause below is important; it is not the clause restricted to normal
base-change maps that suffices only on the exact core.

\begin{theorem}\label{di:thm:denominator-correspondence}
In a di-exact homological category $\C$, saturated thick subcategories
correspond to classes $\mathcal W$ of \emph{normal} morphisms satisfying:
\begin{enumerate}[label=\textup{(TF\arabic*)},leftmargin=19mm]
\item $\mathcal W$ contains the isomorphisms and has two-out-of-three
for a composable triple $f,g,gf$ of normal morphisms;
\item normal monomorphisms in $\mathcal W$ are stable under pullback
along arbitrary morphisms, and normal epimorphisms in $\mathcal W$
are stable under pushout along arbitrary morphisms;
\item if $f=me\in\mathcal W$ with $m$ normal monic, then
$m\in\mathcal W$; dually, if $e$ is normal epic, then
$e\in\mathcal W$;
\item in a commutative square
\[
\begin{tikzcd}
 A\arrow[r,two heads,"e'"]\arrow[d,tail,"m"']&
 B\arrow[d,tail,"m'"]\\
 C\arrow[r,two heads,"e"']&D,
\end{tikzcd}
\]
whose vertical maps are normal monic and horizontal maps normal epic,
$e\in\mathcal W$ implies $e'\in\mathcal W$, and
$m\in\mathcal W$ implies $m'\in\mathcal W$.
\end{enumerate}
The mutually inverse assignments are
\begin{align*}
 \mathcal W_{\N}
 &=\{f\mid f\text{ is normal},\ \Ker f,\Coker f\in\N\},\\
 \N_{\mathcal W}
 &=\{X\mid 0\to X\in\mathcal W\}
  =\{X\mid X\to0\in\mathcal W\}.
\end{align*}
For every such $\mathcal W$, localization at $\mathcal W$ has the
three-arrow calculus of Theorem~\ref{di:thm:composition}, with normal
monic and normal epic denominators in $\mathcal W$, and the exact
quotient properties of Theorem~\ref{di:thm:quotient}.
\end{theorem}
\begin{proof}
Let $\N$ be saturated thick and $Q:\C\to\C/\N$ its exact quotient.
For a normal $f$, normality of $Qf$ and precise annihilation imply
\begin{equation}\label{di:eq:exact-inverted}
 f\in\mathcal W_{\N}\quad\Longleftrightarrow\quad Qf
 \text{ is invertible}.
\end{equation}
Indeed, the forward implication follows by factoring $f$ normally;
the reverse implication follows because $Q$ preserves its kernel
and cokernel and annihilates precisely $\N$. Property (TF1) follows
from two-out-of-three for isomorphisms. Property (TF2) is
Lemma~\ref{di:lem:denominators}(ii). For (TF3), $Qm$ is monic
and is a left factor of the isomorphism $Qf$, hence is also split
epic and therefore invertible; use \eqref{di:eq:exact-inverted}.
The dual argument treats $Qe$. For (TF4), if $Qe$ is invertible,
then $Qm$ monic and $Qm'Qe'=QeQm$ imply that $Qe'$ is monic.
It is a normal epimorphism, so it is invertible. Dually, if $Qm$
is invertible, the equation makes $Qm'$ epic; a normal monomorphism
which is epic has zero cokernel and is invertible. This proves all
four conditions.

Conversely, suppose (TF1)--(TF4), and define $\N=\N_{\mathcal W}$.
The two definitions of $\N$ agree by (TF1), applied to
$0\to X\to0$ whose composite is $1_0$; all three maps are normal.
The class contains zero objects and is replete.
For a normal epimorphism $e:A\to B$ with kernel $k:K\to A$,
the square
\[
\begin{tikzcd}
 K\arrow[r]\arrow[d,tail,"k"']&0\arrow[d,tail]\\
 A\arrow[r,two heads,"e"']&B
\end{tikzcd}
\]
is a pushout. If $e\in\mathcal W$, (TF4) puts $K\to0$ in
$\mathcal W$. Conversely, if $K\in\N$, (TF2) applied to that
pushout gives $e\in\mathcal W$. The dual argument gives
\[
 e\in\mathcal W\Longleftrightarrow\Ker e\in\N,
 \qquad
 m\in\mathcal W\Longleftrightarrow\Coker m\in\N
\]
for normal epimorphisms $e$ and normal monomorphisms $m$.
For a cokernel--kernel factorization $f=me$, (TF3) and (TF1) now give
\[
 f\in\mathcal W\Longleftrightarrow e,m\in\mathcal W
 \Longleftrightarrow\Ker f,\Coker f\in\N.
\]

In a short exact sequence $A\xrightarrow{m}B\xrightarrow{e}C$,
if $B\in\N$, factor $0\to B$ through $m$. Property (TF3) gives
$m\in\mathcal W$, and then (TF1) gives $0\to A\in\mathcal W$.
The dual argument gives $C\in\N$. If instead $A,C\in\N$, the
criterion for $m$ gives $m\in\mathcal W$, and composing it with
$0\to A$ puts $B$ in $\N$. Thus $\N$ is thick.
To verify saturation, let $f:X\to Y$ be arbitrary and suppose
$Y,\Ker f\in\N$. Pull $0\to Y\in\mathcal W$ back along $f$.
Property (TF2) gives $\ker f\in\mathcal W$. Composing this with
$0\to\Ker f\in\mathcal W$ gives $0\to X\in\mathcal W$ by
(TF1). Thus $X\in\N$. The dual proves the second saturation
implication. This proves both inverse identities.

Finally, a normal map in $\mathcal W_{\N}$ factors as an
$\Eden$-denominator followed by an $\Sden$-denominator, and both
normal denominator classes are contained in $\mathcal W_{\N}$.
The two localizations therefore have identical universal properties.
The preceding fraction and quotient theorems apply.
\end{proof}

\begin{remark}
The restriction to \emph{normal triples} in (TF1) and to normal maps
in \eqref{di:eq:exact-inverted} must not be dropped. The full inverse
image of the isomorphisms under $Q$ can contain nonnormal arrows.
The theorem characterizes the normal denominator class, not necessarily
all morphisms inverted by $Q$. For Puppe exact categories this
distinction disappears, saturation is automatic, and the statement
recovers the classical exact-isokernel correspondence, recorded below
as Proposition~\ref{prop:exact-isokernels}. On an arbitrary
di-exact category it adds precisely the arbitrary base-change and
saturation information absent from the exact-core construction.
\end{remark}

\subsection{Nested quotients and two-dimensional homologicity}\label{di:sec:2dim}

\begin{lemma}\label{di:lem:nested}
Let $\N\subseteq\M\subseteq\C$ be saturated thick subcategories of a di-exact homological category. The induced functor
\[
 j:\M/\N\longrightarrow\C/\N
\]
is fully faithful and exact. If $p:\C/\N\to\C/\M$ is the induced exact functor, then $j$ identifies $\M/\N$, up to repletion, with the full annihilation subcategory of $p$. Moreover,
\[
 (\C/\N)/(\M/\N)\simeq\C/\M,
\]
where the middle subcategory is understood as its replete image.
\end{lemma}
\begin{proof}
The restriction assertion of Lemma~\ref{di:lem:saturated-properties} allows the quotient $\M/\N$. In the hom-set formula~\eqref{di:eq:colimit} with $X,Y\in\M$, every intermediate source $U$ is a normal subobject of $X$, and every intermediate target $V$ is a normal quotient of $Y$. Hence both are in $\M$. Kernels and cokernels in $\M$ are computed in $\C$, so the allowed denominators and all refinements are identical in both formulas. Thus $j$ induces a bijection on every hom-set. It is exact by the exact universal property of the quotient of $\M$.

The functor $p$ exists by that same universal property, and $pQ_{\N}=Q_{\M}$. Since the quotients are the identity on objects and have the precise annihilation property,
\[
 p(Q_{\N}X)\cong0\quad\Longleftrightarrow\quad X\in\M.
\]
This identifies the full annihilation subcategory with the replete essential image of $j$. In particular that image is saturated thick, by Lemma~\ref{di:lem:saturated-properties}(ii). Finally, for an exact functor $R:\C/\N\to\D$, annihilation of this image is equivalent to annihilation of $\M$ by $RQ_{\N}$. Apply~\eqref{di:eq:2up} twice, including its natural-transformation assertion, to obtain the displayed equivalence and its exact universal property.
\end{proof}

\begin{proposition}\label{di:prop:2kc}
In $\DiHom$, a $2$-kernel of an exact $F:\C\to\D$ is
\[
 \Ker(F)=\{X\in\C\mid F(X)\cong0\}\hookrightarrow\C,
\]
and a $2$-cokernel of an exact $G:\mathcal A\to\C$ is
\[
 Q_{\N}:\C\longrightarrow\C/\N,
 \qquad\N=\SatTh_{\C}(G(\mathcal A)).
\]
In particular, $2$-kernels are, up to equivalence, precisely the inclusions of saturated thick subcategories, and $2$-cokernels are, up to equivalence, precisely their quotient functors.
\end{proposition}
\begin{proof}
The full annihilation subcategory is saturated thick and di-exact homological. Any exact functor annihilated by $F$ corestricts to it; the corestriction is exact because all kernels and cokernels are ambient. Fullness gives existence and uniqueness of the lifted natural transformations. This is the $2$-kernel universal property.

For cokernels, an exact $T:\C\to\D$ annihilates $G$ if and only if it annihilates $\N$: one direction is immediate, and the other follows because $\Ker T$ is saturated thick and contains $G(\mathcal A)$. Apply~\eqref{di:eq:2up}. Conversely, for every saturated thick $\N$, its inclusion is the $2$-kernel of $Q_{\N}$ by precise annihilation, and $Q_{\N}$ is the $2$-cokernel of that inclusion. The general characterizations now follow from uniqueness of $2$-kernels and $2$-cokernels up to equivalence.
\end{proof}

\begin{theorem}\label{di:thm:main}
The $2$-category $\DiHom$ of small di-exact homological
categories, functors preserving zero objects and all kernels and
cokernels, and arbitrary natural transformations, is $2$-pointed
and $2$-homological. Its $2$-kernels are, up to equivalence, saturated
thick inclusions, and its $2$-cokernels are their localizations.
For $\N\subseteq\M\subseteq\C$ saturated thick, the square
\[
\begin{tikzcd}
 \M\arrow[r,tail]\arrow[d,"Q_{\N}^{\M}"']&
 \C\arrow[d,"Q_{\N}^{\C}"]\\
 \M/\N\arrow[r,tail,"j"']&\C/\N
\end{tikzcd}
\]
is both a bipullback and a bipushout.
\end{theorem}
\begin{proof}
The one-object zero category is di-exact homological. An exact
functor into it is unique; an exact functor out of it selects a zero
object. Between any exact functor and any parallel null exact
functor the unique component maps to, and from, zero objects are
natural. They are permitted $2$-cells. Thus the null functors are
zero objects of their hom-categories, proving the two-sided
$2$-pointedness of Definition~\ref{def:2pointed}.

\smallskip\noindent\emph{Existence and composition of $2$-kernels.}
Proposition~\ref{di:prop:2kc} constructs all $2$-kernels and
$2$-cokernels and gives their normal forms. Replace two composable
$2$-kernels, up to equivalence, by successive saturated thick
inclusions $\N\hookrightarrow\M\hookrightarrow\C$.
Transitivity in Lemma~\ref{di:lem:saturated-properties}(iii) makes
$\N$ saturated thick in $\C$. Its inclusion is therefore another
$2$-kernel by Proposition~\ref{di:prop:2kc}.

\smallskip\noindent\emph{Composition of $2$-cokernels.}
By the same normal forms it suffices to consider consecutive quotients
\[
 \C\xrightarrow{q}\C/\N
 \xrightarrow{r}(\C/\N)/\Pcal,
\]
where $\Pcal$ is saturated thick in $\C/\N$.
Put $\M=q^{-1}(\Pcal)$. This is saturated thick in $\C$ by
Lemma~\ref{di:lem:saturated-properties}(iii), and contains $\N$.
For every di-exact homological target $\X$, the exact quotient
universal properties of Theorem~\ref{di:thm:quotient} give
\begin{align*}
 \Ex((\C/\N)/\Pcal,\X)
 &\simeq\{R\in\Ex(\C/\N,\X)\mid R(\Pcal)\cong0\}\\
 &\simeq\{H\in\Ex(\C,\X)\mid H(\M)\cong0\}.
\end{align*}
For the second equivalence, a functor annihilating $\M$ descends
through $q$ since $\N\subseteq\M$. Its descendant annihilates
$\Pcal$: $q$ is the identity on objects and every object of
$\Pcal$ is $qM$ for an object $M$ of $\M$. The converse is
immediate. The quotient properties include arbitrary natural
transformations, so both displayed equivalences are equivalences of
full categories, not merely bijections of isomorphism classes.
Consequently $rq$ is a $2$-cokernel of $\M\hookrightarrow\C$,
and is equivalent under $\C$ to $Q_{\M}$.

\smallskip\noindent\emph{The conditional homology axiom.}
Represent a $2$-kernel followed by a $2$-cokernel as
\[
 \M\xhookrightarrow{i}\C\xrightarrow{Q_{\N}^{\C}}\C/\N.
\]
By precise annihilation, the hypothesis that the $2$-kernel of the
second map factors through the first is exactly
$\N\subseteq\M$. Their composite factors as
\[
 \M\xrightarrow{Q_{\N}^{\M}}\M/\N
       \xrightarrow{j}\C/\N.
\]
The first arrow is a $2$-cokernel by Proposition~\ref{di:prop:2kc}.
By Lemma~\ref{di:lem:nested}, the second identifies its domain,
up to repletion, with the full annihilator of
$p:\C/\N\to\C/\M$. Its corestriction to that annihilator is an
exact equivalence: a quasi-inverse to a fully faithful, essentially
surjective exact functor preserves kernels and cokernels by their
universal properties. Hence $j$ is a $2$-kernel. This proves (H3)
directly, without invoking any quotient criterion.

\smallskip\noindent\emph{The two universal properties of the square.}
Write $q=Q_{\N}^{\C}$ and $q'=Q_{\N}^{\M}$, choosing the
identity-on-objects models so that $qi=jq'$. For a bipullback cone
consisting of exact $U:\X\to\C$, $V:\X\to\M/\N$ and an
isomorphism $qU\cong jV$, composition with $p$ gives
$Q_{\M}U\cong0$. Thus $U$ corestricts to an exact
$\widetilde U:\X\to\M$. Full faithfulness of $j$ lifts the
specified isomorphism uniquely to $q'\widetilde U\cong V$.
For cones coming from $H,K:\X\to\M$, a compatible pair of
transformations $\alpha:iH\Rightarrow iK$ and
$\beta:q'H\Rightarrow q'K$ satisfies $q\alpha=j\beta$.
Fullness of $i$ gives a unique $\gamma:H\Rightarrow K$ with
$i\gamma=\alpha$, and faithfulness of $j$ then gives
$q'\gamma=\beta$. This proves the full cone-category equivalence.

For a bipushout cocone take exact $H:\C\to\X$,
$K:\M/\N\to\X$ and an isomorphism $Hi\cong Kq'$.
It forces $H$ to annihilate $\N$, so $H$ descends exactly through
$q$. If $\overline H$ is its descendant, the comparison
$(\overline Hj)q'\cong Kq'$ descends uniquely to
$\overline Hj\cong K$, by full faithfulness of precomposition with
$q'$. For two such cocones, a transformation of their $H$-components
extends uniquely through $q$; compatibility with their
$K$-components is equivalent, after precomposition with $q'$, to
the original cocone equation. Full faithfulness for $q'$ then
establishes that compatibility and uniqueness. This proves the
cocone-category equivalence and completes the square assertion.
\end{proof}

\begin{corollary}\label{di:cor:inclusions}
The inclusions $\pExact\to\DiHom\to\HomCat$ preserve strong
bizero objects, local zero objects and zero morphisms, and the
$2$-kernels and $2$-cokernels constructed in their domains.
\end{corollary}
\begin{proof}
The local zero-object argument and the full-annihilation-kernel
construction are unchanged and allow test categories in the larger
$2$-category. For a di-exact source, the quotient universal property
in Theorem~\ref{di:thm:quotient}(iv) allows every pointed target with
kernels and cokernels, hence in particular every homological target.
This proves the assertion for $\DiHom\to\HomCat$.

For a Puppe exact source, saturated thick and thick subcategories
coincide. Moreover, every numerator in a ternary fraction is normal,
so its image and its conjugates by denominator isomorphisms are
normal. Therefore every arrow of its saturated quotient is normal;
the quotient is Puppe exact. It is the classical quotient by the
same universal property, and Theorem~\ref{di:thm:quotient}(iv)
allows all di-exact test targets. Thus its $2$-cokernel property is
preserved under $\pExact\to\DiHom$ as well.
\end{proof}

\begin{remark}[Scope of the direct proof]
The proof separates the fraction-theoretic input from the
$2$-categorical conclusions: denominator stability, equality by
common reduction, exactness of localization, and full faithfulness
of nested quotients are established before the three homological
axioms are verified. The exact core of
Proposition~\ref{prop:exact-core} controls normal subquotients and
normal denominators, while saturation and arbitrary base change
control possibly nonnormal numerators. We next specialize to Puppe exact
categories and compare with Grandis's classical construction.
\end{remark}

\subsection{The classical Puppe-exact quotient and its fraction calculus}

We now specialize the completed di-exact construction to Puppe
exact categories. Saturation is automatic in this case, and every
numerator is normal. The following specialization recovers
Grandis's exact localization and ternary-fraction construction
\cite[Section~5.5]{Grandis13}, \cite{GrandisFractions92}.
We state the result in the pointed setting and derive it from
the preceding di-exact construction.

\begin{theorem}[Grandis's exact quotient theorem]\label{thm:puppe-localization}
Let $\N$ be a thick subcategory of a small Puppe exact category $\C$,
and put
\[
 W_{\N}=\{f\in\Mor\C\mid\Ker(f),\Coker(f)\in\N\}.
\]
The ordinary localization $\C[W_{\N}^{-1}]$ has a Puppe-exact
structure for which its identity-on-objects quotient $q$ is exact.
It annihilates precisely the objects of $\N$ and inverts precisely
$W_{\N}$. Exact functors annihilating $\N$ factor uniquely through
$q$ in the canonical localization model. Every morphism of the
localization has a ternary representative
\begin{equation}\label{eq:puppe-ternary}
 X\xleftarrow{m}U\xrightarrow{f}V\xleftarrow{e}Y,
 \qquad m\in W_{\N}\text{ monic},\quad
 e\in W_{\N}\text{ epic},
\end{equation}
with value $q(e)^{-1}q(f)q(m)^{-1}$. Equality of representatives is
characterized by a common reduction.
\end{theorem}

\begin{proof}
Thickness implies saturation: if $Y,\Ker f\in\N$, cokernel--kernel
factorization expresses $X$ as an extension of $\Ker f$ by
$\Nim f$, and $\Nim f$ is a subobject of $Y$. The dual proves
the other saturation implication. Theorems~\ref{di:thm:composition}
and \ref{di:thm:quotient} therefore apply. Every quotient arrow is
an image of a normal numerator conjugated by isomorphisms, so it is
normal; hence the quotient is Puppe exact. Exactness of the quotient
functor and precise annihilation show that $qf$ is invertible exactly
when $\Ker f,\Coker f\in\N$: a normal morphism with zero kernel
and cokernel is invertible. The stated normal form and equality
criterion are those of Theorem~\ref{di:thm:composition}, and its
universal property, restricted to exact functors, is
Theorem~\ref{di:thm:quotient}(iv).
\end{proof}

The homological structure of the ordinary categories of exact,
Puppe-exact and abelian categories is described in
\cite[Section~1.9, p.~146]{Grandis92} and
\cite[Section~5.6]{Grandis13}.
We write $\C/\N$ for this quotient when no confusion with the
other quotient constructions is possible.

For clarity, a reduction of \eqref{eq:puppe-ternary} is obtained by
replacing $m$ by $mi$ and $e$ by $je$, where $i$ is a monomorphism
in $W_{\N}$ and $j$ is an epimorphism in $W_{\N}$, and replacing
$f$ by $jfi$. The data are taken up to isomorphism of their two
intermediate objects. These operations leave the displayed value
unchanged. The common-reduction assertion in the theorem says that
they also detect all equalities; it is stronger than just the
existence of a representative. A common reduction uses intersections
of source subobjects and common quotients of target quotients.
This explains why the two intermediate objects stay inside a thick
subcategory containing the two endpoints. Related three-arrow
methods, in a different localization setting, appear in \cite{Tho11}.

\begin{remark}\label{remternaryfractionsSerrequotient}
For an abelian category the preceding quotient is the Serre quotient,
and \eqref{eq:puppe-ternary} is the familiar description by a map from
a subobject of $X$ to a quotient of $Y$, with the discarded quotient
and subobject lying in $\N$. In the triangulated case two-arrow roofs
replace these three-arrow fractions. Both are ordinary localizations;
this common feature controls arbitrary natural transformations.
\end{remark}

\begin{lemma}[Natural transformations descend through localization]\label{lem:puppe-nat}
Let $q:\C\to\C[W^{-1}]$ be an ordinary localization.
For any category $\D$, precomposition with $q$ is fully faithful on
functor categories. Its essential image consists of the functors
inverting $W$. The statement includes noninvertible natural
transformations.
\end{lemma}
\begin{proof}
Use the identity-on-objects model generated by the arrows of $\C$
and formal inverses to $W$, subject to the composition relations and
the two inverse equations. A functor inverting $W$ extends by sending
a formal inverse to the inverse image arrow; the relations prove
existence and uniqueness. A transformation between two such functors
has already prescribed components on every object. Naturality for
$s:X\to Y$ in $W$ reads
$G(s)\alpha_X=\alpha_YF(s)$ and is equivalent to
$\alpha_XF(s)^{-1}=G(s)^{-1}\alpha_Y$.
It therefore gives naturality for the formal inverse as well.
Naturality for words follows by composition, and the relations
preserve it. This proves existence and uniqueness of the descended
transformation. Passing to equivalent models gives the stated
essential-image formulation.
\end{proof}

\begin{proposition}[Exact isokernels]\label{prop:exact-isokernels}
Thick subcategories of a Puppe exact category correspond to classes
$W$ of morphisms with the following properties:
\begin{enumerate}[label=\textup{(F\arabic*)},leftmargin=16mm]
\item $W$ contains isomorphisms and has two-out-of-three for composition;
\item monomorphisms in $W$ are stable under pullback, and epimorphisms
in $W$ under pushout;
\item if $f=me\in W$ and $m$ is monic then $m\in W$; dually, if $e$
is epic then $e\in W$;
\item in a commutative square with horizontal epimorphisms $e',e$ and
vertical monomorphisms $m,m'$, if $e\in W$ then $e'\in W$, and if
$m\in W$ then $m'\in W$.
\end{enumerate}
The correspondence sends $\N$ to $W_{\N}$ and $W$ to
$\{X\mid 0\to X\in W\}$. The latter condition is equivalent to
$X\to0\in W$ and to admitting a short exact sequence through $X$
whose two nonzero arrows belong to $W$.
\end{proposition}
\begin{proof}
In a Puppe exact category every morphism is normal and
saturation is automatic. Thus the correspondence and
(F1)--(F4) are the specialization of
Theorem~\ref{di:thm:denominator-correspondence}.
For a short exact sequence $A\to X\to B$, the kernel and
cokernel descriptions of $W$ show that its two maps belong
to $W$ exactly when $A,B\in\N$; extension closure then
gives $X\in\N$. Conversely, for $X\in\N$ use
$0\to X\xrightarrow{1_X}X$.
\end{proof}

\subsection{Two-dimensional quotients and normal functors}

\begin{theorem}\label{thm:puppe-characterization}
The $2$-category $\pExact$ has all $2$-kernels and $2$-cokernels.
For $F:\A\to\B$, a $2$-cokernel is
\[
 \B\longrightarrow\B/\thick_{\B}(F(\A)).
\]
An exact functor $F:\A\to\B$ is a $2$-kernel precisely when it is
fully faithful and has thick replete essential image. It is a
$2$-cokernel precisely when it is essentially surjective and the
induced functor
\[
 \widehat F:\A/\Ker F\longrightarrow\B
\]
is fully faithful. More generally $F$ is $2$-normal precisely when
$\widehat F$ is fully faithful and the replete essential image of
$F$ is thick.
\end{theorem}
\begin{proof}
Kernels were constructed in Proposition~\ref{prop:puppe-kernels}.
An exact functor annihilating $F(\A)$ annihilates its thick closure,
and Theorem~\ref{thm:puppe-localization} gives its exact factorization
through the displayed quotient. Lemma~\ref{lem:puppe-nat} supplies
full faithfulness on arbitrary natural transformations. This proves
the $2$-cokernel universal property.

A fully faithful exact functor with thick essential image identifies
its domain, as a Puppe exact category, with that image. The image is
the annihilation kernel of its quotient. Hence the functor is a
$2$-kernel. Conversely every $2$-kernel is equivalent over its
codomain to the annihilation kernel already constructed, so it is
fully faithful and has thick image.

The induced $\widehat F$ exists since $F$ annihilates $\Ker F$.
The first quotient has the same objects as $\A$, and thus
$\widehat F$ is essentially surjective if and only if $F$ is.
If it is also fully faithful, it is an exact equivalence. Its
quasi-inverse is exact because an equivalence transports all the
kernel and cokernel universal properties. Composing this equivalence
with the quotient exhibits $F$ as a $2$-cokernel. Conversely a
$2$-cokernel is a $2$-cokernel of its own $2$-kernel by
Proposition~\ref{propkeriskerofitscoker}; uniqueness of $2$-cokernels
makes $\widehat F$ an equivalence.

Finally suppose $\widehat F$ is fully faithful with thick image.
By the kernel characterization it is a $2$-kernel, so the displayed
factorization of $F$ is $2$-normal. Conversely if $F$ is $2$-normal,
Proposition~\ref{prop:normal-factor} identifies this factorization,
up to equivalence, with its normal factorization. Its second factor
is fully faithful and has thick image, as required.
\end{proof}

In terms of \eqref{eq:puppe-ternary}, full faithfulness of
$\widehat F$ means both existence of a representing ternary fraction
for every arrow $F(X)\to F(Y)$ and equality detection by common
reduction. Neither essential surjectivity nor fullness alone substitutes
for equality detection. The class of $2$-normal functors need
not be closed under composition.

\begin{theorem}\label{thm:puppe-homological}
The $2$-category $\pExact$ is $2$-pointed and $2$-homological.
For $\N\subseteq\M\subseteq\C$ thick, the fully faithful functor
$\M/\N\to\C/\N$ is the $2$-kernel of $\C/\N\to\C/\M$.
The nested subquotient square is a bipullback and a bipushout.
\end{theorem}
\begin{proof}
The annihilation kernels and quotient functors of
Theorem~\ref{thm:puppe-characterization} are the same as those
constructed in $\DiHom$. Their domains and codomains are Puppe
exact: this is immediate for full annihilators, and was proved for
quotients in Theorem~\ref{thm:puppe-localization}. Thus a $1$-cell
between Puppe exact categories is a $2$-kernel or a $2$-cokernel
in $\pExact$ precisely when it is of the corresponding normal form
in $\DiHom$. This assertion uses the explicit normal forms and
uniqueness of representing objects, not closure of an arbitrary
sub-$2$-category under universal constructions.

Successive thick inclusions compose because thickness is transitive.
Successive quotients compose to the quotient by the inverse image
of the second thick subcategory, by the direct hom-category argument
in Theorem~\ref{di:thm:main}; this inverse image is thick and its
quotient is Puppe exact. For a kernel followed by a cokernel with
$\N\subseteq\M$, the factorization
$\M\to\M/\N\to\C/\N$ likewise takes place entirely in
Puppe exact categories. Lemma~\ref{di:lem:nested} identifies its
second arrow as the full annihilation kernel of
$\C/\N\to\C/\M$. These are the two composition laws and the
conditional homology axiom. The strong bizero argument was given
above. Finally the bipullback and bipushout proofs in
Theorem~\ref{di:thm:main} apply to Puppe-exact test categories and
the same intermediate quotients, establishing both universal
properties in $\pExact$.
\end{proof}

\subsection{Why there is no unconditional interchange theorem}

\begin{example}\label{ex:J-counterexample}
The canonical functor $\mathbf J\vee\mathbf J\to\mathbf J\times\mathbf J$
is exact and has zero $2$-kernel and zero $2$-cokernel, but is not an
equivalence. Its image consists of the two axes. They contain every
nonzero object of the source, so the kernel is zero. The remaining
object $(J,J)$ is an extension of $(J,0)$ by $(0,J)$, so the thick
closure of the image is the whole target, and the cokernel is zero.
The functor is not essentially surjective. Were it $2$-normal, its
normal comparison would be an equivalence by
Proposition~\ref{prop:normal-factor}, and with these zero kernel
and cokernel the functor itself would be an equivalence.
Thus $\pExact$ is not $2$-Puppe exact.
\end{example}

The following stronger example also corrects the possible inference
that all $2$-kernel--$2$-cokernel composites are $2$-normal. It uses only small
abelian categories, so no size issue is involved.

\begin{proposition}\label{prop:abelian-not-diexact}
Neither $\AbCat$ nor $\pExact$ is $2$-di-exact in the sense of
Definition~\ref{def:2-diexact}. In each, the class of $2$-normal
functors is not closed under composition.
\end{proposition}
\begin{proof}
Fix a field $k$. Let $\C$ be the category of finite-dimensional
representations of $1\to2\to3$ and let $\D$ be that of $1\to3$.
Kernels and cokernels are componentwise. Let $\M\subseteq\C$
consist of the representations with $V_2=0$, and let
$\N\subseteq\C$ consist of those with $V_1=V_3=0$. Both are
Serre subcategories. The inclusion $I:\M\to\C$ is the $2$-kernel
of evaluation at vertex $2$.

Define the exact functor
\[
 Q:\C\longrightarrow\D,\qquad
 (V_1\xrightarrow aV_2\xrightarrow bV_3)
 \longmapsto(V_1\xrightarrow{ba}V_3).
\]
It has an exact section
$L(U\xrightarrow hW)=(U\xrightarrow1U\xrightarrow hW)$.
On morphisms,
\[
 Q(f_1,f_2,f_3)=(f_1,f_3),\qquad
 L(f_1,f_3)=(f_1,f_1,f_3).
\]
Because kernels and cokernels of representations are computed
at each vertex, these formulas show that both functors preserve
kernels and cokernels, as well as zero objects, proving their
exactness.
There is a natural transformation $\theta:LQ\Rightarrow1_{\C}$
with components $(1,a,1)$, whose kernel and cokernel are in $\N$.
Also $QL=1$ and $Q\theta=1_Q$.
If an exact $T:\C\to\E$ to a Puppe exact category annihilates
$\N$, then $T\theta$ has zero kernel and cokernel, hence is
invertible. Thus $T\cong(TL)Q$ with $TL$ exact. For exact
$R,S:\D\to\E$ and any $\alpha:RQ\Rightarrow SQ$, define
$\beta=\alpha L$. Naturality on $\theta_X$, and $Q\theta_X=1$,
gives $\alpha_X=\beta_{QX}$. Hence $\beta Q=\alpha$, uniquely
since $QL=1$. This proves that $Q$ is the $2$-cokernel of
$\N\hookrightarrow\C$, both in $\pExact$ and in $\AbCat$.

The replete image of $QI$ is the class of zero-arrow representations
of $1\to3$. It is not extension closed, as the short exact sequence
\[
 0\longrightarrow(0\to k)\longrightarrow
 (k\xrightarrow1 k)\longrightarrow(k\to0)\longrightarrow0
\]
shows. Any $2$-cokernel in either $2$-category is essentially
surjective, and any $2$-kernel has thick (respectively Serre)
replete image. Therefore a $2$-normal functor has such an image,
and $QI$ is not $2$-normal. Both $I$ and $Q$ are $2$-normal
individually, completing the proof. Here the abelian normal forms
follow from the classical Serre quotient theorem
\cite{Gabriel,StacksSerreQuotient}, together with full annihilation
kernels and descent of transformations through localization
(Lemma~\ref{lem:puppe-nat}). They do not require the general quotient
criterion or the later abelian homologicity proof.
\end{proof}

\begin{remark}
The obstruction is extension closure of the quotient image, not a
failure of the existence of fractions. It rules out a sub-$2$-category
with all objects and precisely the $2$-normal functors
as $1$-cells. The valid replacement is the conditional homology axiom
of Theorem~\ref{thm:puppe-homological}.
\end{remark}

\section{Exactness of the 2-category of triangulated categories}\label{sec:tri-direct}

In this section, we prove that the 2-kernels and 2-cokernels in the 2-category of triangulated categories precisely capture the widely used concepts of thick triangulated subcategories and Verdier localizations. As a consequence, our theory unifies Serre quotients and Verdier localizations, via 2-dimensional categorical algebra. We then show further characterizations of 2-kernels and 2-cokernels in the 2-category of triangulated categories, parallel to those established above for $\pExact$ and the Serre quotient specialization. Finally, we prove that the 2-category of triangulated categories is Grandis homological in a 2-dimensional sense.

We first recall the definition of triangulated category. Our main reference is \cite{Neeman,StacksTri}.

\begin{definition}
    A category $\C$ is called \dfn{additive} if it is pointed, it is enriched over the category $\Ab$ of abelian groups and it has finite biproducts.
\end{definition}

\begin{definition}
Let $\T$ be an additive category.

A \dfn{translation functor} on $\T$ is an additive endoequivalence $\Sigma:\T \to \T$.

A \dfn{triangle} for the additive category $\T$ equipped with the translation functor $\Sigma:\T \to \T$ is a sequence of morphisms in $\T$ of the form
\begin{cd}
X \arrow[r, "u"] \& Y \arrow[r, "v"]  \& Z  \arrow[r, "w"] \&  \Sigma X.
\end{cd}

A \dfn{morphism of triangles} is a triple $(f,g,h)$ of morphisms in $\T$ such that the following diagram is commutative
\begin{cd}
X \arrow[r, "u"] \arrow[d,"f"] \& Y \arrow[r, "v"] \arrow[d,"g"]\& Z  \arrow[r, "w"] \arrow[d,"h"] \&  \Sigma X \arrow[d,"\Sigma f"]\\
X' \arrow[r, "u'"] \& Y' \arrow[r, "v'"]  \& Z'  \arrow[r, "w'"] \&  \Sigma X'.
\end{cd}
\end{definition}

\begin{definition}
    A \dfn{triangulated category} $\T$ is an additive category equipped with a translation functor $\Sigma:\T \to \T$ and a specified class of triangles that are called \dfn{exact triangles} (or distinguished triangles), such that:
\begin{itemize}
\item [(TR0)] exact triangles are closed under isomorphisms
\item[(TR1)] for every $X\in \T$ the triangle
$X \aar{\id{X}} X \aar{}   0  \aar{}  \Sigma X$
is an exact triangle
\item [(TR2)] for every morphism $X \aar{a} Y$ in $\T$ there exists an exact triangle of the form $X \aar{a} Y \aar{}   Z  \aar{}  \Sigma X$
\item [(TR3)]  $X \aar{u} Y \aar{v}   Z  \aar{w}  \Sigma X$ is an exact triangle if and only if  $Y \aar{v}   Z  \aar{w}  \Sigma X \aar{-\Sigma u} \Sigma Y$ is an exact triangle
\item [(TR4)] given a commutative diagram with exact triangles as rows as follows
\begin{cd}
X \arrow[r, "u"] \arrow[d,"f"] \& Y \arrow[r, "v"] \arrow[d,"g"]\& Z  \arrow[r, "w"] \&  \Sigma X \arrow[d,"\Sigma f"]\\[-3ex]
X' \arrow[r, "u'"] \& Y' \arrow[r, "v'"]  \& Z'  \arrow[r, "w'"] \&  \Sigma X'
\end{cd}
  there exists a morphism $Z \aar{h} Z'$ that completes the diagram to a morphism of triangles.
\item [(TR5)] given exact triangles $X \aar{u} Y \aar{} Z' \aar{} \Sigma X$, $Y \aar{v} Z \aar{}   X'  \aar{}  \Sigma Y$ and $X \aar{v \c u} Z \aar{}   Y'  \aar{}  \Sigma X$, there exists an exact triangle $Z' \aar{} Y' \aar{}   X' \aar{}  \Sigma Z'$ making the following diagram commutative
\begin{cd}
X \& Y \& {Z'} \& {\Sigma X} \\
	X \& Z \& {Y'} \& {\Sigma X} \\
	Y \& Z \& {X'} \& {\Sigma Y} \\
	{Z'} \& {Y'} \& {X'} \& {\Sigma Z'}
	\arrow["u", from=1-1, to=1-2]
	\arrow[equals, from=1-1, to=2-1]
	\arrow[from=1-2, to=1-3]
	\arrow["v", from=1-2, to=2-2]
	\arrow[from=1-3, to=1-4]
	\arrow[dashed, from=1-3, to=2-3]
	\arrow[equals, from=1-4, to=2-4]
	\arrow["{v \circ u}"', from=2-1, to=2-2]
	\arrow["u"', from=2-1, to=3-1]
	\arrow[from=2-2, to=2-3]
	\arrow[equals, from=2-2, to=3-2]
	\arrow[from=2-3, to=2-4]
	\arrow[dashed, from=2-3, to=3-3]
	\arrow["{\Sigma u}", from=2-4, to=3-4]
	\arrow["v"', from=3-1, to=3-2]
	\arrow[from=3-1, to=4-1]
	\arrow[from=3-2, to=3-3]
	\arrow[from=3-2, to=4-2]
	\arrow[from=3-3, to=3-4]
	\arrow[equals, from=3-3, to=4-3]
	\arrow[from=3-4, to=4-4]
	\arrow[dashed, from=4-1, to=4-2]
	\arrow[dashed, from=4-2, to=4-3]
	\arrow[dashed, from=4-3, to=4-4]
\end{cd}

\end{itemize}
\end{definition}

\begin{remark}\label{remexacttrianglesasses}
    We can think of the exact triangles as a replacement for short exact sequences. The idea is that some important constructions, like passing to the homotopy or derived category of an abelian category, do not in general produce abelian categories, but do produce triangulated structures where the exact triangles somehow play the role of short exact sequences.

    Any composite of two consecutive morphisms in an exact triangle is zero: apply (TR4) from $X\xrightarrow{1_X}X\to0\to\Sigma X$ to $X\xrightarrow{u}Y\xrightarrow{v}Z\to\Sigma X$, using $1_X$ and $u$ as the first two components, to obtain $vu=0$; rotation gives the other composites. So in this sense, exact triangles are complexes.
\end{remark}

\begin{remark}
    The rotation axiom of exact triangles (axiom (TR3)) can be intuitively seen as a form of self-duality of the triangulated structure.
\end{remark}

\begin{remark}\label{remconesascokernels}
    Axiom (TR2) ensures that we can complete any morphism $X \aar{a} Y$ to an exact triangle. The object $Z$ in such a chosen exact triangle is called a \textit{cone} of the morphism $a$ and it can be intuitively thought as an analogue of the cokernel of $a$. By the remark above, combining (TR2) with (TR3), it is also possible to obtain the analogues of kernels.

    Axiom (TR4) is what replaces the universal property of cokernels. The induced morphisms are not unique, but some reformulations of (TR5) in the literature can be seen as conditions that require the morphisms induced by (TR4) to be nice.
\end{remark}

The theory of triangulated categories has been extensively used to tackle and solve homological and cohomological problems in contexts that lack short exact sequences. A preeminent class of examples of triangulated categories is the following.

\begin{example}
    Let $\A$ be an abelian category.

    The category $\operatorname{C}(\A)$ of cochain complexes is abelian. Its homotopy category $\operatorname{K}(\A)$ is triangulated: translation shifts a complex and negates its differential, and exact triangles are those isomorphic to mapping-cone triangles. The derived category $\operatorname{D}(\A)$ is the Verdier localization of $\operatorname{K}(\A)$ at the acyclic complexes, equivalently at quasi-isomorphisms. The bounded versions $\operatorname{K}^b(\A)$ and $\operatorname{D}^b(\A)$ give essentially small examples under our size convention. These standard constructions and their triangulated axioms are established in \cite[Chapter III]{GelfandManin} and \cite{Verdier,StacksDerived}.
\end{example}

The following elementary results about triangulated categories will be useful for us.

\begin{lemma}[\cite{GelfandManin}]\label{lemmaisozero}
   Let $\T$ be a triangulated category and let
   $$X \aar{f} Y \aar{g} Z \aar{h} \Sigma X$$
   be an exact triangle in $\T$. Then the morphism $f$ is an isomorphism if and only if $Z$ is the zero object.
\end{lemma}
\begin{proof}
Applying $\operatorname{Hom}(W,-)$ gives an exact sequence extending in both directions. If $Z=0$, the map $\operatorname{Hom}(W,f)$ is bijective for every $W$, so $f$ is invertible by Yoneda. Conversely, if $f$ is invertible, the same sequence, using also invertibility of $\Sigma f$, gives $\operatorname{Hom}(W,Z)=0$ for every $W$. Take $W=Z$ to obtain $1_Z=0$.
\end{proof}

\begin{remark}
    \lemx \ref{lemmaisozero} is the counterpart for triangulated categories of the fact that given a short exact sequence
    $$0\aar{} X\aar{f} Y \aar{} Z \aar{} 0$$
    in an abelian category $\mathcal{A}$, the morphism $f$ is an isomorphism if and only if $Z$ is the zero object in $\mathcal{A}$. This shows once again how exact triangle can be seen as good replacements of short exact sequences.
\end{remark}

\begin{definition}
    An \dfn{exact} functor $\T \to \U$ between triangulated categories is a pair $(F,\mu)$ where $F\: \T \to \U$ is an additive functor and $\mu\: F \c \Sigma^{\T} \Rightarrow \Sigma^{\U} \c F$ is an invertible natural transformation such that for every exact triangle in $\T$
\vspace{-1.5mm}
\begin{cd}
X \arrow[r, "u"] \& Y \arrow[r, "v"]  \& Z  \arrow[r, "w"] \&  \Sigma^{\T} X
\vspace{-2.5mm}
\end{cd}
the triangle
\vspace{-1.5mm}
\begin{cd}
FX \arrow[r, "Fu"] \& FY \arrow[r, "Fv"]  \& FZ  \arrow[r, "\mu_X \c Fw"] \&  {\Sigma ^{\U} FX}
\vspace{-1.5mm}
\end{cd}
is exact in $\U$.
\end{definition}

\begin{remark}
    Along the lines of \remx\ref{remexacttrianglesasses} and \remx\ref{remconesascokernels}, that intuitively see exact triangles as short exact sequences and cones as cokernels, we can see exact functors between triangulated categories as an analogue of the exact functors between abelian categories.
\end{remark}

\begin{definition}
    An \dfn{exact natural transformation} $\alpha$ between exact functors $(F,\mu^F),(G,\mu^G)$ from $\T$ to $\U$ triangulated categories is a natural transformation $\alpha\:F\aR{}G$ such that the following equality holds:
    \begin{eqD*}
        \begin{cd}*
            \T \& \U \\
	\T \& \U
	\arrow[""{name=0, anchor=center, inner sep=0}, "G", bend left=45, from=1-1, to=1-2]
	\arrow[""{name=1, anchor=center, inner sep=0}, "F"',bend right=20, from=1-1, to=1-2]
	\arrow[""{name=2, anchor=center, inner sep=0}, "{\Sigma^{\T}}"', from=1-1, to=2-1]
	\arrow[""{name=3, anchor=center, inner sep=0}, "{\Sigma^{\U}}", from=1-2, to=2-2]
	\arrow["F"', from=2-1, to=2-2]
	\arrow["{\alpha}"',shorten <=0.5ex, shorten >=0.5ex, Rightarrow, from=1, to=0]
	\arrow["{\mu^F}"',shorten <=2.5ex, shorten >=2.5ex, Rightarrow, from=2, to=3]
        \end{cd} \quad = \quad
        \begin{cd}*
            \T \& \U \\
	\T \& \U
	\arrow["G", from=1-1, to=1-2]
	\arrow[""{name=0, anchor=center, inner sep=0}, "{\Sigma^{\T}}"', from=1-1, to=2-1]
	\arrow[""{name=1, anchor=center, inner sep=0}, "{\Sigma^{\U}}", from=1-2, to=2-2]
	\arrow[""{name=2, anchor=center, inner sep=0}, "F"', bend right=45, from=2-1, to=2-2]
	\arrow[""{name=3, anchor=center, inner sep=0}, "G",bend left=20, from=2-1, to=2-2]
	\arrow["{\mu^G}",shorten <=2.5ex, shorten >=2.5ex, Rightarrow, from=0, to=1]
	\arrow["\alpha"',shorten <=0.5ex, shorten >=0.5ex, Rightarrow, from=2, to=3]
        \end{cd}
    \end{eqD*}
\end{definition}

\begin{proposition}
    Triangulated categories, exact functors and exact natural transformations form a 2-category, that we will denote $\Triang$.
\end{proposition}
\begin{proof}
    The composite $(G,\mu^G)(F,\mu^F)$ has shift comparison $(\mu^G F)(G\mu^F)$. It preserves exact triangles by applying $F$ and then $G$. The identity functor has identity shift comparison. Vertical and horizontal composites of transformations satisfying the displayed shift equation satisfy it by pasting the two equations. The ordinary associativity, unit, and interchange equations in $\Cat$ therefore restrict to these data.
\end{proof}

\begin{remark}
    One could potentially consider a looser version of the 2-category $\Triang$ of triangulated categories, where 2-cells are all natural transformations. Interestingly, all the results of this section hold similarly for such looser 2-category as well.
\end{remark}

\begin{proposition}\label{propffinCatisffinTriang}
    Let $F\:\T\to \U$ be an exact functor. If $F$ is a fully faithful functor, then it is also a fully faithful arrow in $\Triang$. That is, for all exact functors $H,H'\: \mathcal{V} \to \T$ and every exact natural transformation
\begin{cd}[4.5][4.5]
\V  \& \T \arrow[d, Rightarrow, "\lambda", shorten <= -0.5ex, shorten <= -0.5ex] \& \U \\[-3ex]
	\& \T
	\arrow[from=1-1, to=1-2,"H"]
	\arrow[from=1-1, to=2-2, bend right=20,"{H'}"']
	\arrow[from=1-2, to=1-3,"F"]
	\arrow[from=2-2, to=1-3, bend right=20,"F"']
\end{cd}
  there exists a unique exact natural transformation $\overline{\lambda}\: H \Rightarrow H'$ such that $F \star \overline{\lambda} = \lambda$.

  Dually, if $F$ is co-fully faithful as a functor then it is also a co-fully faithful arrow in $\Triang$.
\end{proposition}
\begin{proof}
    It is well-known that fully faithful functors are precisely fully faithful arrows in $\Cat$. Then, starting from the underlying natural transformation of $\lambda$, there exists a unique natural transformation $\overline{\lambda}\:H\aR{}H'$ such that $F\star \overline{\lambda}=\lambda$. It remains to prove that $\overline{\lambda}$ is exact. But the needed axiom can be checked after whiskering with $F$, since $F$ is fully faithful, and that holds because $\lambda$ is an exact natural transformation.

    For the co-fully faithful assertion, lift the underlying transformation by the co-full faithfulness of $F$ in $\Cat$. The two sides of its shift-compatibility equation are transformations between functors out of the codomain of $F$. After precomposition with $F$ and transport through the invertible shift comparison $\mu^F$, their equality is the shift-compatibility equation for the given transformation. Faithfulness of precomposition with $F$ therefore gives the required equality before precomposition. Uniqueness is inherited from the underlying natural-transformation lifting.
\end{proof}

\begin{proposition}\label{propequivinTriang}
    For an exact functor between triangulated categories, being an equivalence in the 2-category $\Triang$ is the same as being an equivalence in the 2-category $\Cat$ of categories.
\end{proposition}
\begin{proof}
    Let $F\:\S\to \T$ be an exact functor between triangulated categories. Of course, if $F$ is an equivalence in $\Triang$, then the underlying functors and natural transformations involved exhibit an equivalence in $\Cat$. Assume now that the underlying functor of $F$ is an adjoint equivalence in $\Cat$, exhibited by a pseudoinverse $G\:\T\to \S$, a unit $\eta\:\Id{}\aR{}G\c F$ and a counit $\epsilon\:F\c G\aR{}\Id{}$. We prove that $G$ can be automatically equipped with the structure of an exact functor, so that unit and counit become exact natural transformations and the adjoint equivalence is lifted to $\Triang$. Surely, $G$ is automatically an additive functor, since equivalences preserve limits and colimits. Then call $\Sigma^{\S}$ and $\Sigma^{\T}$ the translation functors associated to $\S$ and $\T$ respectively, and call $\mu^F$ the isomorphism $F\c \Sigma^{\S}\iso \Sigma^{\T}\c F$ that exhibits $F$ as an exact functor. The mate of $\mu^F$, with respect to the adjunction $F\dashv G$, yields a natural isomorphism $\mu^G\:G\c \Sigma^{\T}\iso \Sigma^{\S}\c G$, as needed. Explicitly, given $A\in \T$, the component $\mu^G_A$ of $\mu^G$ on $A$ is given by the composite
    $$G(\Sigma^{\T}(A))\aiso{G\Sigma^{\T}\epsilon^{-1}} G(\Sigma^{\T}(F(G(A))))\aiso{G(\mu^{F})^{-1}G}G(F(\Sigma^{\S}(G(A))))\aiso{\eta^{-1}\Sigma^{\S}G}\Sigma^{\S}(G(A))$$

    We check the shift equations for the unit and counit explicitly.
The triangular identity $F(\eta_X^{-1})=\epsilon_{F(X)}$
and naturality of $\epsilon$ give, for $A\in\T$,
\begin{align*}
 F(\mu^G_A)
 &= (\mu^F_{G(A)})^{-1}
    (\Sigma^{\T}(\epsilon_A))^{-1}
    \epsilon_{\Sigma^{\T}(A)},\\
 \Sigma^{\T}(\epsilon_A)\,\mu^F_{G(A)}\,F(\mu^G_A)
 &=\epsilon_{\Sigma^{\T}(A)}.
\end{align*}
The second equation is shift compatibility of the counit. For
$X\in\S$, the first equation, naturality of $\epsilon$ and
$\mu^F$, and $\epsilon_{F(X)}^{-1}=F(\eta_X)$ give
\begin{align*}
 F\bigl(\mu^G_{F(X)}\,G(\mu^F_X)\,\eta_{\Sigma^{\S}(X)}\bigr)
 &= (\mu^F_{GF(X)})^{-1}
    (\Sigma^{\T}(\epsilon_{F(X)}))^{-1}\mu^F_X\\
 &=F\bigl(\Sigma^{\S}(\eta_X)\bigr).
\end{align*}
Faithfulness of $F$ gives shift compatibility of the unit.

    We show that $G$ preserves exact triangles. So let
    $$A\aar{l}B\aar{m}C\aar{n}\Sigma^{\T}(A)$$ be an exact triangle in $\T$. We need to prove that
    $$G(A)\aar{G(l)}G(B)\aar{G(m)}G(C)\aar{\mu^G_A\c G(n)}\Sigma^{\S}(G(A))$$
    is an exact triangle in $\S$. Surely, $G(l)$ can be completed to an exact triangle in $\S$
    $$G(A)\aar{G(l)}G(B)\aar{v}W\aar{w}\Sigma^{\S}(G(A))$$
    Then since $F$ preserves exact triangles, also
    $$F(G(A))\aar{F(G(l))}F(G(B))\aar{F(v)}F(W)\aar{\mu^F_{G(A)}\c F(w)}\Sigma^{\T}(F(G(A)))$$
    is an exact triangle. But notice that such triangle is isomorphic to the starting one. Indeed, the commutative square
    \twosquare[n][5][5]{F(G(A))}{F(G(B))}{A}{B}{F(G(l))}{\epsilon_A}{\epsilon_B}{l}
    given by the naturality of the counit $\epsilon$ induces a morphism of triangles
    \begin{cd}[4]
        {F(G(A))} \& {F(G(B))} \& {F(W)} \& {\Sigma^{\T}(F(G(A)))} \\
	A \& B \& C \& {\Sigma^{\T}(A)}
	\arrow["{F(G(l))}", from=1-1, to=1-2]
	\arrow["{\epsilon_A}"',aiso, from=1-1, to=2-1]
	\arrow["{F(v)}", from=1-2, to=1-3]
	\arrow["{\epsilon_B}"',aiso, from=1-2, to=2-2]
	\arrow["{\mu^F_{G(A)}\c F(w)}", from=1-3, to=1-4]
	\arrow["j", dashed, from=1-3, to=2-3]
	\arrow["{\Sigma^{\T}(\epsilon_A)}", from=1-4, to=2-4]
	\arrow["l"', from=2-1, to=2-2]
	\arrow["m"', from=2-2, to=2-3]
	\arrow["n"', from=2-3, to=2-4]
    \end{cd}

    And by the triangulated five lemma, $j\:F(W)\to C$ must be an isomorphism. Then, since $F$ is fully faithful, the isomorphism $F(W)\aar{j}C\aar{\epsilon_C^{-1}}F(G(C))$ in $\T$ induces a unique isomorphism $j'\:W\iso G(C)$ in $\S$ such that $F(j')=\epsilon_C^{-1}\c j$. We can now see that the following is an isomorphism of triangles
    \begin{cd}[4]
    {G(A)} \& {G(B)} \& W \& {\Sigma^{\S}(G(A))} \\
	{G(A)} \& {G(B)} \& {G(C)} \& {\Sigma^{\S}(G(A))}
	\arrow["{G(l)}", from=1-1, to=1-2]
	\arrow[equals, from=1-1, to=2-1]
	\arrow["v", from=1-2, to=1-3]
	\arrow[equals, from=1-2, to=2-2]
	\arrow["w", from=1-3, to=1-4]
	\arrow["j'", dashed, from=1-3, to=2-3]
	\arrow[equals, from=1-4, to=2-4]
	\arrow["{G(l)}"', from=2-1, to=2-2]
	\arrow["{G(m)}"', from=2-2, to=2-3]
	\arrow["{\mu^G_A\c G(n)}"', from=2-3, to=2-4]
    \end{cd}
    since the squares commute after applying the fully faithful functor $F$, using naturality and the counit shift equation proved above. Thus $G$ is exact, and the verified shift equations make $\eta$ and $\epsilon$ exact natural transformations, proving the assertion.
\end{proof}

We show that the $2$-category $\Triang$ has a strong bizero object
in the two-sided sense of Definition~\ref{def:2pointed}.

\begin{proposition}\label{prop:tri-strongzero}
The one-object zero triangulated category $\0$ is a strong bizero
object in $\Triang$. In particular, a null exact functor is both
initial and terminal in its hom-category.

Consequently, $\Triang$ has the canonical $2$-ideal of null
$1$-cells and null $2$-cells and the associated notions of
$2$-kernel and $2$-cokernel; see Remark~\ref{remcanonical2ideal}.
\end{proposition}
\begin{proof}
Equip the category $\0$ with its unique triangulated structure.
There is a unique exact functor $\T\to\0$, including its shift
isomorphism, and a unique transformation between any two such
functors. An exact functor $\0\to\T$ selects a zero object of
$\T$, with the unique possible shift isomorphism. Any two such
choices are related by a unique natural isomorphism, which is
shift-compatible because all its components involve zero objects.
Thus $\0$ is a bizero.

An exact functor is null if and only if it is zero-valued. Indeed,
the unique maps between zero objects identify a zero-valued functor
with a constant zero functor, and this isomorphism is shift-compatible
since its components and those of both shift isomorphisms are maps
between zero objects.

Let $F:\T\to\U$ be arbitrary exact and let $Z:\T\to\U$ be
null. The unique maps $F(X)\to Z(X)$ and $Z(X)\to F(X)$ assemble
into natural transformations $\alpha:F\Rightarrow Z$ and
$\beta:Z\Rightarrow F$: each naturality equation has a zero
codomain or zero domain. They are exact:
for every $X$, the required equations are
\[
 (\Sigma_{\U}\alpha_X)\mu^F_X
   =\mu^Z_X\alpha_{\Sigma_{\T}X},
 \qquad
 (\Sigma_{\U}\beta_X)\mu^Z_X
   =\mu^F_X\beta_{\Sigma_{\T}X}.
\]
In the first equation both sides have the same zero codomain
$\Sigma_{\U}Z(X)$; in the second they have the same zero domain
$Z(\Sigma_{\T}X)$. The equations therefore hold. The components
are forced in either direction, so these are the unique allowed
$2$-cells $F\Rightarrow Z$ and $Z\Rightarrow F$. This proves the
strong bizero property.
\end{proof}

\begin{definition}
A \dfn{zero exact functor} is a null exact functor, that is, an exact
functor isomorphic to one factoring through $\0$. Equivalently, it
is an exact functor taking every object to a zero object.
\end{definition}

We record explicitly the two-sided compatibility used above.

\begin{lemma}\label{lemmanattozeroisexact}
Let $F:\T\to\U$ be an exact functor and $Z:\T\to\U$ a zero
exact functor. There is exactly one natural transformation
$F\Rightarrow Z$ and exactly one $Z\Rightarrow F$, and both are
exact. Consequently $\0$ is also a strong bizero when all natural
transformations are admitted as $2$-cells.
\end{lemma}
\begin{proof}
Existence and uniqueness follow componentwise from the zero-object
property. The two shift-compatibility equations are exactly those
verified in Proposition~\ref{prop:tri-strongzero}. Admitting all
natural transformations leaves these two singleton sets unchanged.
\end{proof}

We want to study the 2-categorical kernels and cokernels in the 2-category $\Triang$ of triangulated categories. We will show that, interestingly, 2-kernels and 2-cokernels correspond respectively to the fundamental concepts of thick triangulated subcategory and of Verdier localization of a triangulated category. We first recall these fundamental notions. We take \cite{StacksTri} as main reference.

\begin{definition}
    A non-empty full subcategory $\S$ of a triangulated category $\T$ is a \textbf{triangulated subcategory} if
\begin{itemize}
\item [(TS1)] $\Sigma^n X\in \S$ for every $X\in \S$ and every $n\in \mathbb{Z}$;
\item [(TS2)] given any exact triangle $X \aar{} Y \aar{} Z \aar{} \Sigma X$, if two objects in $\{X,Y,Z\}$ are in $\S$ so is the third one.
\end{itemize}
A triangulated subcategory $\S$ is \textbf{thick} if in addition the following condition holds:
\begin{itemize}
\item [(TS3)] $\S$ is closed under retracts, i.e.\ given any morphisms $X \aar{\pi} Y$ and $Y\aar{\iota} X$ in $\T$ such that $\pi \circ \iota =\id{Y}$ and $X\in \S$, we have that $Y\in \S$.
\end{itemize}
\end{definition}

Non-emptiness and (TS2) already imply repleteness. For $X\in\S$,
the exact triangle $X\xrightarrow{1_X}X\to0\to\Sigma X$ puts
$0$ in $\S$. If $X\to Y$ is an isomorphism with $X\in\S$,
the exact triangle $X\to Y\to0\to\Sigma X$ then puts $Y$ in
$\S$. Thus ``full'' here is automatically ``full replete''.

For triangulated categories, this notion of thickness is the
``strictly full saturated triangulated subcategory'' convention
of \cite{StacksTri}. Here saturation refers to closure under
retracts; it differs from the condition on arbitrary morphisms
in Definition~\ref{di:def:saturated}. In the Puppe-exact and
di-exact settings, thickness instead refers to closure under
normal subquotients and extensions.

There is a known important notion in the theory of triangulated categories that naively generalizes the notion of kernel from algebra. We will show that such notion actually satisfies the 2-dimensional universal property of 2-categorical kernels. So it can be captured by our theory.

\begin{definition}\label{defker}
    Let $F\:\T \to \U$ be an exact functor. The \dfn{kernel} of $F$, denoted $\ker(F)$ is the full subcategory of $\T$ on the objects $X$ such that $F(X)$ is a zero object in $\U$.
\end{definition}

We now recall the fundamental concept of Verdier localization.

\begin{definition}\label{defVerdloc}
    Let $\S$ be a triangulated subcategory of a triangulated category $\T$. The \textbf{Verdier localization} of $\T$ by $\S$ is a triangulated category $\T/\S$ together with an exact functor $Q\:\T \to \T / \S$ which sends all objects of $\S$ to a zero object, that is the universal such, in the sense that for every triangulated category $\U$ and every exact functor $H\:\T \to \U$ that sends all objects of $\S$ to a zero object, there exists a unique exact functor $\overline{H}\: \T / \S \to \U$ such that the following triangle commutes
    \begin{cd}[4]
        \T \& {\T/\S} \\
	\& \U
	\arrow["Q", from=1-1, to=1-2]
	\arrow["H"', from=1-1, to=2-2]
	\arrow["{\overline{H}}", dashed, from=1-2, to=2-2]
    \end{cd}
\end{definition}

\begin{remark}\label{remtriangleinTriang}
    The triangle drawn in \defx\ref{defVerdloc} exhibiting the universal property of the Verdier localization is a commutative triangle in $\Triang$, respecting on the nose the isomorphisms regulating the image of translations. In particular, the identity 2-cell that fills it is an exact natural transformation.
\end{remark}

It is known that the Verdier localization of a triangulated category by a triangulated subcategory can be explicitly constructed as a category of fractions.

\begin{definition}
    Let $\T$ be a triangulated category. A set $\mathcal{F}$ of morphisms in $\T$ is a \dfn{multiplicative system} if it admits a calculus of left and right fractions (see \cite{GZ,StacksLoc} for further details). A multiplicative system $\mathcal{F}$ is \dfn{compatible with the triangulation} if
    \begin{itemize}
        \item [(1)] given a morphism $f \in \mathcal{F}$, the morphism $\Sigma^n f$ is in $\F$ for every $n\in \mathbb{Z}$;
        \item [(2)] given a morphism $(f,g,h)$ between exact triangles with $f,g\in \F$, there exists a morphism $(f,g,h')$ between the same triangles with $h'\in \F$.
    \end{itemize}
\end{definition}

\begin{construction}
    Let $\S$ be a triangulated subcategory of $\T$ and let $\mathcal{M}(\S)$ be the set of  morphisms $X \aar{f} Y$ in $\T$ such that there exists an exact triangle $X \aar{f} Y \aar{} Z \aar{} \Sigma X$ in $\T$ with $Z \in \S$. It is known that $\mathcal{M}(\S)$  is a multiplicative system compatible with the triangulation, see \cite{StacksTri}. The \textbf{Verdier localization} $\T/\S$ of $\T$ by $\S$ coincides with the category of fractions
$$\T [\mathcal{M}(\S)^{-1}]$$
equipped with its quotient functor.
\end{construction}

\begin{remark}
    The construction above reveals another similarity between abelian (and Puppe exact) categories and triangulated categories. The class $\mathcal M(\S)$ is determined by requiring a cone to lie in $\S$, whereas the class $W_{\N}$ in Remark~\ref{remternaryfractionsSerrequotient} is determined by requiring both kernel and cokernel to lie in $\N$. Exactness of a functor turns the corresponding vanishing conditions into invertibility. The fraction lengths differ, but in both settings the universal property extends to all natural transformations.
\end{remark}

The following classical proposition about Verdier localizations will be useful later in the paper. Its construction and proof can be found in \cite{Verdier,StacksTri,StacksTriQuotientKernel}.

\begin{proposition}\label{propVerdfacts}
   Let $\S$ be a triangulated subcategory of a triangulated category $\T$. Then the category of fractions $\T [\mathcal{M}(\S)^{-1}]$ carries its canonical triangulated structure, whose exact triangles are precisely the triangles isomorphic to images of exact triangles of $\T$, such that the quotient functor $Q\:\T\to \T [\mathcal{M}(\S)^{-1}]$ is exact. This yields an explicit construction of the Verdier localization $\T/\S$.

   Moreover, $\ker (Q)$ is the smallest (replete) thick subcategory of $\T$ containing $\S$.
\end{proposition}

\begin{remark}
    The previous result ensures that the Verdier localization of a triangulated category $\T$ by a triangulated subcategory $\S$ is determined by the smallest (replete) thick subcategory $\overline{\S}$ (also denoted $\langle\S\rangle$) of $\T$ containing $\S$. In fact, $\T/\S$ is canonically equivalent, as a triangulated category, to $\T/\overline{\S}$.
\end{remark}

We now prove a characterization of 2-kernels in the 2-category $\Triang$. In particular, we prove that the known notion of kernel of an exact functor satisfies the 2-dimensional universal property of a 2-kernel. So the 2-categorical theory introduced in \cite{CJM} and continued here captures kernels of exact functors between triangulated categories and unifies them with the kernels of exact functors between abelian (or Puppe exact) categories.

\begin{theorem}\label{teorchar2ker1}
    All 2-kernels in $\Triang$ exist and they are equivalently given by either of the following:
    \begin{itemize}
        \item [(i)] thick triangulated subcategories;
        \item [(ii)] kernels of exact functors in the sense of \defx \ref{defker}.
    \end{itemize}
\end{theorem}

\begin{proof}
The fact that the kernel of an exact functor in the sense of \defx \ref{defker} is a thick triangulated subcategory of the domain triangulated category is a well-known result. Indeed shifts and cones of maps between annihilated objects are annihilated by exactness, and retracts are annihilated by additivity; see also \cite{StacksTriKernel}. Moreover, it is also true that every thick triangulated subcategory can be seen as a kernel in the sense of \defx \ref{defker}. Indeed, given a thick triangulated subcategory $\S$ of a triangulated category $\T$, the quotient functor $Q\: \T \to \T /\S$ is an exact functor whose kernel is $\S$ thanks to \prox \ref{propVerdfacts}.

Let now $F\: \T \to \U$ be an exact functor between triangulated categories. We show that the thick triangulated subcategory $\ker(F)$ satisfies the 2-dimensional universal property of the 2-kernel. The composite
$$\ker(F) \hookrightarrow \T \aar{F} \U$$
is clearly isomorphic to a null morphism by definition of $\ker(F)$. Notice that by \lemx\ref{lemmanattozeroisexact} this natural isomorphism is also exact. Consider now an exact functor $M\: \S \to \T$ such that $F \c M$ is isomorphic to a null exact functor. This means that for every $X\in \S$ the object $F(M(X))\in \U$ is isomorphic to $0$ and thus $M(X)\in\ker(F)$ by definition of $\ker(F)$. So we have an exact functor $\S \aar{\widetilde{M}} \ker(F)$ such that the composite of $\widetilde{M}$ with the inclusion of $\ker(F)$ in $\T$ is equal to $M$. Since this is an equality of exact functors, also respecting the isomorphisms regulating the image of translations, the identity 2-cell between them is automatically exact. We show that, furthermore, the  2-dimensional universal property of the 2-kernel is satisfied. $\ker(F)\hookrightarrow \T$ is an exact functor that is fully faithful as a functor. So by \prox\ref{propffinCatisffinTriang} we obtain that it is a fully faithful arrow in $\Triang$. This precisely means that $\ker(F)\hookrightarrow \T$ satisfies the 2-dimensional universal property of the 2-kernel.

So every exact functor $F$ has a 2-kernel that can be explicitly constructed as $\ker(F)$. This also shows that every 2-kernel is equivalent in $\Triang$ to a kernel in the sense of \defx\ref{defker} (or equivalently to a thick triangulated subcategory), whence we conclude.
\end{proof}

\begin{remark}
    \thex \ref{teorchar2ker1} shows that the intuitive notion of kernel of \defx \ref{defker} actually coincides with that of 2-dimensional kernel in the 2-category $\Triang$. Notice that the morphism $\S \rightarrow \ker(F)$ induced by the 1-dimensional universal property of the 2-kernel $\ker(F)$ in the proof is actually the unique one that makes the triangle strictly commute, since the exact functor $\ker(F) \hookrightarrow \T$ is injective on objects and faithful and thus it is a monomorphism in the 1-category of categories.

    The fact that kernels in the sense of \defx \ref{defker} and equivalently thick triangulated subcategories are actually 2-dimensional kernels allows us to inscribe them in the same framework of other fundamental concepts such as that of Serre subcategory of an abelian category.
\end{remark}

\begin{corollary}\label{cor2kerareff}
    2-kernels in $\Triang$ are fully faithful functors.
\end{corollary}
\begin{proof}
    By \thex\ref{teorchar2ker1}, every 2-kernel $K\to \T$ of an exact functor $F\:\T\to \U$ is equivalent in $\Triang$ to the inclusion $\ker(F)\hookrightarrow \T$, which is a fully faithful functor. Since equivalences in $\Triang$ are in particular equivalences in $\Cat$, we conclude.
\end{proof}

We now prove that, furthermore, Verdier localizations are intimately connected with 2-cokernels in the 2-category $\Triang$. Indeed, we prove that Verdier localizations satisfy the 2-dimensional universal property of the 2-cokernel. Moreover, 2-cokernels in $\Triang$ can be constructed via Verdier localizations.

\begin{theorem}\label{teorchar2coker1}
 All 2-cokernels in $\Triang$ exist and they precisely correspond to Verdier localizations.
\end{theorem}
\begin{proof}
   We first prove that every Verdier localization is a 2-cokernel in $\Triang$. Let $\S$ be a thick triangulated subcategory of a triangulated category $\T$ and let $Q\: \T \to \T /\S$ be the corresponding Verdier localization. We prove that $Q$ is the 2-cokernel of the inclusion $\S\hookrightarrow \T$. The composite
   $$\S \hookrightarrow \T \aar{Q} \T / \S$$
   is isomorphic to a null exact functor because the inclusion functor $\S \hookrightarrow \T$ is the kernel of $Q$ in the sense of \defx \ref{defker} thanks to \prox \ref{propVerdfacts}. Moreover such natural isomorphism is exact by \lemx\ref{lemmanattozeroisexact}. The fact that the Verdier localization satisfies the 1-dimensional universal property of the 2-cokernel is clearly guaranteed by the universal property of the Verdier localization (see \defx \ref{defVerdloc} and \remx\ref{remtriangleinTriang}). We now prove the 2-dimensional universal property. We prove that $Q$ is a co-fully faithful arrow in $\Cat$, whence we will deduce that it is a co-fully faithful arrow in $\Triang$ using \prox\ref{propffinCatisffinTriang}. So let $H,H'\: \T/\S \to \U$ be functors and let $\lambda\:H \c Q \Rightarrow H' \c Q$ be a natural transformation. We need to show that there exists a unique natural transformation $\overline{\lambda}\: H \Rightarrow H'\:\T/\S\to \U$ such that $\overline{\lambda}\star Q=\lambda$. We construct $\overline{\lambda}$ using the explicit construction of $\T/\S$ as category of fractions. Given $X\in \T/\S$, which means $X\in \T$ since $\T$ and $\T/\S$ have the same objects, we define $\overline{\lambda}_X:= \lambda_X$. Naturality for $Q(f)$ follows from naturality of $\lambda$. For a denominator $v$, its naturality equation can be multiplied by $H(Q(v))^{-1}$ and $H'(Q(v))^{-1}$ to obtain naturality for $Q(v)^{-1}$. Composing these two equations proves naturality for every roof $Q(f)Q(v)^{-1}$. No component of $\lambda$ is required to be invertible. In particular, we use that every fraction \begin{cd}*[2][2] \& D \& \\[-2ex]
	X \&\& Y
	\arrow["v"', from=1-2, to=2-1]
	\arrow["f", from=1-2, to=2-3]\end{cd} coincides with the composite $Q(f)\c Q(v)^{-1}$ in $\T/\S$. It is then clear that $\overline{\lambda}$ is such that its whiskering with $Q$ is $\lambda$ and that it is the unique such natural transformation as the components are forced to be those of $\lambda$. So we have shown that Verdier localizations are 2-cokernels in $\Triang$.

   Consider now a 2-cokernel
   $$\T \aar{G} \U \aar{c_{G}} C(G)$$
   in $\Triang$. Thanks to \prox \ref{propkeriskerofitscoker}, $c_{G}$ is also the 2-cokernel of its 2-kernel $\ker(c_{G})$, which is a thick triangulated subcategory by \thex \ref{teorchar2ker1}. And by what we have just proved the Verdier localization  $\U/\ker(c_{G})$ is the 2-cokernel of the inclusion $\ker(c_G)\hookrightarrow \U$.  So all 2-cokernels in $\Triang$ are (equivalent to) Verdier localizations.

   It remains to prove that all 2-cokernels in $\Triang$ exist. By what we have proved above, we know that we can build them via Verdier localizations. So let $G\:\T\to \U$ be an exact functor between triangulated categories. We build the 2-cokernel of $G$ as follows. Consider the full subcategory $\I$ of $\U$ on the objects that are $G(X)$ for some $X\in \T$. That is, the full subcategory of $\U$ on the objects of the image of $G$. We construct the 2-cokernel of $G$ to be the Verdier localization of $\U$ by the smallest (replete) thick triangulated subcategory $\langle \I \rangle$ generated by $\I$. This smallest thick subcategory is the intersection of all full replete thick subcategories containing $\I$. We prove that the canonical projection $Q\:\U\to \U/\langle \I\rangle$ satisfies the universal property of the 2-cokernel of $G$. Surely, $Q\c G$ is isomorphic to a null exact functor, since $Q$ maps to a zero object all objects of $\langle \I \rangle \contain \I$ (see also \lemx\ref{lemmanattozeroisexact}). Let then $M\:\U\to \mathcal{V}$ be an exact functor such that $M\c G$ is isomorphic to a null exact functor. This means that the kernel of $M$ in the sense of \defx\ref{defker} contains all objects of $\I$. But by \thex\ref{teorchar2ker1}, the kernel of $M$ is a thick triangulated subcategory of $\U$, whence it contains $\langle \I \rangle$. By the universal property of the Verdier localization $\U/\langle \I \rangle$, we obtain a unique exact functor $\overline{M}\:\U/\langle \I \rangle\to \mathcal{V}$ such that $\overline{M}\c Q=M$. Whence $Q$ satisfies the 1-dimensional universal property of the 2-cokernel of $G$ (see also \remx\ref{remtriangleinTriang}). It remains to prove that $Q$ is a co-fully faithful arrow in $\Triang$. But this is guaranteed by the fact that $Q$ is a 2-cokernel, precisely of the inclusion $\langle \I\rangle\hookrightarrow \U$, by what we have proved above.
\end{proof}

\begin{remark}
    Theorem \ref{teorchar2coker1} thus captures the fundamental concept of Verdier localization, which is widely used in derived category theory, as a 2-categorical cokernel. As a consequence, our theory unifies Serre quotients and Verdier localizations and opens the way to the study of further connected fundamental concepts in other 2-categories.

    Moreover, \thex\ref{teorchar2coker1} provides a way to associate a Verdier localization to any exact functor between triangulated categories, rather than just to a triangulated category equipped with a triangulated subcategory. In this sense, 2-cokernels in $\Triang$ are a natural generalization of the fundamental notion of Verdier localization.
\end{remark}

Starting from the characterization of 2-kernels and 2-cokernels in the 2-category of Puppe exact categories (or of abelian categories) presented in \thex\ref{thm:puppe-characterization} and replacing the words appropriately following the remarks above, we obtain a characterization of 2-kernels and 2-cokernels in $\Triang$. We make this precise below.

\begin{remark}
    The theory developed in Section \ref{sectionabelian} for Puppe exact categories and abelian categories has a very interesting analogue in the theory of triangulated categories, as we now present.

    Recall that cones of morphisms play the role of cokernels, and, after rotation, of kernels. Thus $W_{\N}$ from Theorem~\ref{thm:puppe-localization} corresponds to $\mathcal M(\S)$ above. The operation $(-)^{\mathsf{SE}^{\infty}}$ of Lemma~\ref{LemD} corresponds to thick closure: translations, finite biproducts, existing retracts, and cones are adjoined, and the process is iterated. Equivalently, thick closure is the intersection of all full replete thick subcategories containing the given objects. This intersection description proves existence without imposing idempotent completeness. In the abelian case, the first operation is precisely Serre closure.
\end{remark}

\begin{lemma}\label{lemmamultsystF}
    Let $F\: \T \to \U$ be an exact functor between triangulated categories. The class of morphisms $$\mathcal{F}^F=\set{X\aar{f} Y \text{ in }\T}{F(f) \text{ is an isomorphism in }\U}$$
    is a multiplicative system that is compatible with the triangulation.

    In fact, we have $\mathcal{F}^F=\mathcal{M}(\ker(F)).$
\end{lemma}

\begin{proof}
    Let $X\aar{f} Y$ be a morphism in $\T$ such that $F(f)$ is an isomorphism in $\U$. Since $\T$ is triangulated there exists an exact triangle in $\T$ of the form
    $$X \aar{f} Y \aar{} Z \aar{} \Sigma X.$$
    Since $F$ is an exact functor, the triangle
     $$F(X) \aar{F(f)} F(Y) \aar{} F(Z) \aar{} \Sigma F(X)$$
     is exact in $\U$. Since $F(f)$ is an isomorphism by hypothesis, by \lemx \ref{lemmaisozero} we have that $F(Z)$ is $0$. And thus $Z\in \ker(F)$.

     Let now $X\aar{f} Y$ be a morphism in $\T$ that sits in an exact triangle
     $$X \aar{f} Y \aar{} Z \aar{} \Sigma X$$
     with $Z\in \ker(F)$. Then the triangle
     $$F(X) \aar{F(f)} F(Y) \aar{} 0 \aar{} \Sigma F(X)$$
     is exact in $\U$. And thus, again by \lemx\ref{lemmaisozero}, $F(f)$ is an isomorphism. So $f\in \mathcal{F}^F$. We have thus shown that $\mathcal{F}^F=\mathcal{M}(\ker(F))$ and so we conclude that $\mathcal{F}^F$ is a multiplicative system compatible with the triangulation.
    \end{proof}

\begin{theorem} \label{teorchar2ker2coker}
An exact functor $F\:\S \to \T$ between triangulated categories is a 2-kernel if and only if the following conditions hold
\begin{itemize}
\item [(FF)]$F$ is a fully faithful functor;
\item [(C1)] the replete essential image of $F$ is a thick triangulated subcategory of its codomain.
\end{itemize}
An exact functor $F\:\T \to \U$ between triangulated categories is a 2-cokernel if and only if the following conditions hold
\begin{itemize}
\item [(ES)]$F$ is essentially surjective on objects (as a functor);
\item [(C2)] for every $X,Y\in \T$ and every $F(X) \aar{g} F(Y)$ in $\U$, there exists a fraction \begin{cd}*[2][2] \& Z \& \\[-2ex]
	X \&\& Y
	\arrow["w"', from=1-2, to=2-1]
	\arrow["h", from=1-2, to=2-3]\end{cd} in $\T$ with respect to the multiplicative system $\mathcal{F}^F$ of \lemx\ref{lemmamultsystF} such that $F(h) \circ F(w)^{-1}=g$. Moreover, any two such fractions are equivalent.
\end{itemize}
\end{theorem}
\begin{proof}
    We prove the first part of the statement.

    Assume that $F\:\S\to \T$ is a 2-kernel in $\Triang$. We prove that it satisfies (FF) and (C1). By \corx\ref{cor2kerareff}, $F$ is a fully faithful functor. In fact, by \thex\ref{teorchar2ker1}, $F$ can be expressed, up to an equivalence $j$ in $\Triang$ (and an invertible 2-cell filling the triangle), as the full subcategory $i\:\ker(H)\hookrightarrow \T$ on all objects that are sent to a zero object by some exact functor $H\:\T\to \mathcal{V}$. Since $j$ is essentially surjective and $F\cong ij$, the replete essential image of $F$ is exactly the full subcategory $\ker(H)$ of $\T$. By \thex\ref{teorchar2ker1}, this subcategory is thick. Thus $F$ satisfies (C1).

    We now prove that if an exact functor $F\:\S\to \T$ satisfies (FF) and (C1) then it is a 2-kernel. By \prox\ref{propkeriskerofitscoker}, for $F$ to be a 2-kernel, it needs to be the 2-kernel of its 2-cokernel. By \thex\ref{teorchar2coker1}, the 2-cokernel of $F$ is given by the Verdier localization $\T/\langle \I \rangle$ of $\T$ by the smallest (replete) thick subcategory $\langle \I \rangle$ generated by the full subcategory $\I$ of $\T$ spanned by the objects in the image of $F$. Call $Q$ the quotient exact functor $\T\to \T/\langle \I \rangle$. By \prox\ref{propVerdfacts} and \thex\ref{teorchar2ker1}, the 2-kernel of $Q$ is given by the inclusion $\langle\I\rangle\hookrightarrow \T$. It suffices to prove that $F$ is equivalent in $\Triang$ to such inclusion. Since $Q$ is the 2-cokernel of $F$, we know that $Q\c F$ is isomorphic to a zero exact functor. So the universal property of the 2-kernel $\langle \I \rangle$ yields an exact functor $\overline{F}\:\S\to \langle \I \rangle$ and an invertible 2-cell in $\Triang$
    \begin{cd}[4]
        \S \& \T \\
	{\langle \I \rangle}
	\arrow["F", from=1-1, to=1-2]
	\arrow[xshift=-2ex, shift right=5, iso, "\theta"{inner sep=1ex},from=1-1, to=1-2]
	\arrow["{\overline{F}}"', dashed, from=1-1, to=2-1]
	\arrow[tail, from=2-1, to=1-2]
    \end{cd}
    By \prox\ref{propequivinTriang}, it remains to prove that $\overline{F}$ is an equivalence of categories. But, as a functor, $\overline{F}$ is isomorphic to the restriction of $F$ on the codomain to $\I$ followed by the inclusion of $\I$ into $\langle \I\rangle$. This is because $\langle \I\rangle\hookrightarrow \T$ is a fully faithful functor. By (C1), the inclusion $\I\hookrightarrow\langle\I\rangle$ is an equivalence of categories. And by (FF) the restriction of $F$ on the codomain to $\I$ is an equivalence of categories. So $\overline{F}$ is an equivalence of categories which is an exact functor, whence we conclude that $F$ is a 2-kernel.

    We prove the second part of the statement.

    Assume that $F\:\T \to \U$ is a 2-cokernel. We prove that it satisfies conditions $(ES)$ and (C2). By \thex\ref{teorchar2coker1} and \prox\ref{propkeriskerofitscoker}, $F$ is equivalent in $\Triang$ to the Verdier localization $Q\:\T\to \T/\ker(F)$ of $\T$ by the thick triangulated subcategory $\ker(F)\cont \T$. In fact, there exists a unique equivalence $\epsilon\:\T/\ker(F)\simeq \U$ in $\Triang$ such that $\epsilon\c Q=F$, by the universal property of the Verdier localization. Using \prox\ref{propequivinTriang}, since $Q$ is the identity on objects, $F$ satisfies (ES). We prove that $Q$ satisfies (C2). Then, since $\epsilon\c Q=F$, we will conclude that $F$ satisfies (C2) as well. So consider $X,Y\in \T$ and morphism $g\:Q(X)\to Q(Y)$ in $\T/\ker(F)$, i.e.\ a fraction \begin{cd}*[2][2] \& Z \& \\[-2ex]
	X \&\& Y
	\arrow["w"', from=1-2, to=2-1]
	\arrow["h", from=1-2, to=2-3]\end{cd} with respect to $\mathcal{M}(\ker(F))$. Note that $\ker(F)=\ker(Q)$, and so $\mathcal{M}(\ker(F))=\mathcal{M}(\ker(Q))=\mathcal{F}^Q$, by \lemx\ref{lemmamultsystF}. So it suffices to prove that $Q(h)\c Q(w)^{-1}=g$. But $Q(h)$ is the span \begin{cd}*[2][2] \& Z \& \\[-2ex]
	Z \&\& Y
	\arrow[equal, from=1-2, to=2-1]
	\arrow["h", from=1-2, to=2-3]\end{cd} and $Q(w)^{-1}$ is the span \begin{cd}*[2][2] \& Z \& \\[-2ex]
	X \&\& Z
	\arrow["w"', from=1-2, to=2-1]
	\arrow[equal, from=1-2, to=2-3]\end{cd}. Since the composition of two fractions does not depend on the choice of the Ore square, we can just use the square of four identities of $Z$ as Ore square. Then $Q(h)\c Q(w)^{-1}$ is precisely the starting fraction $g$. Finally, let \begin{cd}*[2][2] \& Z' \& \\[-2ex]
	X \&\& Y
	\arrow["w'"', from=1-2, to=2-1]
	\arrow["h'", from=1-2, to=2-3]\end{cd} be another fraction with respect to $\mathcal{F}^Q$ such that $Q(h')\c Q(w')^{-1}=g$. By the same argument we showed above, $Q(h')\c Q(w')^{-1}$ coincides with the starting fraction exhibited by $w'$ and $h'$, seen as a morphism in $\T/\ker(F)$. So such fraction is equal to $g$ as morphisms in $\T/\ker(F)$, which precisely means that it is equivalent to the fraction $g$ as fractions with respect to $\mathcal{M}(\ker(F))=\mathcal{F}^Q$. We conclude that $Q$ satisfies (C2), whence $F$ satisfies (C2) as well.

    We now prove that if an exact functor $F\:\T\to \U$ satisfies (ES) and (C2) then it is a 2-cokernel. By \prox\ref{propkeriskerofitscoker}, every 2-cokernel is the 2-cokernel of its 2-kernel. So for $F$ to be a 2-cokernel, it needs to be the 2-cokernel of its 2-kernel $\ker(F)\hookrightarrow \T$. Moreover, by \thex\ref{teorchar2coker1}, the Verdier localization $\T/\ker(F)$ is a representative of the 2-cokernel of the inclusion $\ker(F)\hookrightarrow \T$. So it suffices to prove that $F$ is equivalent in $\Triang$ to the quotient exact functor $Q:\T\to \T/\ker(F)$. Notice that $F$ sends every object of $\ker(F)$ to a zero object. So, by the universal property of the Verdier localization $\T/\ker(F)$, there exists a unique exact functor $\phi\:\T/\ker(F)\to \U$ such that $F=\phi\c Q$ (and the identity 2-cell is a 2-cell in $\Triang$). By \prox\ref{propequivinTriang}, in order to show that $\phi$ is an equivalence in $\Triang$, it just remains to show that $\phi$ is an equivalence in $\Cat$. Explicitly, it is well known that $\phi$ acts as follows. Given $X\in \T/\ker(F)$, $\phi(X)=F(X)$. Then given a morphism $\ell\:X\to Y$ in $\T/\ker(F)$, i.e.\ a fraction \begin{cd}*[2][2] \& Z \& \\[-2ex]
	X \&\& Y
	\arrow["w"', from=1-2, to=2-1]
	\arrow["h", from=1-2, to=2-3]\end{cd} in $\T$ with respect to the multiplicative system $\mathcal{M}(\ker(F))=\mathcal{F}^F$, we have that $\phi(\ell)=F(h)\c F(w)^{-1}$. We prove that $\phi$ is essentially surjective on objects and fully faithful. Given $A\in \U$, since $F$ satisfies (ES), there exists $X_A\in \T$ such that $F(X_A)\iso A$. But then $X_A\in \T/\ker(F)$ and $\phi(X_A)=F(X_A)\iso A$. So $\phi$ is essentially surjective on objects. Consider now $X,Y\in \T/\ker(F)$ and a morphism $g\:\phi(X)=F(X)\to F(Y)=\phi(Y)$ in $\U$. Since $F$ satisfies (C2), there exists a fraction \begin{cd}*[2][2] \& Z \& \\[-2ex]
	X \&\& Y
	\arrow["w"', from=1-2, to=2-1]
	\arrow["h", from=1-2, to=2-3]\end{cd} in $\T$ with respect to the multiplicative system $\mathcal{F}^F=\mathcal{M}(\ker(F))$ such that $F(h) \circ F(w)^{-1}=g$. Moreover, any two such fractions are equivalent. But this precisely means that there exists a unique morphism $\ell\:X\to Y$ in $\T/\ker(F)$ such that $\phi(\ell)=g$. So $\phi$ is fully faithful. We conclude that $\phi$ is an equivalence of categories which is also an exact functor. Whence $F$ is equivalent to $Q$ in $\Triang$ and is thus a 2-cokernel.
\end{proof}

\begin{remark}
    The characterization above gives concrete criteria for $2$-kernels and $2$-cokernels in $\Triang$, and hence for thick triangulated subcategories and Verdier localizations.
\end{remark}

\begin{theorem}\label{teor2kerclosed}
    2-kernels in $\Triang$ are closed under composition. Equivalently, given a chain of triangulated categories $\S \subseteq \T \subseteq \U$, where $\S$ is a thick triangulated subcategory of $\T$ and $\T$ is a thick triangulated subcategory of $\U$, we have that $\S$ is a thick triangulated subcategory of $\U$.
\end{theorem}

\begin{proof}
The equivalence of the two statements is guaranteed by \thex \ref{teorchar2ker1}. We prove the second one. Let $i\:\S \hookrightarrow \T$ and $j\: \T \hookrightarrow \U$ be the inclusions and let $X\in \S$. Then
$$\Sigma^n_{\U}j(i(X))=j(\Sigma^n_{\T}i(X))\in j(i(\S))$$
because $j$ preserves the translation functor strictly and $\S$ is a triangulated subcategory of $\T$. Moreover, given an exact triangle in $\U$
$$X \aar{u} Y \aar{v} Z \aar{w} \Sigma_{\U} X$$
with $X,Y\in \S$, this corresponds to the exact triangle in $\U$
$$j(i(X)) \aar{u} j(i(Y)) \aar{v} Z \aar{w} \Sigma_{\U} j(i(X))$$
and since $\T$ is a triangulated subcategory of $\U$ we conclude that $Z\in \T$. So we have an exact triangle in $\T$
$$i(X) \aar{u} i(Y) \aar{v} Z \aar{w} \Sigma_{\T} i(X).$$
But $\S$ is a triangulated subcategory of $\T$ and thus $Z\in \S$. So we have shown that $\S$ is a triangulated subcategory of $\U$. Consider now a chain of morphisms in $\U$
$$X \aar{\pi} Y \aar{\iota} X$$
such that $\pi \c \iota= \id{Y}$ and assume $X\in \S$. Since $X\in \T$ and $\T$ is a thick triangulated subcategory of $\U$, we conclude that $Y\in \T$. But then the starting diagram is a diagram in $\T$ and $\S$ is a thick triangulated subcategory of $\T$, so we conclude that $Y\in \S$. This shows that $\S$ is a thick triangulated subcategory of $\U$.
\end{proof}

\begin{theorem}\label{teor2cokerclosed}
    2-cokernels in $\Triang$ are closed under composition. Equivalently, given two Verdier localizations $Q\: \T \to \T/\S$ and $Q'\:\T /\S \to (\T/\S)/\U$, we have that $Q' \c Q\:\T \to (\T/\S)/\U$ is a Verdier localization of $\T$.
\end{theorem}

\begin{proof}
    The equivalence of the two statements is guaranteed by \thex \ref{teorchar2coker1}. We prove the first one. Let $F\:\T \to \U$ and $G\:\U \to \V$ be 2-cokernels. Thanks to \thex \ref{teorchar2ker2coker}, it suffices to prove that the composite functor $G \c F$ is essentially surjective on objects and satisfies condition (C2). The fact that $G\c F$ is essentially surjective on objects immediately follows from the fact that both $F$ and $G$ are. Consider then $X,Y\in \T$ and a morphism $g\: G(F(X)) \to G(F(Y))$ in $\V$. Since $G$ is a 2-cokernel, there exists a fraction in $\U$ \begin{cd}*[2][2] \& R \& \\[-2ex]
	F(X)\&\& F(Y)
	\arrow["w"', from=1-2, to=2-1]
	\arrow["f", from=1-2, to=2-3]\end{cd} with respect to $\mathcal{F}^G$ such that $G(f) \c G(w)^{-1}=g$. Furthermore, since $F$ is a 2-cokernel and thus satisfies (ES) and (C2), there exist $R'\in\T$ together with an isomorphism $j\: F(R') \aiso{} R$ and fractions in $\T$ \begin{cd}*[2][2] \& P^{w} \& \\[-2ex]
	R'\&\& X
	\arrow["u"', from=1-2, to=2-1]
	\arrow["h", from=1-2, to=2-3]\end{cd} and \begin{cd}*[2][2] \& P^{f} \& \\[-2ex]
	R'\&\& Y
	\arrow["v"', from=1-2, to=2-1]
	\arrow["k", from=1-2, to=2-3]\end{cd} with respect to $\F^{F}$ such that
    $$F(h) \c F(u)^{-1}= w\c j \quad \text{ and } \quad F(k) \c F(v)^{-1}= f\c j$$
    Notice that $h\in \F^{G\c F}$ since $G(F(h))=G(w) \c G(j) \c G(F(u))$. Moreover, also $u,v\in \mathcal{F}^{G \c F}$. So the two fractions above can be thought as fractions with respect to $\F^{G\c F}$, and the first one is an invertible fraction. We can thus compose the second fraction with the inverse of the first: we obtain a fraction
    \begin{cd}[4][4]
    \&\& Q \&\& \\
	\& {P^{w}} \&\& {P^{f}} \\
	X \&\& {R'} \&\& Y
	\arrow["{\tilde{v}}"', from=1-3, to=2-2]
	\arrow["{\tilde{u}}", from=1-3, to=2-4]
	\arrow["h"', from=2-2, to=3-1]
	\arrow["u", from=2-2, to=3-3]
	\arrow["v"', from=2-4, to=3-3]
	\arrow["k", from=2-4, to=3-5]
    \end{cd}
    where $\tilde{u}$ and $\tilde{v}$ are given by the Ore condition, such that\vsep[-1]
    \begin{center}
		\linesep{1.3}
		\begin{tabular}{LL}
			G(F(k))\c G(F(\tilde{u}))\c G(F(\tilde{v}))^{-1}\c G(F(h))^{-1}=\\
            =G(F(k))\c G(F(v))^{-1}\c G(F(u))\c G(F(h))^{-1}=g
		\end{tabular}\vsep[-1]
	\end{center}
    It remains to prove the uniqueness of such a fraction. So consider two fractions \begin{cd}*[2][2] \& D \& \\[-2ex]
	X\&\& Y
	\arrow["l"', from=1-2, to=2-1]
	\arrow["r", from=1-2, to=2-3]\end{cd} and \begin{cd}*[2][2] \& D' \& \\[-2ex]
	X\&\& Y
	\arrow["l'"', from=1-2, to=2-1]
	\arrow["r'", from=1-2, to=2-3]\end{cd} with respect to $\F^{G\c F}$ such that
    $$G(F(r))\c G(F(l))^{-1}=g=G(F(r'))\c G(F(l'))^{-1}$$
    The functor $F$ maps these two fractions to equivalent fractions with respect to $\F^G$, since $G$ is a 2-cokernel and thus satisfies (C2) by \thex\ref{teorchar2ker2coker}. So there exists a fraction \begin{cd}*[2][2] \& P \& \\[-2ex]
	F(X)\&\& F(Y)
	\arrow["\overline{l}"', from=1-2, to=2-1]
	\arrow["\overline{r}", from=1-2, to=2-3]\end{cd} with respect to $\F^G$ and morphisms $a\:P\to F(D)$ and $b\:P\to F(D')$ such that the following diagram commutes:
    \begin{cd}[4.5]
        \& {F(D)} \& \\
	{F(X)} \& P \& {F(Y)} \\
	\& {F(D')}
	\arrow["{F(l)}"', from=1-2, to=2-1]
	\arrow["{F(r)}", from=1-2, to=2-3]
	\arrow["a"', from=2-2, to=1-2]
	\arrow["{\overline{l}}", from=2-2, to=2-1]
	\arrow["{\overline{r}}"', from=2-2, to=2-3]
	\arrow["b", from=2-2, to=3-2]
	\arrow["{F(l')}", from=3-2, to=2-1]
	\arrow["{F(r')}"', from=3-2, to=2-3]
    \end{cd}
    Then since $F$ is a 2-cokernel and thus satisfies (ES) and (C2), there exists $P'\in \T$ together with an isomorphism $j'\:F(P')\iso P$ and fractions in $\T$
    \begin{cd}*[2][2] \& N^{a} \& \\[-2ex]
	P'\&\& D
	\arrow["s^a"', from=1-2, to=2-1]
	\arrow["t^a", from=1-2, to=2-3]\end{cd} and \begin{cd}*[2][2] \& N^{b} \& \\[-2ex]
	P'\&\& D'
	\arrow["s^b"', from=1-2, to=2-1]
	\arrow["t^b", from=1-2, to=2-3]\end{cd} with respect to $\F^{F}$ such that $F(t^a) \c F(s^a)^{-1}= a\c j'$ and $F(t^b) \c F(s^b)^{-1}= b\c j'$. Notice that $t^a\in \F^{G\c F}$ because $G(a)$ is an isomorphism, since both $G(\overline{l})$ and $G(F(l))$ are isomorphisms. By the commutative diagram drawn above, both the composites of fractions
    \begin{eqD*}
        \begin{cd}*[3][3]
    \& {N^a} \&\& D \& \\
	{P'} \&\& D \&\& X
	\arrow["{s^a}"', from=1-2, to=2-1]
	\arrow["{t^a}", from=1-2, to=2-3]
	\arrow[equal, from=1-4, to=2-3]
	\arrow["l", from=1-4, to=2-5]
    \end{cd}
    \quad \text{ and }\quad
    \begin{cd}*[3][3]
        {P'} \&\& {D'} \&\& X \\
	\& {N^b} \&\& {D'}
	\arrow["{s^b}", from=2-2, to=1-1]
	\arrow["{t^b}"', from=2-2, to=1-3]
	\arrow[equal, from=2-4, to=1-3]
	\arrow["{l'}"', from=2-4, to=1-5]
    \end{cd}
    \end{eqD*}
    capture $\overline{l}\c j'$ when applying $F$. Since $F$ is a 2-cokernel, these two composites of fractions must be equivalent with respect to $\F^F$. So there exists a fraction \begin{cd}*[2][2] \& \Phi \& \\[-2ex]
	P'\&\& X
	\arrow["k"', from=1-2, to=2-1]
	\arrow["h", from=1-2, to=2-3]\end{cd} with respect to $\F^F$ and morphisms $\phi^a\:\Phi\to N^a$ and $\phi^b\:\Phi\to N^b$ such that the following diagram commutes:
    \begin{cd}[1.5][3]
    \&\& {N^a} \&\& \\
	\& {N^a} \&\& D \\
	{P'} \&\& \Phi \&\& X \\
	\& {N^b} \&\& {D'} \\
	\&\& {N^b}
	\arrow[equal, from=1-3, to=2-2]
	\arrow["{t^a}", from=1-3, to=2-4]
	\arrow["{s^a}"', from=2-2, to=3-1]
	\arrow["l", from=2-4, to=3-5]
	\arrow["{\phi^a}"', from=3-3, to=1-3]
	\arrow["k"', from=3-3, to=3-1]
	\arrow["h", from=3-3, to=3-5]
	\arrow["{\phi^b}", from=3-3, to=5-3]
	\arrow["{s^b}", from=4-2, to=3-1]
	\arrow["{l'}"', from=4-4, to=3-5]
	\arrow[equal, from=5-3, to=4-2]
	\arrow["{t^b}"', from=5-3, to=4-4]
    \end{cd}
    Notice that $\phi^a\in \F^F$, since $F(k)$ and $F(s^a)$ are both isomorphisms. Intuitively, $\Phi$ is allowing us to dominate both $l$ and $l'$ at once. But it does not necessarily play well with $r$ and $r'$. Similarly, using $r$ and $r'$ in place of $l$ and $l'$, we obtain a fraction \begin{cd}*[2][2] \& \Psi \& \\[-2ex]
	P'\&\& Y
	\arrow["k'"', from=1-2, to=2-1]
	\arrow["h'", from=1-2, to=2-3]\end{cd} with respect to $\F^F$ and morphisms $\psi^a\:\Psi\to N^a$ and $\psi^b\:\Psi\to N^b$ such that the following diagram commutes:
    \begin{cd}[1.5][3]
        \&\& {N^a} \&\& \\
	\& {N^a} \&\& D \\
	{P'} \&\& \Psi \&\& Y \\
	\& {N^b} \&\& {D'} \\
	\&\& {N^b}
	\arrow[equal, from=1-3, to=2-2]
	\arrow["{t^a}", from=1-3, to=2-4]
	\arrow["{s^a}"', from=2-2, to=3-1]
	\arrow["r", from=2-4, to=3-5]
	\arrow["{\psi^a}"', from=3-3, to=1-3]
	\arrow["{k'}"', from=3-3, to=3-1]
	\arrow["{h'}", from=3-3, to=3-5]
	\arrow["{\psi^b}", from=3-3, to=5-3]
	\arrow["{s^b}", from=4-2, to=3-1]
	\arrow["{r'}"', from=4-4, to=3-5]
	\arrow[equal, from=5-3, to=4-2]
	\arrow["{t^b}"', from=5-3, to=4-4]
    \end{cd}
    And $\psi^a\in \F^F$. We can then apply the Ore condition with respect to the multiplicative system $\F^F$ to obtain a square
    \begin{cd}
    \Lambda \& \Phi \\
	\Psi \& {N^a}
	\arrow["{\widetilde{\psi^a}}", from=1-1, to=1-2]
	\arrow["{\widetilde{\phi^a}}"', from=1-1, to=2-1]
	\arrow["{\phi^a}", from=1-2, to=2-2]
	\arrow["{\psi^a}"', from=2-1, to=2-2]
    \end{cd}
    with $\widetilde{\psi^a}\in \F^F$. Now, notice that $s^b\in \F^F$ coequalizes $\phi^b\c \widetilde{\psi^a}$ and $\psi^b\c \widetilde{\phi^a}$. Since $\F^F$ is a multiplicative system, there exists $\Xi\in \T$ and $\xi\:\Xi\to \Lambda$ in $\F^F$ such that
    \begin{cd}[4]
        \& \Lambda \& \Phi \& \\
	\Xi \&\&\& {N^b} \\
	\& \Lambda \& \Psi
	\arrow["{\widetilde{\psi^a}}", from=1-2, to=1-3]
	\arrow["{\phi^b}", from=1-3, to=2-4]
	\arrow["\xi", from=2-1, to=1-2]
	\arrow["\xi"', from=2-1, to=3-2]
	\arrow["{\widetilde{\phi^a}}"', from=3-2, to=3-3]
	\arrow["{\psi^b}"', from=3-3, to=2-4]
    \end{cd}
    Therefore, $\Xi$ provides a fraction with respect to $\F^{G\c F}$ that dominates both the starting fractions
    \begin{cd}*[2][2] \& D \& \\[-2ex]
	X\&\& Y
	\arrow["l"', from=1-2, to=2-1]
	\arrow["r", from=1-2, to=2-3]\end{cd} and \begin{cd}*[2][2] \& D' \& \\[-2ex]
	X\&\& Y
	\arrow["l'"', from=1-2, to=2-1]
	\arrow["r'", from=1-2, to=2-3]\end{cd}
    Indeed, the following diagram commutes:
    \begin{cd}[4]
        \&\&\&\& X \&\&\&\& \\
	{D'} \& {N^b} \& \Phi \& \Lambda \& \Xi \& \Lambda \& \Phi \& {N^a} \& D \\
	\&\&\&\& Y
	\arrow["{l'}",,bend left=10, from=2-1, to=1-5]
	\arrow["{r'}"',bend right=10, from=2-1, to=3-5]
	\arrow["{t^b}"', from=2-2, to=2-1]
	\arrow["{\phi^b}"', from=2-3, to=2-2]
	\arrow["{\widetilde{\psi^a}}"', from=2-4, to=2-3]
	\arrow["\xi"', from=2-5, to=2-4]
	\arrow["\xi", from=2-5, to=2-6]
	\arrow["{\widetilde{\psi^a}}", from=2-6, to=2-7]
	\arrow["{\phi^a}", from=2-7, to=2-8]
	\arrow["{t^a}", from=2-8, to=2-9]
	\arrow["l"',bend right=10, from=2-9, to=1-5]
	\arrow["r",bend left=10, from=2-9, to=3-5]
    \end{cd}
    And the composite $\Xi\aar{\xi}\Lambda\aar{\widetilde{\psi^a}}\Phi\aar{\phi^a}N^a\aar{t^a}D\aar{l}X$ is in $\F^{G\c F}$ since $\xi,\widetilde{\psi^a},\phi^a\in \F^F$ and $t^a,l\in \F^{G\c F}$.
\end{proof}

We prove that more is true. The 2-category $\Triang$ actually satisfies a 2-dimensional generalization of the fundamental concept of homological category, in the sense of Grandis \cite{Grandis13}.

We use Definition~\ref{def:homological}; in particular its conditional axiom \textup{(H3)}, rather than unconditional interchange, is required. As explained in Section~\ref{sec:framework}, $2$-Puppe exactness implies $2$-homologicity.

We will need a characterization of those exact functors between triangulated categories that factorize (up to isomorphism) as a 2-cokernel followed by a 2-kernel.

\begin{theorem}\label{theorcharactVerdierfunctors}
   An exact functor $F\:\T \to \U$ between triangulated categories factorizes (up to isomorphism) as a 2-cokernel followed by a 2-kernel if and only if the following conditions hold:
\begin{itemize}
\item [(C1)] the replete essential image of $F$ is a thick triangulated subcategory of its codomain;
\item [(C2)] for every $X,Y\in \T$ and every $F(X) \aar{g} F(Y)$ in $\U$, there exists a fraction \begin{cd}*[2][2] \& Z \& \\[-2ex]
	X \&\& Y
	\arrow["w"', from=1-2, to=2-1]
	\arrow["h", from=1-2, to=2-3]\end{cd} in $\T$ with respect to the multiplicative system $\mathcal{F}^F$ of \lemx\ref{lemmamultsystF} such that $F(h) \circ F(w)^{-1}=g$. Moreover, any two such fractions are equivalent.
\end{itemize}
\end{theorem}

\begin{proof}
    Let $F\: \T \to \U$ be an exact functor such that there exists an invertible 2-cell
    \begin{cd}[4]
        \T \&\& \U \\
	\& \V
	\arrow["F", from=1-1, to=1-3]
	\arrow[""{name=0, anchor=center, inner sep=0}, "G"', from=1-1, to=2-2]
	\arrow[""{name=1, anchor=center, inner sep=0}, "H"', from=2-2, to=1-3]
	\arrow["\tau"'{inner sep=1ex},iso, shift left=2ex,from=0, to=1]
    \end{cd}
    in $\Triang$ where $G$ is a 2-cokernel and $H$ is a 2-kernel. Let $X$ be an object in the smallest thick triangulated subcategory generated by the image of $H\c G$. Notice that $X$ is then also in the smallest thick triangulated subcategory generated by the image of $H$. Since $H$ satisfies condition (C1) by \thex \ref{teorchar2ker2coker}, there exists $X'\in \V$ such that $H(X')$ is isomorphic to $X$. Moreover, since $G$ is essentially surjective on objects by \thex \ref{teorchar2ker2coker}, there exists $X''\in \T$ such that $X'$ is isomorphic to $G(X'')$. But then $X$ is isomorphic to $H(G(X''))$. So $H \c G$ satisfies condition (C1) and hence $F$ satisfies it as well. Indeed, the smallest (replete) thick triangulated subcategory containing the image of $F$ is the same as the smallest (replete) thick triangulated subcategory containing the image of $H\c G$, by repleteness. We now prove that $F$ satisfies (C2). Let $X,Y\in \T$ and let $g\: H(G(X)) \to H(G(Y))$ be a morphism in $\U$. Since $H$ is fully faithful, there exists a morphism $g'\: G(X) \to G(Y)$ such that $H(g')=g$. Since $G$ satisfies the condition (C2), there exists a fraction
\begin{cd}*[2][2] \& D \& \\[-2ex]
	X \&\& Y
	\arrow["w"', from=1-2, to=2-1]
	\arrow["f", from=1-2, to=2-3]\end{cd}
    with $w\in \F^{G}$ such that $G(f) \c G(w)^{-1}=g'$. But then $w\in \F^{H\c G}$ and $H(G(f)) \c H(G(w))^{-1}=g$. Consider now two fractions \begin{cd}*[2][2] \& D \& \\[-2ex]
	X \&\& Y
	\arrow["w"', from=1-2, to=2-1]
	\arrow["f", from=1-2, to=2-3]\end{cd} and \begin{cd}*[2][2] \& D' \& \\[-2ex]
	X \&\& Y
	\arrow["w'"', from=1-2, to=2-1]
	\arrow["f'", from=1-2, to=2-3]\end{cd}
    such that
    $$H(G(f)) \c H(G(w))^{-1}=g= H(G(f')) \c H(G(w'))^{-1}.$$
    Since $H$ is fully faithful, we obtain $w\in \F^{G}$ and $w'\in \F^{G}$ and $G(f) \c G(w)^{-1}=G(f') \c G(w')^{-1}$. Since $G$ satisfies condition (C2), we conclude that the two fractions are equivalent for $\F^G$ and thus they are equivalent for $\F^{H \c G}$. So $H\c G$ satisfies condition (C2). And it is straightforward to see that then $F$ satisfies (C2) as well, using the invertible natural transformation $\tau$. We have thus proved that $F$ satisfies both (C1) and (C2).

    Let now $F\: \T \to \U$ be an exact functor that satisfies (C1) and (C2). Consider $F(\T)$ the full subcategory of $\U$ spanned by the objects in the image of $F$, and consider $\langle F(\T) \rangle$ the smallest (replete) thick triangulated subcategory generated by $F(\T)$. By condition (C1), the inclusion $F(\T)\hookrightarrow \langle F(\T) \rangle$ is an equivalence of categories. Notice that $F$ factorizes as a functor as the composite
    $$\T \aar{\widetilde{F}} F(\T) \xhookrightarrow{I}\langle F(\T)\rangle \xhookrightarrow{J} \U$$
    where $\widetilde{F}$ is the restriction of $F$ on the codomain. Furthermore, $\langle F(\T)\rangle$ is a triangulated category, as it is a thick triangulated subcategory of $\U$, and it is straightforward to see that $F$ actually factorizes in $\Triang$ as
    $$\T \aar{\overline{F}} \langle F(\T)\rangle \xhookrightarrow{J} \U$$
    where $\overline{F}=I\c \widetilde{F}$. Of course, $J$ is a 2-kernel in $\Triang$, by \thex\ref{teorchar2ker1}. It then suffices to prove that $\T \aar{\overline{F}} \langle F(\T) \rangle$ is a 2-cokernel. Surely, $\overline{F}$ is essentially surjective on objects because it is the composite of an essentially surjective on object functor followed by an equivalence. So, by \thex \ref{teorchar2ker2coker}, it just remains to prove that $\overline{F}$ satisfies condition (C2). Let then $X,Y\in \T$ and consider $g\: \overline{F}(X)=F(X) \to F(Y)=\overline{F}(Y)$ in $\langle F(\T) \rangle\subseteq \U$. Since $F$ satisfies (C2), there exists a unique fraction  \begin{cd}*[2][2] \& D \& \\[-2ex]
	X \&\& Y
	\arrow["w"', from=1-2, to=2-1]
	\arrow["f", from=1-2, to=2-3]\end{cd} in $\T$ with respect to $\F^{F}$ such that $F(f)\c F(w)^{-1}=g$. But $\F^{F}=\F^{\overline{F}}$, since $J$ is conservative (being fully faithful). And $\overline{F}(f)\c \overline{F}(w)^{-1}=F(f)\c F(w)^{-1}=g$. Moreover, the fraction above is the unique such with respect to $\mathcal{F}^{\overline{F}}$, since $\F^{F}=\F^{\overline{F}}$. This proves that $\overline{F}$ is a 2-cokernel and thus that $F$ factorizes as a 2-cokernel followed by a 2-kernel.
    \end{proof}

The proof of \thex \ref{theorcharactVerdierfunctors} actually shows the following interesting result.

\begin{corollary}
Let $G\:\T \to \V$ and $H\: \V \to \U$ be exact functors between triangulated categories and consider the composite $H\c G$. The following facts hold:
\begin{enumerate}
    \item if $G$ is essentially surjective on objects and $H$ satisfies (C1), then $H\c G$ satisfies (C1);
    \item if $G$ satisfies (C2) and $H$ is fully faithful, then $H\c G$ satisfies (C2).
\end{enumerate}
\end{corollary}
\begin{proof}
For (1), essential surjectivity of $G$ identifies the replete essential images of $HG$ and $H$. For (2), full faithfulness of $H$ identifies maps $HG(X)\to HG(Y)$ with maps $G(X)\to G(Y)$ and reflects isomorphisms. Hence the denominator classes for $G$ and $HG$ agree and the existence and equality of their representing fractions agree. This proves both claims.
\end{proof}

\begin{remark}
    Once again, \thex \ref{theorcharactVerdierfunctors} is surprisingly close to what we presented for abelian (and Puppe exact) categories (see \thex\ref{thm:puppe-characterization}).
\end{remark}

This motivates the following definition.

\begin{definition}
    We call \dfn{Verdier functor} an exact functor between triangulated categories that factorizes up to isomorphism as a 2-cokernel followed by a 2-kernel.
\end{definition}

\begin{remark}
    \thex\ref{theorcharactVerdierfunctors} yields a characterization of Verdier functors. Notice that, moreover, we have actually proved that every Verdier functor can be written as exactly equal to a 2-cokernel followed by a 2-kernel.
\end{remark}

\begin{theorem}\label{thm:tri-direct}
    The 2-category $\Triang$ is 2-homological.
\end{theorem}

\begin{proof}
    The existence of all 2-kernels and all 2-cokernels in $\Triang$ is guaranteed by \thex \ref{teorchar2ker1} and \thex \ref{teorchar2coker1} respectively. Furthermore, the closure of 2-kernels and 2-cokernels under composition is guaranteed by \thex \ref{teor2kerclosed} and \thex \ref{teor2cokerclosed} respectively. So it remains to prove that $\Triang$ satisfies $(H3)$. Let $F\:\T \to \U$ be an exact functor between triangulated categories such that $F$ is isomorphic in $\Triang$ to the composite of a 2-kernel $G$ followed by a 2-cokernel $H$ and suppose that the 2-kernel of $H$ factors through $G$ (up to isomorphism) as shown in the following diagram.
        \begin{cd}
            \S \&\& \\[-1ex]
	\T \arrow[rr, bend right=40, "F"'] \& \V \& \U
	\arrow[dashed, from=1-1, to=2-1]
	\arrow["{\ker(H)}"{inner sep=0.3ex}, from=1-1, to=2-2]
	\arrow["G", tail, from=2-1, to=2-2]
	\arrow[shift left=5,xshift=-2ex,iso, from=2-1, to=2-2]
	\arrow[shift right=5,iso,from=2-1, to=2-3]
	\arrow["H", two heads, from=2-2, to=2-3]
        \end{cd}
    We need to prove that $H\c G$ can be factorized up to isomorphism as a 2-cokernel followed by a 2-kernel. Thanks to \thex \ref{theorcharactVerdierfunctors}, it suffices to prove that $H\c G$ satisfies (C1) and (C2). Let $X,Y\in \T$ and let $g\:H(G(X)) \to H(G(Y))$. Since $H$ satisfies (C2), there exists a fraction in $\V$ \begin{cd}*[2][2] \& D \& \\[-2ex]
	G(X) \&\& G(Y)
	\arrow["w"', from=1-2, to=2-1]
	\arrow["f", from=1-2, to=2-3]\end{cd} with respect to $\F^{H}$ such that $H(f)\c H(w)^{-1}=g$. By \lemx \ref{lemmamultsystF}, $\F^H=\M(\S)$ and so there exists an exact triangle in $\V$
    $$D \aar{w} G(X) \aar{v} Z \aar{} \Sigma D$$
    with $Z\in \S$. Thanks to the fact that $\ker(H)$ factors through $G$ up to isomorphism, there exists an object $Z'\in \T$ together with an isomorphism $j\: Z \aiso{} G(Z')$. As a consequence, by definition of triangulated subcategory, $D$ belongs to the smallest thick triangulated subcategory of $\V$ containing the image of $G$ and so, since $G$ satisfies (C1), there exist an object $D'$ together with an isomorphism $i\:D\iso G(D')$ for some $D'\in \T$. Then, since $G$ is fully faithful, there exist morphisms $D' \aar{w'} X$ and $D'\aar{f'} Y$ such that $w\c i^{-1}=G(w')$ and $f \c i^{-1}=G(f')$. Then $H(G(w'))$ is an isomorphism, since $w\in \F^H$, and so
    \begin{cd}*[2][2] \& D' \& \\[-2ex]
	X \&\& Y
	\arrow["w'"', from=1-2, to=2-1]
	\arrow["f'", from=1-2, to=2-3]\end{cd} is a fraction with respect to $\F^{H\c G}$ such that
    $$H(G(f')) \c H(G(w'))^{-1}=H(f) \c H(w)^{-1}=g.$$
    Consider now two fractions in $\T$ \begin{cd}*[2][2] \& R \& \\[-2ex]
	X \&\& Y
	\arrow["w"', from=1-2, to=2-1]
	\arrow["f", from=1-2, to=2-3]\end{cd} and \begin{cd}*[2][2] \& R' \& \\[-2ex]
	X \&\& Y
	\arrow["w'"', from=1-2, to=2-1]
	\arrow["f'", from=1-2, to=2-3]\end{cd} with respect to $\F^{H\c G}$, such that
    $$H(G(f)) \c H(G(w))^{-1}=g=H(G(f')) \c H(G(w'))^{-1}.$$
    Since $H$ satisfies (C2), applying $G$ to the previous two fractions yields the two equivalent fractions with respect to $\F^H$ shown in the following diagram
    \begin{cd}[4][5]
        \& {G(R)} \& \\
	{G(X)} \& K \& {G(Y)} \\
	\& {G(R')}
	\arrow["{G(w)}"', from=1-2, to=2-1]
	\arrow["{G(f)}", from=1-2, to=2-3]
	\arrow["a"', from=2-2, to=1-2]
	\arrow[from=2-2, to=2-1]
	\arrow[from=2-2, to=2-3]
	\arrow["b", from=2-2, to=3-2]
	\arrow["{G(w')}", from=3-2, to=2-1]
	\arrow["{G(f')}"', from=3-2, to=2-3]
    \end{cd}
    where $G(w),G(w)\c a,G(w')\in \F^H$. Using that $\ker(H)$ factors through $G$ and that $G(w)\c a\in \F^H$, we  conclude that there exist $K'\in \T$ and an isomorphism $i\: K \aiso{} G(K')$ (with a similar argument to the one used earlier in the proof). Since $G$ is fully faithful by \thex \ref{teorchar2ker2coker}, we obtain a diagram
    \begin{cd}[4]
        \& R \& \\
	X \& K' \& Y \\
	\& R'
	\arrow["w"', from=1-2, to=2-1]
	\arrow["f", from=1-2, to=2-3]
	\arrow["{a'}"', from=2-2, to=1-2]
	\arrow[from=2-2, to=2-1]
	\arrow[from=2-2, to=2-3]
	\arrow["{b'}", from=2-2, to=3-2]
	\arrow["{w'}", from=3-2, to=2-1]
	\arrow["{f'}"', from=3-2, to=2-3]
    \end{cd}
    where $G(a')=a\c i^{-1}$ and $G(b')=b\c i^{-1}$, after transporting the displayed representatives through $i$.
    Since $G(w\c a')=G(w) \c a\c i^{-1}\in\F^H$, we have that $w\c a'\in \F^{H\c G}$ and thus the two starting fractions with respect to $\F^{H\c G}$ are equivalent. This concludes the proof that $H\c G$ satisfies (C2).

    We now prove that $H\c G$ satisfies condition (C1). We first record the inference used repeatedly below: if $H(V)\cong H(G(T))$, a representing roof has both legs in $\mathcal F^H$; their cones lie in $\ker H$, which is contained in the replete image of $G$. Two applications of closure under triangles therefore put $V$ in that image. We do so by showing that the image of $H\c G$ is closed up to isomorphism under all the operations involved in the construction of the smallest thick triangulated subcategory containing it. The fact that the image of $H\c G$ is closed up to isomorphism under shifts and finite coproducts is simply guaranteed by the fact that both $G$ and $H$ are exact functors between triangulated categories and thus they preserve shifts (up to isomorphism) and biproducts. We now show that the image of $H\c G$ is closed under direct summands. In a triangulated category every retraction $B\rightleftarrows A$ admits a complement: complete its section to a triangle, use exactness of $\operatorname{Hom}(-,B)$ to see that the connecting map is zero, and the splitting lemma identifies the triangle with $B\to B\oplus C\to C\to\Sigma B$ \cite[Section 13.4]{StacksDerived}. Thus this verifies closure under all existing retracts, not merely chosen decompositions. Let $A,B\in\U$ such that $A \oplus B$ is isomorphic to $H(G(C))$ for some $C\in \T$. Since $H$ is essentially surjective on objects, choose $A',B'\in\V$ with $H(A')\cong A$ and $H(B')\cong B$. As $H$ preserves biproducts, these choices give
    $$H(G(C)) \iso H(A') \oplus H(B') \iso H(A' \oplus B').$$
    Moreover, since $H$ satisfies (C2), there exists a fraction \begin{cd}*[2][2] \& D \& \\[-2ex]
	G(C) \&\& {A'\oplus B'}
	\arrow["w"', from=1-2, to=2-1]
	\arrow["f", from=1-2, to=2-3]\end{cd} with respect to $\F^H$ such that $H(f) \c H(w)^{-1}$ is an isomorphism. But $H(w)$ is an isomorphism because $w\in \F^H$ and so $H(f)$ is an isomorphism. This implies that $f\in \F^H$. Since $w\in \F^H$ and $G$ satisfies (C1), we have that $D\aiso{} G(E)$ for some $E\in \T$. So there exists $P\in \T$ such that $A'\oplus B' \aiso{} G(P)$. Since the image of $G$ is a thick triangulated subcategory, it is closed under direct summands, so there exist  $A''$ and $B''$ in $\T$ such that $A'\iso G(A'')$, $B' \aiso{} G(B'')$ and $G(A'' \oplus B'')\aiso{} A' \oplus B'$. Thus we conclude that
    $$H(G(A'')) \oplus H(G(B'')) \aiso{} H(G(A''\oplus B'')) \aiso{} H(A' \oplus B') \aiso{} H(G(C)).$$
    In particular $H(G(A''))\cong A$ and $H(G(B''))\cong B$, by the choices of $A',B'$. This shows that the image of $H\c G$ is closed under direct summands. It remains to prove that it is closed under cones. Consider an exact triangle in $\U$
 $$A \aar{} B \aar{} C \aar{} \Sigma A$$
 such that $A\aiso{} H(G(A'))$ and $B\aiso{} H(G(B'))$. We need to prove that there exists an object $C'\in\T$ such that $C\aiso{} H(G(C'))$. We have an isomorphism of exact triangles
 \begin{cd}[5][5]
     A \& B \& C \& {\Sigma A} \\
	{H(G(A'))} \& {H(G(B'))} \& C \& {\Sigma H(G(A'))}
	\arrow[from=1-1, to=1-2]
	\arrow[from=1-1, to=2-1, iso]
	\arrow[from=1-2, to=1-3]
	\arrow[from=1-2, to=2-2, iso]
	\arrow[from=1-3, to=1-4]
	\arrow[equals, from=1-3, to=2-3]
	\arrow[from=1-4, to=2-4, iso]
	\arrow[from=2-1, to=2-2]
	\arrow[from=2-2, to=2-3]
	\arrow[from=2-3, to=2-4]
 \end{cd}
 Since $H$ is essentially surjective on objects, there exists $C'\in \V$ such that $C\aiso{} H(C')$. Moreover, since $H$ is a 2-cokernel (and thus a Verdier localization), all exact triangles in $\U$ are (up to isomorphism) images of exact triangles in $\V$. So there exists an exact triangle in $\V$
 $$L \aar{} M \aar{} {N} \aar{} \Sigma L$$
 together with an isomorphism of exact triangles
 \begin{cd}
     {H(L)} \& {H(M)} \& {H(N)} \& {\Sigma H(L)} \\
	{H(G(A'))} \& {H(G(B'))} \& {H(C')} \& {\Sigma H(G(A'))}
	\arrow[from=1-1, to=1-2]
	\arrow[from=1-1, to=2-1,iso]
	\arrow[from=1-2, to=1-3]
	\arrow[from=1-2, to=2-2, iso]
	\arrow[from=1-3, to=1-4]
	\arrow[iso, from=1-3, to=2-3]
	\arrow[from=1-4, to=2-4, iso]
	\arrow[from=2-1, to=2-2]
	\arrow[from=2-2, to=2-3]
	\arrow[from=2-3, to=2-4]
 \end{cd}
 Consider now the isomorphism $H(L) \aiso{} H(G(A'))$. Since $H$ satisfies (C2) there exists a fraction \begin{cd}*[2][2] \& D \& \\[-2ex]
	L \&\& {G(A')}
	\arrow["w"', from=1-2, to=2-1]
	\arrow["f", from=1-2, to=2-3]\end{cd}
 with respect to $\F^H$ such that $H(f) \c H(w)^{-1}$ coincides with the starting isomorphism. Thanks to the fact that $G$ satisfies (C1) and that $\ker(H)$ factors through $G$ by hypothesis we conclude that there exists $L'\in \T$ such that $L \aiso{} G(L')$ (by the same argument used earlier in the proof). And analogously we conclude that there exists $M'\in \T$ such that $M \aiso{} G(M')$. So we have an exact triangle in $\V$
 $$G(L') \aar{} G(M') \aar{} N \aar{} \Sigma G(L').$$
 Since $G$ satisfies (C1), this implies that there exists $N'\in \T$ such that $N \aiso{} G(N')$. And thus we have
 $$C \aiso{} H(C') \aiso{} H(N) \aiso{} H(G(N')).$$
 This proves that the image of $H\c G$ is closed under cones up to isomorphism and thus the statement is proved.
\end{proof}

\begin{example}\label{ex:tri-not-diexact}
    The following counterexample shows that $\Triang$ is not
    $2$-di-exact in the sense of Definition~\ref{def:2-diexact}.
    It also shows that $2$-normal functors are not closed under
    composition. As a consequence, $\Triang$ is not $2$-Puppe exact.

    The strategy will be to find an exact functor $F\:\T\to \V$ between triangulated categories that satisfies the following three conditions:
    \begin{enum}
        \item $F$ is a composite $\T\xhookrightarrow{G} \U \aar{H}\V$ where $G$ is an inclusion of a thick triangulated subcategory $\T$ into $\U$ and $H$ is the Verdier localization of $\U$ by a thick triangulated subcategory $\S$ of $\U$;
        \item the 2-kernel of $F$ is $\0$;
        \item $F$ is not fully faithful.
    \end{enum}
    Indeed, we will then be able to deduce that $F$ cannot be written (up to isomorphism) as a composite of a 2-cokernel $H'$ followed by a 2-kernel $G'$, although $F$ is a composite of a 2-kernel followed by a 2-cokernel (see \thex\ref{teorchar2ker1} and \thex\ref{teorchar2coker1}). This is because if it were possible to write $F\iso G'\c H'$ we would have that $H'=\coker(\ker(H'))$ by \prox\ref{propkeriskerofitscoker} and
    $$\ker(H')=\ker(G'\c H')=\ker(F)=0$$
    Whence $H'$ is the 2-cokernel of a zero exact functor and must then be an equivalence. And then $F\iso G'\c H'$ would be fully faithful because $G'$ is so, contradicting the condition above.

    So we build an exact functor that satisfies the three conditions listed above.

    Let $k$ be a field and $Q$ be the directed quiver with three vertices:
    $$1 \aar{\alpha} 2 \aar{\beta} 3$$
    Let $A = kQ$ be the path algebra of this quiver, and consider the abelian category $\mod(A)$ of finite dimensional left $A$-modules. With the path-multiplication convention for which left modules correspond to representations in the displayed direction, $\mod(A)$ is equivalent to the category of finite-dimensional $k$-linear representations of $Q$ (\cite[Chapter III]{ASS}). We define $\U$ to be the bounded derived category
    $$\U:=\operatorname{D^b}(\mod(A))$$
    of the abelian category $\mod(A)$. Notice that the category $\mod(A)$ naturally embeds into the derived category $\U$. Consider the simple $A$-modules $S_1,S_2,S_3$ corresponding to the three vertices $1,2,3$ of the quiver, respectively. These are the quiver representations $(k \to 0 \to 0)$, $(0 \to k \to 0)$ and $(0 \to 0 \to k)$ respectively. We can then view $S_1, S_2, S_3$ inside the derived category $\U$ and consider the smallest thick subcategory $\langle S_2 \rangle$ of $\U$ generated by $S_2$, that we call $\S$, and the smallest thick subcategory $\langle S_1, S_3 \rangle$ of $\U$ generated by $S_1$ and $S_3$, that we call $\T$. We define the exact functor $F$ to be the composite
    $$\T\xhookrightarrow{G} \U \aar{H} \U/\S$$
    where $G$ is the inclusion of the thick triangulated subcategory $\T$ into $\U$ and $H$ is the Verdier localization of $\U$ by the thick triangulated subcategory $\S$. Clearly, $F$ is an exact functor that satisfies condition (i).

    We show that $F$ satisfies condition (ii). By \thex\ref{teorchar2ker1} and \prox\ref{propVerdfacts}, the 2-kernel of $F$ is given by (the full subcategory of $\U$ spanned by)
    $$\ker(F)=\set{X\in \T}{F(X)\iso 0}=\set{X\in \U}{H(X)\iso 0}\cap \T=\S\cap \T$$
    So it suffices to show that $\S\cap \T\iso\0$. Now, the path algebra $A = kQ$ of a quiver without oriented cycles is a hereditary algebra (by \cite[Chapter VII]{ASS}), which means that its global dimension is $\le 1$ (see \cite[Chapter VII]{ASS}). As a consequence (see \cite{ChenRingel}), in the derived category $\U=\operatorname{D^b}(\mod(A))$, every complex $X^{\bullet}$ is (non-canonically) isomorphic to the direct sum of its shifted cohomology modules:
    \begin{equation}\label{eqcomplex}
        X^{\bullet} \cong \bigoplus_{i \in \mathbb{Z}} \operatorname{H}^i(X^{\bullet})[-i]
    \end{equation}
    We can use this to deduce that every object $X \in \S=\langle S_2 \rangle$ has cohomology modules $\operatorname{H}^i(X)\in \mod(A)$ supported entirely on $\{2\}$. Indeed, $S_2$ satisfies this, and this property is preserved by all operations involved in taking the smallest thick subcategory generated by $S_2$, since evaluation at each vertex is exact, so the long exact cohomology sequence, shifts and retracts preserve this support condition. Analogously, every object $Y\in \T=\langle S_1,S_3 \rangle$ has cohomology modules $\operatorname{H}^i(Y)\in \mod(A)$ supported entirely on $\{1,3\}$. So every object in the intersection $\S\cap \T$ must have cohomology modules $\operatorname{H}^i(X)=0$ for every $i$. By \refs{eqcomplex}, we deduce that $\S\cap \T\iso\0$.

    It remains to prove that $F$ satisfies condition (iii). Consider $S_1$ and $S_3[1]$, which are both in $\T$. We prove that
    $$\HomC{\T}{S_1}{S_3[1]}=\HomC{\U}{S_1}{S_3[1]}=\{0\}$$
    while
    $$\HomC{\U/\S}{S_1}{S_3[1]}\neq\{0\}$$
    For this, we use that morphisms between (shifts of) modules in the derived category $\U=\operatorname{D^b}(\mod(A))$ are described by $\operatorname{Ext}$ groups:
    $$\HomC{\U}{S_1}{S_3[1]}=\operatorname{Ext}^1_A(S_1,S_3)$$
    Take any extension between $S_1$ and $S_3$:
    $$0 \to S_3 \to M \to S_1 \to 0$$
    Since dimension vectors of quiver representations are additive across short exact sequences, the dimension vector of $M$ must be the sum of the dimension vectors of $S_3$ and $S_1$. So since $\dim(S_3) = (0, 0, 1)$ and $\dim(S_1) = (1, 0, 0)$, we have $\dim(M) = (1, 0, 1)$. But this means that the quiver representation $M$ is forced to be of the form
    $$(k \xrightarrow{!} 0 \xrightarrow{!} k)$$
    And this is exactly the trivial split extension given by the direct sum $S_1 \oplus S_3$. Thus $\HomC{\U}{S_1}{S_3[1]}=\{0\}$. Intuitively, this is given by the fact that there is no arrow from $1$ to $3$ in the quiver $Q$.

    Finally, we build a non-zero morphism $S_1\to S_3[1]$ in $\U/\S$, i.e.\ a non-zero fraction \begin{cd}*[2][2] \& W \& \\[-2ex]
	S_1 \&\& S_3[1]
	\arrow["v"', from=1-2, to=2-1]
	\arrow["g", from=1-2, to=2-3]\end{cd}
    in $\U$ with respect to $\mathcal{M}(\S)$. The arrow $\alpha$ in the quiver has an associated quiver representation $E=(k\aar{\id{}} k \to 0)$, which is an extension of $S_2$ and $S_1$. The short exact sequence $0 \to S_2 \to E \xrightarrow{v} S_1 \to 0$, where $v$ projects the first vector space and forgets the second, yields an exact triangle
    $$S_2 \to E \xrightarrow{v} S_1 \to S_2[1]$$
    in $\U$. But then the cone of $v$ is precisely $S_2[1]\in \S$, whence $v\in \mathcal{M}(\S)$. We still need to build a morphism $g\:E\to S_3[1]$ in $\U$. But this corresponds with an extension
    $$0 \to S_3 \to P \to E \to 0$$
    in $\mod(A)$. We observe that $P=(k\aar{\id{}} k \aar{\id{}} k)$ yields such an extension. More precisely, the morphism $g$ in $\U$ that we need is the connecting morphism in the exact triangle
    $$S_3 \to P \to E \xrightarrow{g} S_3[1]$$
    associated to the extension above. It just remains to prove that the fraction \linebreak {\begin{cd}*[2][2] \& E \& \\[-2ex]
	S_1 \&\& S_3[1]
	\arrow["v"', from=1-2, to=2-1]
	\arrow["g", from=1-2, to=2-3]\end{cd}} that we built is non-zero. Equivalently, we need to prove that $g$, viewed as a morphism in $\U/\S$, is non-zero, since $v$ becomes invertible in the localization. But, by \cite[Lemma 13.6.7]{StacksTri}, this holds if and only if in $\U$ the morphism $g$ does not factor through an object of $\S$. Notice that $g$ is non-zero in $\U$, since the extension it corresponds to is not split, since a section $E\to P$ would have identity components at vertices $1,2$, while commutativity at the arrow $2\to3$ would force its component at vertex $2$ to be zero. Assume by contradiction that $g$ factors in $\U$ as $E\aar{a} N\to S_3[1]$ with $N\in \S$. We show that the only possible map $a\:E\to N$ is the zero morphism, which is a contradiction by the above argument. By what we showed above,
    \begin{equation}\label{eqN1234}
        N \cong \bigoplus_{i \in \mathbb{Z}} \operatorname{H}^i(N)[-i]
    \end{equation}
    We have also already showed that $N\in \S$ has cohomology modules $\operatorname{H}^i(N)\in \mod(A)$ supported entirely on $\{2\}$. Whence such cohomology modules must be finite direct sums of $S_2$. Since the Hom functor is additive and \refs{eqN1234} holds, it suffices to prove that
    $$\HomC{\U}{E}{S_2[j]}\iso\operatorname{Ext}^j_A(E,S_2)=\{0\}$$
    for every $j$. For $j<0$, negative-degree derived morphisms between modules vanish (we use $\operatorname{Ext}^j=0$ in these degrees). For $j=0$, $\operatorname{Hom}_A(E,S_2)=0$: commutativity at the arrow $1\to2$ forces the component at vertex $2$ to be zero. For $j\geqslant 2$, it holds because the global dimension of $A$ is $\leqslant 1$. It remains to prove that $\operatorname{Ext}^1_A(E,S_2)=\{0\}$. So consider an extension
    $$0 \to S_2 \to M \to E \to 0$$
    in $\mod(A)$. We show that it must be a split extension. Since dimension vectors are additive across short exact sequences and $\dim(S_2) = (0, 1, 0)$ and $\dim(E) = (1, 1, 0)$, we have that $\dim(M)=(1, 2, 0)$. So the quiver representation $M$ is of the form
    $$M = (k \xrightarrow{\alpha} k^2 \xrightarrow{0} 0)$$
    Write the quotient $M\to E$ at vertices $1,2$ as $p_1:k\to k$ and $p_2:k^2\to k$. Choose $u\in k$ with $p_1(u)=1$ and put $v=\alpha(u)$. Commutativity gives $p_2(v)=1$. The maps $1\mapsto u$ and $1\mapsto v$, with zero component at vertex $3$, form a section $E\to M$. Hence the extension splits, and $\operatorname{Ext}^1_A(E,S_2)=0$.
\end{example}

\section{A general quotient criterion for structured pointed categories}\label{sec:criterion}\label{subsec:criterion}

The direct di-exact and triangulated proofs establish the same
formal pattern with different inputs. We now formulate that pattern
using a specified $2$-functor to pointed categories.

\subsection{The ambient data and normal subcategories}

Recall that $\Pt$ has categories with a zero object as objects,
zero-object-preserving functors as $1$-cells, and all natural
transformations as $2$-cells. A \emph{$2$-category over $\Pt$} is
a $2$-category $\Lcat$ together with a specified $2$-functor
\[
 U:\Lcat\longrightarrow\Pt.
\]
It is \emph{locally faithful} if, for every $\C,\D$, the functor
\[
 U_{\C,\D}:\Lcat(\C,\D)
 \longrightarrow \Pt(U\C,U\D)
\]
is faithful. This requires injectivity on each set of parallel
$2$-cells; it does not require injectivity on objects or $1$-cells,
or fullness on $2$-cells. Write
\[
 \underline{\C}:=U(\C),\qquad
 \underline{F}:=U(F),\qquad
 \underline{\alpha}:=U(\alpha).
\]
An object of this specified $\Lcat$, with underlying category
$\underline{\C}$, is what we mean by a \emph{structured pointed
category}. Its morphisms and transformations are the $1$- and
$2$-cells of $\Lcat$.

Throughout this section, assume that $\Lcat$ is $2$-pointed in the
sense of Definition~\ref{def:2pointed}, with strong bizero object $0$,
that $U$ is locally faithful, and that
\[
 U(0)\simeq\mathbf{1}
\]
in $\Pt$, where $\mathbf{1}$ is the one-object, one-morphism category.
All categories are subject to the size convention of
Section~\ref{sec:framework}.

\begin{lemma}[Detection of null $1$-cells]\label{lem:criterion-null}
A $1$-cell $F:\C\to\D$ is null if and only if
$\underline{F}(C)$ is a zero object of $\underline{\D}$ for every
$C\in\underline{\C}$.
\end{lemma}
\begin{proof}
If $F$ is null, then $\underline{F}$ factors through $U(0)$ up to
natural isomorphism. Every object of $U(0)$ is a zero object, and
its image under a $1$-cell of $\Pt$ is again a zero object.

Conversely, suppose that $\underline{F}$ sends every object to zero.
Choose a parallel null $1$-cell $z:\C\to\D$. Strong pointedness
gives unique $2$-cells $\alpha:F\Rightarrow z$ and
$\beta:z\Rightarrow F$. Both underlying functors are zero-valued,
so the components of $\underline{\alpha}$ and
$\underline{\beta}$ are mutually inverse morphisms between zero
objects. Hence
\[
 U(\beta\alpha)=1_{U(F)},\qquad
 U(\alpha\beta)=1_{U(z)}.
\]
Local faithfulness gives $\beta\alpha=1_F$ and
$\alpha\beta=1_z$. Thus $F\iso z$ is null.
\end{proof}

For each object $\C$, specify a family $\mathfrak{N}(\C)$ of full
replete subcategories of $\underline{\C}$, called \emph{normal}.
For every $\N\in\mathfrak{N}(\C)$, specify an object $\C_{\N}$
and a $1$-cell
\[
 i_{\N}^{\C}:\C_{\N}\longrightarrow\C,
 \qquad
 U(\C_{\N})=\N,
 \qquad
 U(i_{\N}^{\C})=(\N\hookrightarrow\underline{\C}).
\]
These equalities are part of the chosen data. The $1$-cells
$i_{\N}^{\C}$ are called \emph{normal inclusions}.

Intersections, containment, inverse images, and collections of
objects below refer to the underlying categories. Thus, for
$F:\B\to\C$ and $\N\subseteq\underline{\C}$,
$\underline{F}^{-1}(\N)$ denotes the full subcategory of
$\underline{\B}$ on the objects $B$ with $\underline{F}(B)\in\N$.
Define the \emph{underlying annihilator} of $F$ by
\[
 \mathcal{K}_U(F)
 :=\{B\in\underline{\B}\mid\underline{F}(B)\iso0\},
\]
again regarded as a full subcategory. For
$\N\in\mathfrak{N}(\C)$, let
\[
 \Ann_{\N}(\C,\X)
 :=\{H\in\Lcat(\C,\X)\mid H i_{\N}^{\C}
       \text{ is null}\}
\]
be the full subcategory on the indicated $1$-cells.
By Lemma~\ref{lem:criterion-null}, this condition says exactly
that $\underline{H}$ sends every object of $\N$ to zero.

\subsection{The quotient criterion}

\begin{theorem}[Normal-subcategory quotient criterion]\label{thm:criterion}
Let $(\Lcat,U)$ and the normal-subcategory data be as above.
Suppose the following hold.
\begin{enumerate}[label=\textup{(Q\arabic*)},leftmargin=16mm]
\item Every normal subcategory contains the zero objects.
Normal subcategories are closed under arbitrary intersections,
including the empty intersection $\underline{\C}$, and under
inverse images along underlying functors of $1$-cells.
Normality is transitive: if
$\M\in\mathfrak{N}(\C)$ and
$\N\in\mathfrak{N}(\C_{\M})$, then
$\N\in\mathfrak{N}(\C)$, using
$U(\C_{\M})=\M\subseteq\underline{\C}$.
Furthermore, if $\N,\M\in\mathfrak{N}(\C)$ and
$\N\subseteq\M$, then
$\N\in\mathfrak{N}(\C_{\M})$.
For a collection $\Scl$ of objects of $\underline{\C}$, write
$\cl_{\C}(\Scl)$ for the intersection of all normal
subcategories containing $\Scl$.

\item For every $1$-cell $F:\B\to\C$,
$\mathcal{K}_U(F)$ belongs to $\mathfrak{N}(\B)$, and its
specified normal inclusion is a $2$-kernel of $F$.
Explicitly, postcomposition induces an equivalence
\[
 \Lcat(\X,\B_{\mathcal{K}_U(F)})
 \simeq
 \{G\in\Lcat(\X,\B)\mid FG\text{ is null}\},
\]
pseudonatural in $\X$, where the right-hand side is full.

\item For every $\N\in\mathfrak{N}(\C)$ there are an object
$\C/\N$ and a $1$-cell $q_{\N}^{\C}:\C\to\C/\N$ such that
\[
 \operatorname{Ob}U(\C/\N)=\operatorname{Ob}\underline{\C},
 \qquad
 \operatorname{Ob}U(q_{\N}^{\C})=\operatorname{id},
 \qquad
 \mathcal{K}_U(q_{\N}^{\C})=\N.
\]
Precomposition induces an equivalence
\begin{equation}\label{eq:normalquotientUP}
 \Lcat(\C/\N,\X)
 \simeq \Ann_{\N}(\C,\X),
\end{equation}
pseudonatural in $\X$. This is a statement about entire
hom-categories, including all their $2$-cells.

\item Let $\N\subseteq\M$ belong to $\mathfrak{N}(\C)$.
Put $\widehat{\M}:=\C_{\M}$. The restriction clause of
(Q1) gives $\N\in\mathfrak{N}(\widehat{\M})$, so (Q3)
supplies $q_{\N}^{\widehat{\M}}$. Moreover,
$q_{\N}^{\C}i_{\M}^{\C}$ annihilates $\N$ and therefore
induces a $1$-cell
\[
 j:\widehat{\M}/\N\longrightarrow\C/\N,
 \qquad
 j q_{\N}^{\widehat{\M}}
 \iso q_{\N}^{\C}i_{\M}^{\C}.
\]
The underlying functor $U(j)$ is fully faithful.
This requirement is independent of the choice of descent $j$.
In addition, $U$ reflects equivalences: whenever a $1$-cell
$v$ has $U(v)$ an equivalence of categories, $v$ is an
equivalence in $\Lcat$.
\end{enumerate}
Then $\Lcat$ is $2$-homological. For every $F:\B\to\C$, put
\[
 F\B:=\{\underline{F}(B)\mid B\in\operatorname{Ob}\underline{\B}\}.
\]
A $2$-cokernel of $F$ is
\begin{equation}\label{eq:generalcoker}
 \twocoker(F)=
 \bigl(q_{\cl_{\C}(F\B)}^{\C}:\C\longrightarrow
       \C/\cl_{\C}(F\B)\bigr).
\end{equation}
The $2$-kernels are precisely the normal inclusions up to
equivalence, and the $2$-cokernels are precisely the quotient
$1$-cells in \textup{(Q3)} up to equivalence.
\end{theorem}

\begin{proof}
\emph{Existence and normal forms.}
Condition (Q2) gives all $2$-kernels. For $F:\B\to\C$ and
$H:\C\to\X$, Lemma~\ref{lem:criterion-null} gives
\[
 HF\text{ is null}
 \quad\Longleftrightarrow\quad
 F\B\subseteq\mathcal{K}_U(H).
\]
The subcategory $\mathcal{K}_U(H)$ is normal by (Q2).
Consequently the latter condition is equivalent to
\[
 \cl_{\C}(F\B)\subseteq\mathcal{K}_U(H),
\]
or equivalently to
$H i_{\cl_{\C}(F\B)}^{\C}$ being null.
These conditions define the same full subcategory of
$\Lcat(\C,\X)$. Thus (Q3) proves
\eqref{eq:generalcoker}, including its universal property on
$2$-cells.

By (Q2) and the annihilator equality in (Q3), every normal
inclusion $i_{\N}^{\C}$ is a $2$-kernel of $q_{\N}^{\C}$.
By \eqref{eq:normalquotientUP}, every $q_{\N}^{\C}$ is a
$2$-cokernel of $i_{\N}^{\C}$. The converse normal forms follow
from the preceding constructions and uniqueness of representing
objects up to equivalence.

We record a consequence that makes the subsequent uses of these
normal forms precise. Suppose that $h:\B\to\C$ has fully
faithful underlying functor and that its full replete essential
image $\Scl$ belongs to $\mathfrak{N}(\C)$. The composite
$q_{\Scl}^{\C}h$ is null by Lemma~\ref{lem:criterion-null}.
Since $i_{\Scl}^{\C}$ is its quotient's $2$-kernel, there are
$t:\B\to\C_{\Scl}$ and an invertible $2$-cell
\[
 i_{\Scl}^{\C}t\iso h.
\]
The functor $U(t)$ is fully faithful and essentially surjective.
By the final clause of (Q4), $t$ is an equivalence in $\Lcat$.
Hence $h$ is a $2$-kernel, indeed a $2$-kernel of
$q_{\Scl}^{\C}$. More generally, the same argument shows that
$h$ is a $2$-kernel of any $p:\C\to\D$ satisfying
$\mathcal{K}_U(p)=\Scl$.
Conversely, every $2$-kernel has fully faithful underlying functor
and normal full replete essential image, by its normal form.

\emph{Composition of $2$-kernels.}
Let $k:\A\to\B$ and $m:\B\to\C$ be $2$-kernels.
Choose a normal form
\[
 m\iso i_{\M}^{\C}e,
 \qquad e:\B\longrightarrow\C_{\M}
 \quad\text{an equivalence}.
\]
If $k$ is a $2$-kernel of $f$ and $d$ is a quasiinverse of $e$,
its universal property shows that $ek$ is a $2$-kernel of $fd$.
Thus its underlying functor is fully faithful
and its full replete essential image $\Pcal$ is normal in
$U(\C_{\M})=\M$. Transitivity in (Q1) makes $\Pcal$ normal
in $\underline{\C}$. Since $\M$ is full and replete in
$\underline{\C}$, the underlying functor of $mk$ is fully faithful
with full replete essential image $\Pcal$. The preceding
consequence therefore makes $mk$ a $2$-kernel.

\emph{Composition of $2$-cokernels.}
Using normal forms and transporting across the equivalence at
the intermediate object, it suffices to consider
\[
 \C\xrightarrow{q}\C/\N
 \xrightarrow{r}(\C/\N)/\Scl,
 \qquad
 q=q_{\N}^{\C},\quad r=q_{\Scl}^{\C/\N},
\]
where $\Scl\in\mathfrak{N}(\C/\N)$. Put
\begin{equation}\label{eq:Ppreimage}
 \Pcal:=\underline{q}^{-1}(\Scl).
\end{equation}
This is normal by (Q1), and it contains $\N$, because
$\underline{q}$ sends $\N$ to zero and $\Scl$ contains the
zero objects. Since $\underline{q}$ is the identity on objects,
every object of $\Scl$ is $\underline{q}(P)$ for an object
$P\in\Pcal$.

Under the equivalence induced by $q$, a $1$-cell
$T:\C/\N\to\X$ annihilates $\Scl$ if and only if $Tq$
annihilates $\Pcal$, by Lemma~\ref{lem:criterion-null} and
the preceding object-surjectivity. Conversely, every $H$ that
annihilates $\Pcal$ annihilates $\N$, so it descends through
$q$ by (Q3); its descendant annihilates $\Scl$ by the same
argument. Restricting the full hom-category equivalences in (Q3)
therefore gives
\[
 \Lcat((\C/\N)/\Scl,\X)
 \simeq \Ann_{\Scl}(\C/\N,\X)
 \simeq \Ann_{\Pcal}(\C,\X),
\]
with composite induced by precomposition with $rq$.
Thus $rq$ is a $2$-cokernel of $i_{\Pcal}^{\C}$, and
\begin{equation}\label{eq:iteratedgeneral}
 (\C/\N)/\Scl\simeq\C/\Pcal
\end{equation}
compatibly with the quotient $1$-cells from $\C$.

\emph{The homology axiom.}
Use normal forms to replace the $2$-kernel and $2$-cokernel
in (H3) by
\[
 m=i_{\M}^{\C}:\widehat{\M}\longrightarrow\C,
 \qquad
 e=q_{\N}^{\C}:\C\longrightarrow\C/\N,
 \qquad \widehat{\M}=\C_{\M}.
\]
The omitted equivalences do not affect the containment hypothesis
or the existence of a $2$-normal factorization.
The hypothesis that a $2$-kernel of $e$ factors through $m$
implies $\N\subseteq\M$: apply $U$ to the factorization of
$i_{\N}^{\C}$ through $i_{\M}^{\C}$ and use repleteness
of $\M$.

The quotient $q_{\M}^{\C}$ annihilates $\N$, so (Q3) gives
\[
 p:\C/\N\longrightarrow\C/\M,
 \qquad p q_{\N}^{\C}\iso q_{\M}^{\C}.
\]
Because $U(q_{\N}^{\C})$ is the identity on objects,
an object of $U(\C/\N)$, labelled by
$C\in\underline{\C}$, belongs to $\mathcal{K}_U(p)$
exactly when $C\in\M$.
Let
\[
 j:\widehat{\M}/\N\longrightarrow\C/\N
\]
be the descent in (Q4). Its underlying functor is fully faithful,
and its full replete essential image is exactly
$\mathcal{K}_U(p)$: the comparison
$j q_{\N}^{\widehat{\M}}\iso q_{\N}^{\C}i_{\M}^{\C}$
identifies its image objects up to isomorphism, and
$U(q_{\N}^{\widehat{\M}})$ is the identity on objects.
By (Q2), $\mathcal{K}_U(p)$ is normal. The consequence proved
above therefore makes $j$ a $2$-kernel of $p$.
Hence
\begin{equation}\label{eq:generalhomology}
 \widehat{\M}\xrightarrow{q_{\N}^{\widehat{\M}}}
 \widehat{\M}/\N\xrightarrow{j}\C/\N
\end{equation}
is a $2$-cokernel followed by a $2$-kernel and agrees with $em$
up to invertible $2$-cell. This proves (H3), and completes the
proof of $2$-homologicality.
\end{proof}

\subsection{The third isomorphism and the subquotient square}

\begin{corollary}[The third isomorphism and subquotient square]\label{cor:third}
Under the hypotheses of Theorem~\ref{thm:criterion}, let
$\N\subseteq\M$ belong to $\mathfrak{N}(\C)$ and put
$\widehat{\M}=\C_{\M}$. Use the notation
\[
 i=i_{\M}^{\C},\qquad
 a=q_{\N}^{\widehat{\M}},\qquad
 b=q_{\N}^{\C},
 \qquad \theta:bi\xRightarrow{\sim}ja
\]
for the induced square and its comparison. Then
$p:\C/\N\to\C/\M$, characterized by
$pb\iso q_{\M}^{\C}$, is a $2$-cokernel of $j$.
If $\mathcal{R}:=\mathcal{K}_U(p)$, then $\mathcal{R}$ is
the normal full replete essential image of $U(j)$ and
\begin{equation}\label{eq:third}
 (\C/\N)/\mathcal{R}\simeq\C/\M.
\end{equation}
When this equivalence is written as
$(\C/\N)/(\widehat{\M}/\N)\simeq\C/\M$, the denominator
means the normal subcategory $\mathcal{R}$, not a literal
identification of the domain of $j$ with a full subcategory.
The square
\begin{equation}\label{eq:subquotientsquare}
\begin{tikzcd}[column sep=large,row sep=large]
 \widehat{\M} \arrow[r,tail,"i"] \arrow[d,"a"'] &
 \C \arrow[d,"b"] \\
 \widehat{\M}/\N \arrow[r,tail,"j"'] & \C/\N
\end{tikzcd}
\end{equation}
with comparison $\theta$ is both a bipullback and a bipushout
in $\Lcat$.
\end{corollary}

\begin{proof}
The theorem's proof gives that $j$ is a $2$-kernel of $p$,
that $\mathcal{R}$ is its normal full replete essential image,
and that $i$ is a $2$-kernel of $pb\iso q_{\M}^{\C}$.

\emph{The $2$-cokernel and the third isomorphism.}
For $H:\C/\N\to\X$, Lemma~\ref{lem:criterion-null} and
object-surjectivity of $U(a)$ give
\[
 Hj\text{ is null}
 \quad\Longleftrightarrow\quad
 Hja\text{ is null}
 \quad\Longleftrightarrow\quad
 Hbi\text{ is null}.
\]
Since $\N\subseteq\M$, (Q3) therefore restricts to an
equivalence between the full category of $H$ annihilating $j$
and $\Ann_{\M}(\C,\X)$. The latter is represented by
$\C/\M$, with the corresponding functor induced by
precomposition with $p$, because $pb\iso q_{\M}^{\C}$.
Thus $p$ is a $2$-cokernel of $j$.
By Lemma~\ref{lem:criterion-null}, annihilating $j$ is also
equivalent to annihilating its full replete essential image
$\mathcal{R}$. Applying (Q3) and uniqueness of representing
objects proves \eqref{eq:third}, compatibly with the $1$-cells
from $\C/\N$.

\emph{The bipullback.}
Fix $\X$. A cone consists of $u:\X\to\C$,
$v:\X\to\widehat{\M}/\N$, and an invertible $2$-cell
$\alpha:bu\Rightarrow jv$.
A morphism to $(u',v',\alpha')$ is a pair of $2$-cells
$\eta:u\Rightarrow u'$ and $\nu:v\Rightarrow v'$ satisfying
$(j\nu)\circ\alpha=\alpha'\circ(b\eta)$.
Since $pj$ is null, $pbu$ is null. The $2$-kernel property of
$i$ gives $t:\X\to\widehat{\M}$ and an invertible $2$-cell
$\lambda:it\Rightarrow u$.
The composite
\[
 jat\xRightarrow{\theta^{-1}t}bit
 \xRightarrow{b\lambda}bu
 \xRightarrow{\alpha}jv
\]
lifts uniquely to an invertible $2$-cell $at\Rightarrow v$,
because postcomposition with the $2$-kernel $j$ is fully
faithful on each hom-category. Hence every cone is isomorphic
to one induced by $t$.

For morphisms between cones induced by $t,t'$, full faithfulness
of postcomposition with $i$ uniquely lifts the component into
$\C$ to a $2$-cell $t\Rightarrow t'$. The compatibility
equation for a cone morphism, and faithfulness of
postcomposition with $j$, force its other component to be the
whiskering of this same $2$-cell with $a$.
Thus the induced functor from $\Lcat(\X,\widehat{\M})$ to
the cone category is fully faithful and essentially surjective.
These functors are pseudonatural in $\X$, proving the bipullback
property.

\emph{The bipushout.}
Fix $\X$. A cocone consists of $H:\C\to\X$,
$K:\widehat{\M}/\N\to\X$, and an invertible $2$-cell
$\sigma:Hi\Rightarrow Ka$.
A morphism to $(H',K',\sigma')$ is a pair of $2$-cells
$\eta:H\Rightarrow H'$ and $\nu:K\Rightarrow K'$ satisfying
$(\nu a)\circ\sigma=\sigma'\circ(\eta i)$.
Since $U(a)$ sends every object of $\N$ to zero, the
comparison $\sigma$ shows that $\underline{H}$ sends every
object of $\N\subseteq\underline{\C}$ to zero.
Lemma~\ref{lem:criterion-null} and (Q3) give
$T:\C/\N\to\X$ and an invertible $2$-cell
$\varepsilon:Tb\Rightarrow H$.
The composite
\[
 (Tj)a\xRightarrow{T\theta^{-1}}Tbi
 \xRightarrow{\varepsilon i}Hi
 \xRightarrow{\sigma}Ka
\]
lifts uniquely to an invertible $2$-cell $Tj\Rightarrow K$,
because precomposition with the $2$-cokernel $a$ is fully
faithful. This makes the given cocone isomorphic to one induced
by $T$.

For morphisms between cocones induced by $T,T'$, full
faithfulness of precomposition with $b$ uniquely lifts the
component on $\C$ to a $2$-cell $T\Rightarrow T'$.
The cocone compatibility equation and faithfulness of
precomposition with $a$ identify its other component with
the whiskering of this $2$-cell with $j$.
Thus the induced functor from $\Lcat(\C/\N,\X)$ to the
cocone category is fully faithful and essentially surjective,
pseudonaturally in $\X$. This proves the bipushout property.
\end{proof}

\begin{remark}[What the criterion packages]\label{rem:criterion-role}
The containment $\N\subseteq\M$ in
\eqref{eq:generalhomology} is the hypothesis used in the direct
proof of Theorem~\ref{di:thm:main} and, in roof language, in
Theorem~\ref{thm:tri-direct}. The criterion abstracts their final
composition and homology arguments.
Its hypotheses refer to the specified $2$-functor $U$, the
chosen normal inclusions, and the quotient universal properties
inside $\Lcat$. In particular, the final clause of (Q4) is an
explicit hypothesis, not a consequence of local faithfulness.
The criterion neither constructs the quotient objects nor proves
that their quotient functors have the properties required of
$1$-cells of $\Lcat$; these remain inputs in each application.
It does not assert that $U$ preserves $2$-cokernels.
\end{remark}

\section{Pointed categories and pointed ideal quotients}\label{sec:pointed}

We apply the criterion first to categories with only a zero object
and zero-object-preserving functors. Here the quotient is an explicit
ideal quotient, rather than an exact or Verdier localization.

\subsection{The strong bizero and all 2-kernels}

In a pointed category, an object $X$ is zero if and only if $1_X=0$:
this equation forces every morphism to or from $X$ to be zero. Functors
preserving zero objects preserve zero morphisms and retraction diagrams.
A retract of a zero object is zero.

Let $\Zcat$ be the category with one object and one morphism.

\begin{proposition}\label{prop:pointedzero}
The category $\Zcat$ is a strong bizero, hence a $2$-zero object,
in $\Pt$. A $1$-cell is null exactly when it takes every object to
a zero object. For every zero-preserving $F:\C\to\D$ and every
parallel null $Z:\C\to\D$, there is exactly one natural
transformation $F\Rightarrow Z$ and exactly one $Z\Rightarrow F$.
Thus the null functors are precisely the zero objects of the
hom-categories of $\Pt$.
\end{proposition}
\begin{proof}
There is exactly one functor $\C\to\Zcat$. A zero-preserving functor
$\Zcat\to\C$ selects a zero object of $\C$, and any two selections
have exactly one natural transformation between them. The latter is
invertible. A zero-valued functor is naturally isomorphic to a constant
zero functor, using the unique maps between zero objects. Conversely,
factoring through $\Zcat$ up to isomorphism implies being zero-valued.
Let now $F:\C\to\D$ be any zero-preserving functor and let
$Z:\C\to\D$ be zero-valued. For each $X$, there are unique maps
$F(X)\to Z(X)$ and $Z(X)\to F(X)$. They are the components of
natural transformations: for $u:X\to Y$, the first naturality
square commutes because its composites have zero codomain $Z(Y)$,
and the second because its composites have zero domain $Z(X)$.
Every transformation in either direction has these components,
proving existence and uniqueness in both directions. This is exactly
Definition~\ref{def:2pointed}.
\end{proof}

\begin{lemma}[Full annihilation kernels]\label{lem:annihilation}
For a zero-preserving functor $F:\C\to\D$, the inclusion
\begin{equation}\label{eq:KF}
 k_F:\K_F\hookrightarrow\C,
 \qquad \Ob(\K_F)=\{X\in\C\mid F(X)\iso0\},
\end{equation}
is a 2-kernel in $\Pt$. The subcategory $\K_F$ is full, replete,
pointed, and closed under retracts in $\C$.

More generally, let $(\Lcat,U)$ satisfy the ambient assumptions of
Section~\ref{sec:criterion}, and let $F:\B\to\C$ be a
$1$-cell. Suppose that there are an object $\B'$ and a
$1$-cell $i:\B'\to\B$ such that
\[
 U(\B')=\mathcal K_U(F),\qquad
 U(i)=(\mathcal K_U(F)\hookrightarrow U(\B)).
\]
Assume, for every object $\X$, that each $G:\X\to\B$ with
$FG$ null admits $\overline G:\X\to\B'$ and an invertible
$2$-cell $i\overline G\iso G$, and that, for all
$H,H':\X\to\B'$, every $2$-cell $iH\Rightarrow iH'$
is $i\beta$ for a unique $\beta:H\Rightarrow H'$.
Then $i$ is a $2$-kernel of $F$.
\end{lemma}
\begin{proof}
Preservation of zero objects, isomorphisms, and retractions proves the
closure assertions. For any $\X$, a zero-preserving $G:\X\to\C$
has $FG\iso0$ exactly when all its values lie in $\K_F$. Then $G$
corestricts uniquely to $\K_F$. Every natural transformation between two
such functors lifts uniquely, component by component, since the inclusion
is full. This gives the isomorphism of functor categories in
\eqref{eq:kernelUP}, natural under precomposition.

In the general case, $Fi$ is null by
Lemma~\ref{lem:criterion-null}. The stated factorization and lifting
assumptions make postcomposition with $i$ essentially surjective and
fully faithful onto the full subcategory of $\Lcat(\X,\B)$
annihilated by $F$. These functors commute with precomposition, so
\eqref{eq:kernelUP} holds pseudonaturally in $\X$.
\end{proof}

\subsection{Pointed quotients and retract closure}

For a collection $\Scl$ of objects of $\C$, define a two-sided ideal
of the underlying category by
\begin{equation}\label{eq:pointedideal}
 I_{\Scl}(X,Y)=\{0_{X,Y}\}\cup
 \{ba\mid X\xrightarrow{a}S\xrightarrow{b}Y,\ S\in\Scl\}.
\end{equation}
It is closed under arbitrary pre- and postcomposition. Let
$\pq{\C}{\Scl}$ have the same objects as $\C$, with each hom-set
obtained by identifying all members of $I_{\Scl}(X,Y)$ to one element
and making no other identifications. Explicitly,
\begin{equation}\label{eq:Rees}
 [f]=[g]\quad\Longleftrightarrow\quad
 f=g\ \text{or}\ \bigl(f,g\in I_{\Scl}(X,Y)\bigr).
\end{equation}
Composition $[g][f]=[gf]$ is well-defined by the ideal property. The
original zero object remains a zero object. We call this the
\emph{pointed quotient}; it is the object-generated Rees quotient of
the category. The definition \eqref{eq:Rees}, rather than a choice of
terminology, specifies the construction completely.

Write $\ret_{\C}(\Scl)$ for the full replete subcategory of retracts
of objects of $\Scl$ together with the zero objects; equivalently, of
retracts of $\Scl\cup\{0\}$. A retract here is an existing diagram
$X\xrightarrow{a}S\xrightarrow{b}X$ with $ba=1_X$.

\begin{proposition}\label{prop:pointedquotient}
The quotient $q:\C\to\pq{\C}{\Scl}$ is zero-preserving and has the
following properties.
\begin{enumerate}[label=\textup{(\roman*)}]
\item $q(X)\iso0$ if and only if $X\in\ret_{\C}(\Scl)$.
\item $I_{\Scl}=I_{\ret_{\C}(\Scl)}$.
\item For every pointed category $\X$, precomposition gives an
isomorphism of categories
\begin{equation}\label{eq:pointedquotientUP}
 \Pt(\pq{\C}{\Scl},\X)\ \iso\
 \{H\in\Pt(\C,\X)\mid H(S)\iso0\text{ for all }S\in\Scl\},
\end{equation}
where the right-hand side contains all natural transformations.
\end{enumerate}
\end{proposition}
\begin{proof}
An object $q(X)$ is zero exactly when $[1_X]=[0]$. By
\eqref{eq:Rees}, either $1_X=0$, or $1_X$ factors through an object
of $\Scl$. These alternatives say exactly that $X$ is a zero object
or the indicated retract, proving (i). Any morphism factoring through a
retract of $S$ also factors through $S$. This proves (ii).

If $H$ annihilates $\Scl$, it sends every element of
$I_{\Scl}(X,Y)$ to zero. Hence
$\overline H(X)=H(X)$, $\overline H([f])=H(f)$ defines a unique
zero-preserving factor $\overline Hq=H$. Conversely, every such factor
annihilates $\Scl$ by (i). A transformation between two composites
with $q$ descends with unchanged components: every quotient morphism
is represented by a morphism of $\C$, so its naturality is already
known. The descent is unique because $q$ is the identity on objects.
This proves (iii), naturally in $\X$.
\end{proof}

Call a full replete pointed subcategory \emph{retract-closed} when it
contains every retract in the ambient category of any of its objects.
For a zero-preserving $F:\B\to\C$, the preceding proposition gives
\begin{equation}\label{eq:pointedcoker}
 \twocoker(F)=
 \bigl(\C\longrightarrow\pq{\C}{\ret_{\C}(F\B)}\bigr).
\end{equation}
Thus all 2-kernels and 2-cokernels already exist in $\Pt$.
Moreover, every retract-closed inclusion $\N\hookrightarrow\C$ is
the 2-kernel of its pointed quotient, and the quotient is the 2-cokernel
of that inclusion. Together with uniqueness of 2-kernels, this shows:
\begin{equation}\label{eq:pointednormalforms}
 \begin{split}
 \text{2-kernels in }\Pt
 &\ \longleftrightarrow\ \text{retract-closed full pointed inclusions},\\
 \text{2-cokernels in }\Pt
 &\ \longleftrightarrow\ \text{their pointed quotients},
 \end{split}
\end{equation}
where the correspondences are understood up to equivalence.

\subsection{Application of the criterion}

\begin{theorem}\label{thm:pointed}
The 2-category $\Pt$ is 2-pointed and 2-homological.
\end{theorem}
\begin{proof}
Take $U=1_{\Pt}$ and realize each normal subcategory by
itself with its pointed structure and inclusion. Then $U$
is locally faithful, sends the one-object zero category to
$\mathbf1$, and reflects equivalences.

Use retract-closed full replete pointed subcategories as the normal
subcategories in Theorem~\ref{thm:criterion}. Intersections, inverse
images, and restriction to a full pointed subcategory preserve the
required properties. For transitivity, a retract in $\C$ of an object
of $\N\subseteq\M\subseteq\C$ first belongs to $\M$; fullness
puts its retraction diagram in $\M$, so it then belongs to $\N$.
The smallest normal subcategory on $\Scl$ is $\ret_{\C}(\Scl)$.
Thus (Q1) holds. Lemma~\ref{lem:annihilation} gives (Q2), and
Proposition~\ref{prop:pointedquotient} gives (Q3).

For (Q4), let $\N\subseteq\M\subseteq\C$ be normal and
$X,Y\in\M$. Every factorisation $X\to N\to Y$ with $N\in\N$
lies in $\M$ by fullness. Hence
\[
 I_{\N}^{\M}(X,Y)=I_{\N}^{\C}(X,Y).
\]
The hom-set equivalence relations \eqref{eq:Rees} restrict identically,
so $\pq{\M}{\N}\to\pq{\C}{\N}$ is fully faithful.
Proposition~\ref{prop:pointedzero} provides strong pointedness,
so the criterion applies.
\end{proof}

\section{Additive categories and the reusable localization step}\label{sec:additive}

\subsection{What is inherited from the pointed case}

An additive category is a preadditive category with finite biproducts,
including a zero object; see \cite{StacksAdd}. An additive functor
preserves zero objects and finite biproducts. Thus there is a forgetful
2-functor $\Add\to\Pt$. The one-object zero category is again a
strong bizero, by Proposition~\ref{prop:pointedzero}. Indeed,
zero-valued functors are additive, and the unique natural
transformations from any additive functor to a parallel zero-valued
functor, and in the reverse direction, are allowed $2$-cells of
$\Add$. Hence every null additive functor is a zero object in its
hom-category.

For an additive $F:\A\to\B$, the pointed annihilator $\K_F$ is
closed under finite biproducts, since
$F(X\oplus Y)\iso F(X)\oplus F(Y)$. It has its induced additive
structure, and additive functors landing in it corestrict additively.
All natural transformations still lift through the full inclusion.
Lemma~\ref{lem:annihilation} therefore gives the 2-kernel in $\Add$
without a new 2-dimensional argument.

The quotient, however, must change: collapsing only an ideal's zero
class does not in general respect addition. We now replace the pointed
quotient by the additive quotient. For a standard reference to the construction
see \cite[Recollection 2.7]{BalmerDA}.

\subsection{Additive ideal quotients}

For a collection $\Scl\subseteq\Ob(\A)$, let $[\Scl]_{\A}(X,Y)$
be the subgroup of $\A(X,Y)$ of all sums
\begin{equation}\label{eq:addideal}
 \sum_{i=1}^{n} b_i a_i,
 \qquad X\xrightarrow{a_i}S_i\xrightarrow{b_i}Y,
 \quad S_i\in\Scl.
\end{equation}
This is the two-sided additive ideal generated by the identities of
objects of $\Scl$. The quotient $\aq{\A}{\Scl}$ has the same
objects as $\A$ and hom-groups
\begin{equation}\label{eq:addquotient}
 (\aq{\A}{\Scl})(X,Y)=\A(X,Y)/[\Scl]_{\A}(X,Y).
\end{equation}
Composition descends by bilinearity; the equations defining biproducts
descend as well. Hence it is additive and its quotient functor is
additive, the identity on objects, and surjective on hom-groups.

Let $\add_{\A}(\Scl)$ be the full replete subcategory of retracts of
finite biproducts of objects of $\Scl$, including the empty biproduct.
It is the least retract-closed additive subcategory containing $\Scl$.

\begin{lemma}\label{lem:additivequotient}
For the additive quotient $a:\A\to\aq{\A}{\Scl}$,
\begin{equation}\label{eq:additivezero}
 a(X)\iso0\quad\Longleftrightarrow\quad
 1_X\in[\Scl]_{\A}(X,X)
 \quad\Longleftrightarrow\quad X\in\add_{\A}(\Scl).
\end{equation}
Moreover $[\Scl]_{\A}=[\add_{\A}(\Scl)]_{\A}$, and for every
additive $\X$, precomposition gives an isomorphism
\begin{equation}\label{eq:additiveUP}
 \Add(\aq{\A}{\Scl},\X)\iso
 \{H\in\Add(\A,\X)\mid H(\Scl)\iso0\},
\end{equation}
including arbitrary natural transformations.
\end{lemma}
\begin{proof}
An identity in the ideal can be written as $1_X=\sum_i b_i a_i$.
The column $(a_i):X\to\bigoplus_i S_i$ and the row
$(b_i):\bigoplus_i S_i\to X$ exhibit $X$ as a retract. The converse
is the same matrix equation. Factoring through a retract of a biproduct
produces a sum of the form \eqref{eq:addideal}, proving equality of the
ideals.

An additive $H$ annihilates $\Scl$ exactly when it kills the displayed
ideal. Thus it descends uniquely by $\overline H([f])=H(f)$.
The natural-transformation argument of
Proposition~\ref{prop:pointedquotient}(iii) applies verbatim: the quotient
is the identity on objects and full, so naturality on a representative
is exactly naturality on its class. This proves \eqref{eq:additiveUP}.
\end{proof}

\subsection{Additive homologicity from the pointed proof}

\begin{theorem}\label{thm:additive}
The 2-category $\Add$ is 2-pointed and 2-homological. For
$F:\A\to\B$,
\begin{equation}\label{eq:additivecoker}
 \twocoker(F)=
 \bigl(\B\longrightarrow\aq{\B}{\add_{\B}(F\A)}\bigr).
\end{equation}
A 1-cell is a 2-kernel precisely when it is fully faithful and its
essential image is retract-closed. Every 2-cokernel is, up to equivalence,
an additive quotient by a retract-closed full additive subcategory.
\end{theorem}
\begin{proof}
Use the canonical forgetful $2$-functor $U:\Add\to\Pt$.
It is locally faithful and sends the one-object zero category
to $\mathbf1$; the inherited additive categories and their
inclusions realize the normal-subcategory data.

Take the normal subcategories to be the retract-closed full replete
additive subcategories. The proof of (Q1) in the pointed case still
applies, with finite biproduct closure added. Inverse images preserve
this extra condition because the functors are additive. For transitivity,
biproduct diagrams in a full additive subcategory are ambient biproduct
diagrams, so the same inclusions are additive. The closure operation
is $\add_{\A}$.

The preceding subsection on inherited kernels gives (Q2).
Lemma~\ref{lem:additivequotient} gives (Q3). For (Q4), with
$\N\subseteq\M\subseteq\A$ and $X,Y\in\M$, every factorisation
through an object of $\N$ is internal to $\M$, so
\begin{equation}\label{eq:addrestriction}
 [\N]_{\M}(X,Y)=[\N]_{\A}(X,Y).
\end{equation}
Consequently $\aq{\M}{\N}\to\aq{\A}{\N}$ is fully faithful.
An additive functor that is an equivalence of underlying
categories admits an additive quasi-inverse. Hence $U$ reflects
equivalences, and all hypotheses of Theorem~\ref{thm:criterion}
hold.

The normal-form characterisation follows from that theorem; the
essential image of a fully faithful additive functor is already closed
under finite biproducts, up to isomorphism. Its remaining normality
condition is exactly retract closure.
\end{proof}

In particular the additive homology square is
\begin{equation}\label{eq:additivesquare}
\begin{tikzcd}[column sep=large,row sep=large]
 \M\arrow[r,tail]\arrow[d] & \A\arrow[d]\\
 \aq{\M}{\N}\arrow[r,tail] & \aq{\A}{\N}.
\end{tikzcd}
\end{equation}
Both rows are 2-kernels, both columns are 2-cokernels, and the square
is a bipullback and a bipushout. Closure under composition requires
no further ideal calculation: it is already covered by the common
proof, including all 2-cells.

\subsection{Two localization lemmas}\label{subsec:localization}

To pass from the additive case to the abelian and triangulated cases,
we need to invert additional morphisms. We record the exact extra
arguments. Ordinary localization and the calculus of fractions used
here are classical; see \cite[Chapter I]{GZ} and \cite{StacksLoc}.
The proofs below supply the natural-transformation and full-subcategory
features required for our 2-dimensional application.

\begin{lemma}[Natural transformations descend through localization]\label{lem:localizationNat}
Let $L:\C\to\C[W^{-1}]$ be the ordinary localization, in its
identity-on-objects model. For every category $\X$, precomposition
is fully faithful and induces an equivalence
\[
 \Cat(\C[W^{-1}],\X)\simeq
 \{H:\C\to\X\mid H(w)\text{ is invertible for all }w\in W\},
\]
where the right-hand side is full on all natural transformations.
Whenever an additional componentwise equation defines allowed 2-cells,
the same result restricts to those 2-cells if the equation can be checked
on objects of $\C$.
\end{lemma}
\begin{proof}
The equivalence on full functor categories is
Lemma~\ref{lem:puppe-nat}. In the identity-on-objects model,
the descended transformation has the same components as the
original one. Hence any specified componentwise equation
that can be checked on those objects holds before descent
if and only if it holds afterwards.
\end{proof}

\begin{lemma}[Restriction of roofs]\label{lem:roofs}
Let $\M\subseteq\C$ be full. Suppose $W$ admits a calculus of
fractions by roofs $X\xleftarrow{s}Z\xrightarrow{f}Y$, $s\in W$,
and $W_{\M}=W\cap\Mor(\M)$ admits the corresponding calculus in
$\M$. Suppose also
\begin{equation}\label{eq:roofclosed}
 s:Z\to X\in W,\quad X\in\M
 \quad\Longrightarrow\quad Z\in\M.
\end{equation}
Then $\M[W_{\M}^{-1}]\to\C[W^{-1}]$ is fully faithful.
\end{lemma}
\begin{proof}
A morphism between $X,Y\in\M$ in the larger localization has a roof
as displayed. Condition \eqref{eq:roofclosed} puts $Z$ in $\M$;
fullness puts both arrows there, proving fullness after localization.
If two roofs from $\M$ become equal, the fraction calculus gives a
common refining roof with denominator $s':Z'\to X$ in $W$.
Again $Z'\in\M$, and all comparison arrows lie in $\M$ by
fullness. The same refinement therefore witnesses equality in the
smaller localization, proving faithfulness.
\end{proof}

We will also use the elementary equality criterion for this calculus:
for $f,g:X\to Y$ in $\C$,
\begin{equation}\label{eq:localequality}
 Lf=Lg\quad\Longleftrightarrow\quad
 fs=gs\text{ for some }s:Z\to X\text{ in }W.
\end{equation}
This is the identity-denominator case of the fraction equivalence;
see \cite[Lemma 4.27.14]{StacksLoc}.

\section{Abelian categories and Serre quotients}\label{sec:abelian}

The ordinary homological structure is due to Grandis
\cite[Section~1.9, p.~146]{Grandis92}.
Lemmas~\ref{lem:annihilation} and~\ref{lem:SerreUP}
identify its annihilation subcategories and exact quotients
through universal properties on full hom-categories.
Together with Proposition~\ref{prop:truncation} and strong
pointedness, these identifications also give the
$2$-homological conclusion from Grandis's result.
The argument below verifies the quotient criterion directly
and records the resulting subquotient square.

\subsection{Serre closure and inherited 2-kernels}

A \emph{Serre subcategory} of an abelian category $\A$ is a full replete
subcategory containing zero and closed under subobjects, quotients,
and extensions. It is abelian with the induced exact structure and
its inclusion is exact; see \cite[Lemma 12.10.2]{StacksSerre}.
It is in particular additive and retract-closed: biproducts are split
extensions, and retracts are subobjects. For a collection $\Scl$,
write $\Ser_{\A}(\Scl)$ for the least Serre subcategory containing it,
obtained by intersecting all such subcategories.

Let $F:\A\to\B$ be exact. Its additive 2-kernel $\K_F$ is Serre.
Indeed, applying $F$ to a short exact sequence
$0\to X\to Y\to Z\to0$ shows that $F(Y)=0$ forces
$F(X)=F(Z)=0$, and that $F(X)=F(Z)=0$ forces $F(Y)=0$.
Thus all the additive kernel data already exist with their required
abelian structure. An exact functor landing in $\K_F$ corestricts
exactly, because exactness there is induced from $\A$; all natural
transformations lift through its full inclusion. Consequently
Lemma~\ref{lem:annihilation} gives the 2-kernel in $\AbCat$.
This also proves that $\AbCat\to\Add$ preserves 2-kernels.
The one-object zero category is abelian and zero-valued functors
are exact. Since all natural transformations between exact functors
are allowed, Proposition~\ref{prop:pointedzero} supplies exactly one
$2$-cell in each direction between an exact functor and a parallel
null exact functor. Thus the same category is a strong bizero in
$\AbCat$ under Definition~\ref{def:2pointed}.

For later use, Serre subcategories have all the closure properties (Q1).
Intersections and restriction to a Serre subcategory have the claimed
properties directly. Inverse images are Serre because an exact functor
preserves short exact sequences. For transitivity, if
$\N\subseteq\M\subseteq\A$ are successive Serre inclusions, a
subobject or quotient in $\A$ of $N\in\N$ first lies in $\M$,
and then lies in $\N$. An ambient extension of objects of $\N$
first lies in $\M$, where its exact sequence is still exact, and
then belongs to $\N$.

\subsection{The classical localization input and its additive precursor}

For a Serre subcategory $\N\subseteq\A$, put
\begin{equation}\label{eq:SerreW}
 W_{\N}=\{s\in\Mor(\A)\mid\Ker(s),\Coker(s)\in\N\}.
\end{equation}
The classical Serre quotient theorem states that $W_{\N}$ admits a
calculus of left and right fractions, that
\begin{equation}\label{eq:Serrequ}
 \sq{\A}{\N}=\A[W_{\N}^{-1}]
\end{equation}
is abelian, and that $q_{\N}:\A\to\sq{\A}{\N}$ is exact and
universal among exact functors annihilating $\N$. References are
\cite[Chapter III, Section 1]{Gabriel} and
\cite[Lemma 12.10.6, Tag 02MS]{StacksSerreQuotient}.
We use this classical existence theorem, and now verify the additional
properties needed for 2-homologicity.

The additive quotient already constructed in Section~\ref{sec:additive}
gives a canonical factorisation
\begin{equation}\label{eq:Serreadditivefactor}
 \A\xrightarrow{a_{\N}}\aq{\A}{\N}
 \xrightarrow{\ell_{\N}}\sq{\A}{\N},
 \qquad q_{\N}=\ell_{\N}a_{\N}.
\end{equation}
Indeed $q_{\N}$ is additive and kills $\N$, so
Lemma~\ref{lem:additivequotient} supplies $\ell_{\N}$ together with
its full natural-transformation property. More precisely,
$\ell_{\N}$ exhibits the Serre quotient as the \emph{additive
localization} of $\aq{\A}{\N}$ at the images of $W_{\N}$.
To see this, precompose an additive functor out of $\aq{\A}{\N}$
with $a_{\N}$. Inverting the specified images is exactly inverting
$W_{\N}$, so it factors through the localization
\eqref{eq:Serrequ}. The factor preserves the finite biproduct
diagrams obtained by applying $q_{\N}$ to those of $\A$, because
its composite with $q_{\N}$ is the original additive functor.
Since $q_{\N}$ is additive and the identity on objects, these
supply biproducts for every finite family in $\sq{\A}{\N}$.
A functor between additive categories preserving finite biproducts
is additive. Lemmas~\ref{lem:additivequotient} and
\ref{lem:localizationNat} give the same statement on natural
transformations.

This factorisation is the precise reuse of the additive 2-cokernel;
$\ell_{\N}$ is generally necessary and is not asserted to be an
equivalence. In particular, the additive quotient need not itself be
an abelian category with an exact quotient functor.

\begin{lemma}[The full 2-dimensional Serre quotient property]\label{lem:SerreUP}
For a Serre $\N\subseteq\A$, the quotient has
$\Ker(q_{\N})=\N$, and precomposition induces
\begin{equation}\label{eq:SerreUP}
 \AbCat(\sq{\A}{\N},\B)\simeq
 \{H\in\AbCat(\A,\B)\mid H(\N)\iso0\}
\end{equation}
as an equivalence of categories with all natural transformations.
\end{lemma}
\begin{proof}
If $X\in\N$, then $0\to X$ belongs to $W_{\N}$, so $q_{\N}X$
is zero. Conversely, if $q_{\N}X$ is zero, then
$q_{\N}(1_X)=q_{\N}(0)$. By \eqref{eq:localequality}, there is
$s:Z\to X$ in $W_{\N}$ with $s=0$. Its cokernel is $X$, so
$X\in\N$. Thus the quotient has the claimed annihilator.

The classical theorem gives the factorisation of any exact $H$
annihilating $\N$. For clarity, exactness of the factor is not an
extra assumption: a quotient morphism is $q(f)q(s)^{-1}$, so its
kernel and cokernel are, up to the source isomorphism $q(s)$, the
images under $q$ of the kernel and cokernel of $f$. Since $H$ and
$q$ are exact, the factor preserves these. Finally,
Lemma~\ref{lem:localizationNat} gives existence and uniqueness of the
descendant of every natural transformation. This proves
\eqref{eq:SerreUP}, pseudonaturally in $\B$.
\end{proof}

\begin{corollary}\label{cor:Serrecoker}
Every exact $F:\B\to\A$ has the 2-cokernel
\begin{equation}\label{eq:Serrecoker}
 \A\longrightarrow\sq{\A}{\Ser_{\A}(F\B)}.
\end{equation}
Thus the normal subcategories in $\AbCat$ are precisely the Serre
subcategories, and the normal quotients are precisely their Serre
quotients, up to equivalence.
\end{corollary}
\begin{proof}
An exact functor kills $F\B$ if and only if it kills its Serre closure,
because its annihilator is Serre. Apply Lemma~\ref{lem:SerreUP}.
The kernel calculation in that lemma and the previously constructed
2-kernels give both normal-form assertions.
\end{proof}

\subsection{Nested Serre subquotients and the homology axiom}

Let $\N\subseteq\M\subseteq\A$ be Serre subcategories. The
additive subquotient square \eqref{eq:additivesquare} already exists.
Applying the additional localizations gives a commutative diagram
\begin{equation}\label{eq:Serrelocdiagram}
\begin{tikzcd}[column sep=large,row sep=large]
 \aq{\M}{\N} \arrow[r,tail] \arrow[d,"\ell_{\N}^{\M}"'] &
 \aq{\A}{\N} \arrow[d,"\ell_{\N}^{\A}"]\\
 \sq{\M}{\N} \arrow[r,"j"'] & \sq{\A}{\N}.
\end{tikzcd}
\end{equation}
The bottom functor is exact by the Serre quotient property. It remains
to prove it is fully faithful.

\begin{proposition}\label{prop:Serresubquotient}
The induced exact functor
$j:\sq{\M}{\N}\to\sq{\A}{\N}$ is fully faithful. It is the
2-kernel of the induced exact functor
$p:\sq{\A}{\N}\to\sq{\A}{\M}$.
\end{proposition}
\begin{proof}
For a morphism of $\M$, its kernel and cokernel computed in $\M$
are the same as in $\A$, up to the canonical isomorphisms. Thus
$W_{\N}^{\M}=W_{\N}^{\A}\cap\Mor(\M)$.
Consider $s:Z\to X$ in $W_{\N}^{\A}$ with $X\in\M$. Its image
is a subobject of $X$, so belongs to $\M$; its kernel belongs to
$\N\subseteq\M$. The short exact sequence
\[
 0\longrightarrow\Ker(s)\longrightarrow Z
 \longrightarrow\im(s)\longrightarrow0
\]
then puts $Z$ in $\M$. This is exactly the roof-closure condition
\eqref{eq:roofclosed}. Lemma~\ref{lem:roofs} proves full faithfulness
of $j$.

The functor $p$ exists because $q_{\M}$ kills $\N$. On objects,
$p(q_{\N}A)=q_{\M}A$, which is zero exactly when $A\in\M$.
Hence $j$ has the full annihilator of $p$ as its essential image.
A fully faithful exact functor onto a Serre subcategory is an exact
equivalence onto that subcategory: its inverse preserves the abelian
kernel and cokernel universal properties. The inherited 2-kernel
construction now identifies $j$ as $\twoker(p)$.
\end{proof}

\begin{theorem}\label{thm:abelian}
The 2-category $\AbCat$ is 2-pointed and 2-homological. For
$\N\subseteq\M\subseteq\A$ Serre, its homology factorisation is
\[
 \M\xrightarrow{\text{2-cokernel}}\sq{\M}{\N}
 \xrightarrow{\text{2-kernel}}\sq{\A}{\N}.
\]
Moreover
\[
 (\sq{\A}{\N})/_{\!\mathrm S}(\sq{\M}{\N})
 \simeq\sq{\A}{\M}.
\]
Here the denominator denotes the full replete essential image of
$\sq{\M}{\N}\to\sq{\A}{\N}$, equivalently the annihilator of
$\sq{\A}{\N}\to\sq{\A}{\M}$. The corresponding subquotient
square is a bipullback and a bipushout.
\end{theorem}
\begin{proof}
Use the canonical locally faithful forgetful $2$-functor
$U:\AbCat\to\Pt$, which sends the one-object zero category
to $\mathbf1$. Realize each normal subcategory with its
inherited abelian structure and inclusion. An equivalence
of underlying categories preserves kernel and cokernel
universal properties, so an exact functor whose underlying
functor is an equivalence has an exact quasi-inverse.
Thus $U$ reflects equivalences.

The closure and kernel arguments in the first subsection establish
(Q1)--(Q2) of Theorem~\ref{thm:criterion}. Lemma~\ref{lem:SerreUP}
establishes (Q3), and Proposition~\ref{prop:Serresubquotient} establishes
(Q4). Strong pointedness was inherited from the pointed case. The
criterion and Corollary~\ref{cor:third} prove the assertions.
\end{proof}

\section{Linear categories over a commutative rig}\label{sec:rig-linear}

The additive and abelian results of Theorems~\ref{thm:additive}
and~\ref{thm:abelian} admit extensions to categories enriched in
semimodules over a fixed commutative rig. We first treat enriched
categories with a zero object, both with and without finite
biproducts, and then impose Puppe exactness. The latter construction
uses exact linear localizations and includes linear abelian
categories as its finite-biproduct case. The proofs again use the
quotient criterion of Theorem~\ref{thm:criterion}.

Fix a commutative rig $R$, that is, a commutative semiring with
$0$ and $1$. An $R$-linear category is a category whose hom-sets
are unital $R$-semimodules and whose composition is $R$-bilinear.
Thus each hom-set is a commutative monoid under addition, equipped
with an $R$-action preserving addition and zero in both variables.
No additive inverses are assumed. An $R$-linear functor preserves
the semimodule structures on hom-sets.

Write $\mathbf{Lin}_{R,0}$ for the $2$-category of $R$-linear
categories admitting a zero object, $R$-linear functors, and all
natural transformations. Write $\mathbf{Lin}_{R}^{\oplus}$ for
its full sub-$2$-category on the categories admitting finite
biproducts. The size convention of Section~\ref{sec:framework}
remains in force. The zero-object assumption in
$\mathbf{Lin}_{R,0}$ is additional to enrichment: zero morphisms
alone do not supply a zero object.

Every $R$-linear functor preserves zero objects, since
$X\iso0$ is equivalent to $1_X=0$. It also preserves every
existing finite biproduct: the equations
\[
 p_i\iota_j=\delta_{ij},\qquad
 \sum_i\iota_i p_i=1
\]
are preserved by linear functors and characterize a biproduct.
Here $\delta_{ij}$ denotes the appropriate identity or zero
morphism. Consequently both $2$-categories have canonical
locally faithful forgetful $2$-functors to $\Pt$.

For $R=\mathbb N$, an $R$-linear category is precisely a category
enriched in commutative monoids. This is linearity in the sense
of Lawvere~\cite{LawvereLinear}; requiring finite biproducts gives
the semiadditive case. For $R=\mathbb Z$, the categories in
$\mathbf{Lin}_{R}^{\oplus}$ are precisely additive categories.
We prove both versions simultaneously, so that additive
homologicity follows as a specialization.

\subsection{Linear ideal quotients without subtraction}

For a collection $\Scl$ of objects of $\C$, let
$[\Scl]_{\C}(X,Y)$ consist of all finite sums
\begin{equation}\label{eq:rig-ideal}
 \sum_{i=1}^{n} b_i a_i,
 \qquad
 X\xrightarrow{a_i}S_i\xrightarrow{b_i}Y,
 \quad S_i\in\Scl,
\end{equation}
including the empty sum. These are $R$-subsemimodules, stable
under composition on either side; scalars can be absorbed into
the factors. Thus $[\Scl]_{\C}$ is the two-sided $R$-linear
ideal generated by the identities of the objects of $\Scl$.

On each hom-set define
\begin{equation}\label{eq:rig-congruence}
 f\sim_{\Scl}g
 \quad\Longleftrightarrow\quad
 f+u=g+v
 \text{ for some }u,v\in[\Scl]_{\C}(X,Y).
\end{equation}
This is an equivalence relation compatible with addition,
scalar multiplication, and composition. For example, the
equations $f+u=g+v$ and $g+u'=h+v'$ imply
$f+(u+u')=h+(v'+v)$. Compatibility with composition follows
from bilinearity and the ideal property. Moreover,
$\sim_{\Scl}$ is the least such congruence identifying every
element of $[\Scl]_{\C}$ with zero: in any such congruence,
$f+u=g+v$ implies $f\sim g$.

Define $\C/\Scl$ to have the same objects as $\C$ and
hom-semimodules
\[
 (\C/\Scl)(X,Y):=\C(X,Y)/{\sim_{\Scl}}.
\]
Composition and identities are induced from $\C$. The quotient
functor $q_{\Scl}:\C\to\C/\Scl$ is $R$-linear, the identity
on objects, and surjective on hom-sets. It preserves a zero
object and all existing finite biproducts. In particular, the
quotient belongs to the same one of the two $2$-categories as
$\C$.

\begin{lemma}\label{lem:rig-quotient}
For every $X\in\C$,
\begin{equation}\label{eq:rig-annihilation}
 q_{\Scl}(X)\iso0
 \quad\Longleftrightarrow\quad
 1_X+u=v
 \text{ for some }u,v\in[\Scl]_{\C}(X,X).
\end{equation}
For either $\Lcat=\mathbf{Lin}_{R,0}$ or
$\Lcat=\mathbf{Lin}_{R}^{\oplus}$, precomposition induces an
isomorphism of categories
\begin{equation}\label{eq:rig-quotient-up}
 \Lcat(\C/\Scl,\D)\iso
 \{H\in\Lcat(\C,\D)\mid H(S)\iso0
                         \text{ for every }S\in\Scl\},
\end{equation}
where the right-hand side is full, retaining all natural
transformations.
\end{lemma}
\begin{proof}
The object $q_{\Scl}(X)$ is zero exactly when its identity is
zero. Applying~\eqref{eq:rig-congruence} to $1_X$ and $0$
gives~\eqref{eq:rig-annihilation}.

A linear functor $H$ annihilates $\Scl$ exactly when it sends
every element of $[\Scl]_{\C}$ to zero. It then respects
$\sim_{\Scl}$ and descends uniquely by
$\overline H([f])=H(f)$. Conversely, $q_{\Scl}$ annihilates
every object of $\Scl$, since its identity belongs to the ideal.
These constructions give the asserted bijection on functors.
For a natural transformation between two such functors, retain
its components on the unchanged object set. Naturality on a
class $[f]$ is precisely naturality on $f$, proving existence
and uniqueness of the descended transformation. The resulting
isomorphisms are natural in $\D$.
\end{proof}

\subsection{Normal subcategories and the quotient criterion}

\begin{definition}\label{def:rig-normal}
A full replete subcategory $\N\subseteq\C$ is \emph{normal}
if, for every $X\in\C$,
\begin{equation}\label{eq:rig-normal}
 \bigl(1_X+u=v\text{ for some }
       u,v\in[\N]_{\C}(X,X)\bigr)
 \quad\Longrightarrow\quad X\in\N.
\end{equation}
Write $\cl^{R}_{\C}(\Scl)$ for the intersection of all normal
full replete subcategories containing $\Scl$.
\end{definition}

Every normal subcategory contains the zero objects and is
closed under retracts and all existing finite biproducts of its
objects. Indeed, $1_0=0$, a retraction expresses an identity as
one composite through the subcategory, and a biproduct expresses
an identity as a finite sum of such composites. In each case
apply~\eqref{eq:rig-normal} with $u=0$. Thus a normal subcategory
has the inherited linear structure, and has finite biproducts
whenever $\C$ does. These inherited structures and their full
inclusions provide the chosen normal-subcategory data of
Theorem~\ref{thm:criterion}.

\begin{theorem}\label{thm:rig-linear}
For every commutative rig $R$, both $\mathbf{Lin}_{R,0}$ and
$\mathbf{Lin}_{R}^{\oplus}$ are $2$-pointed and
$2$-homological. Their $2$-kernels are, up to equivalence, the
inclusions of the normal subcategories of
Definition~\ref{def:rig-normal}, and their $2$-cokernels are,
up to equivalence, the corresponding linear ideal quotients.
More explicitly,
\begin{equation}\label{eq:rig-closure}
 \cl^{R}_{\C}(\Scl)
 =\{X\in\C\mid 1_X+u=v
       \text{ for some }u,v\in[\Scl]_{\C}(X,X)\},
\end{equation}
as a full subcategory, and a $2$-cokernel of $F:\B\to\C$ is
\[
 \C\longrightarrow\C/\cl^{R}_{\C}(F\B).
\]
\end{theorem}
\begin{proof}
Let $\Lcat$ be either of the two $2$-categories and let
$U:\Lcat\to\Pt$ be its forgetful $2$-functor. The one-object
category with zero endomorphism semimodule is a strong bizero
object. Indeed, linear functors out of it select zero objects,
the functor into it is unique, and a zero-valued functor is a
zero object in every relevant functor category: the required
natural transformations are unique componentwise. The functor
$U$ is locally faithful and sends this bizero object to
$\mathbf1$. We verify the four conditions of
Theorem~\ref{thm:criterion}.

\smallskip\noindent\emph{(Q1).}
Normal subcategories contain zero objects, as observed above.
Arbitrary intersections are normal: a witness in the ideal of
an intersection is a witness in the ideal of each member.
The whole category is normal, giving the empty intersection.
For a linear functor $F:\B\to\C$ and a normal
$\N\subseteq\C$, apply $F$ to an equation as
in~\eqref{eq:rig-normal} whose factors pass through
$F^{-1}(\N)$. Normality of $\N$ then gives
$F(X)\in\N$, proving normality of the inverse image.

For transitivity, suppose $\M$ is normal in $\C$ and $\N$
is normal in $\M$. A witness for $X$ with factors through
$\N$ is also a witness with factors through $\M$, so
$X\in\M$. Fullness now makes the original witness internal
to $\M$, and its normality condition gives $X\in\N$.
Conversely, if $\N\subseteq\M$ are both normal in $\C$,
every witness internal to $\M$ is also a witness in $\C$;
hence $\N$ is normal in $\M$. This proves the restriction
clause as well.

\smallskip\noindent\emph{(Q2).}
For $F:\B\to\C$, let $\K$ be the full subcategory on
the objects $B$ with $F(B)\iso0$. Applying $F$ to a witness
$1_X+u=v$ through $\K$ gives $1_{F(X)}+0=0$.
Thus $F(X)\iso0$, proving that $\K$ is normal.
A linear functor $G:\X\to\B$ with $FG$ null corestricts
uniquely to $\K$. The corestriction is linear, and every
natural transformation lifts uniquely through the full
inclusion. This proves the full hom-category universal
property of the $2$-kernel.

\smallskip\noindent\emph{(Q3).}
For normal $\N$, take the quotient $q_{\N}$ constructed
above. Its annihilator is exactly $\N$, by
\eqref{eq:rig-annihilation} and normality. Its full
hom-category universal property is
Lemma~\ref{lem:rig-quotient}.

\smallskip\noindent\emph{(Q4).}
Let $\N\subseteq\M$ be normal in $\C$. For $X,Y\in\M$,
fullness gives
\begin{equation}\label{eq:rig-restriction}
 [\N]_{\M}(X,Y)=[\N]_{\C}(X,Y),
\end{equation}
because every factorization through an object of $\N$ is
already internal to $\M$. The congruences
\eqref{eq:rig-congruence} therefore agree on these hom-sets.
Consequently the induced linear functor
$\M/\N\to\C/\N$ is fully faithful.
Finally, an $R$-linear functor that is an equivalence of
underlying categories has an $R$-linear quasi-inverse:
transport the inverses of its bijective hom-semimodule maps
along chosen object isomorphisms. All unit and counit
transformations are allowed $2$-cells. Thus $U$ reflects
equivalences.

The criterion applies and proves the asserted homologicity
and the descriptions of kernels and cokernels. To identify
the closure explicitly, construct $q_{\Scl}$ for an arbitrary
collection $\Scl$. By (Q2), its annihilator is normal, and
it contains $\Scl$. If $\N$ is any normal subcategory
containing $\Scl$, then $[\Scl]_{\C}\subseteq[\N]_{\C}$;
hence~\eqref{eq:rig-normal} puts every object of that
annihilator in $\N$. It is therefore the least normal
subcategory containing $\Scl$. Equation~\eqref{eq:rig-annihilation}
gives~\eqref{eq:rig-closure}.
\end{proof}

In particular, nested normal subcategories have the bipullback
and bipushout subquotient square, and satisfy the third
isomorphism theorem, by Corollary~\ref{cor:third}.

\subsection{Semiadditive categories and recovery of the additive case}

\begin{corollary}\label{cor:rig-additive}
The $2$-category of semiadditive categories, functors preserving
finite biproducts, and all natural transformations is
$2$-homological. The $2$-category $\Add$ is also
$2$-homological, and its normal subcategories are precisely
the retract-closed full replete additive subcategories.
Its quotients are the usual additive ideal quotients.
\end{corollary}
\begin{proof}
For $R=\mathbb N$, semimodule enrichment is commutative-monoid
enrichment. In a category with finite biproducts this enrichment
is uniquely determined by the biproduct structure, and a
functor preserves it exactly when it preserves finite
biproducts. Thus the first assertion is the
$\mathbf{Lin}_{\mathbb N}^{\oplus}$ case of the theorem.

For $R=\mathbb Z$, unital semimodules are abelian groups,
since $(-1)x$ is the additive inverse of $x$. Hence
$\mathbf{Lin}_{\mathbb Z}^{\oplus}=\Add$ under the canonical
identification of the structures. In this case
\eqref{eq:rig-congruence} reduces to
\[
 f\sim_{\Scl}g
 \quad\Longleftrightarrow\quad
 f-g\in[\Scl]_{\C}(X,Y),
\]
and~\eqref{eq:rig-closure} reduces to
$1_X\in[\Scl]_{\C}(X,X)$. Writing such an identity as
$1_X=\sum_i b_i a_i$ exhibits $X$ as a retract of
$\bigoplus_i S_i$, and every such retraction gives this
identity. Thus the closure is exactly closure under finite
biproducts and retracts, and the quotient is the ordinary
additive ideal quotient. Thus Theorem~\ref{thm:rig-linear}
recovers Theorem~\ref{thm:additive}.
\end{proof}

\subsection{Puppe exact linear categories}\label{subsec:puppe-linear}

The same enrichment can be combined with Puppe exactness. Fix a
commutative rig $R$, and write $\mathbf{PEx}_{R}$ for the
$2$-category of $R$-linear Puppe exact categories, exact
$R$-linear functors, and all natural transformations. Here
exactness means preservation of zero objects and all kernels
and cokernels. Finite biproducts are not required. Write
$\mathbf{PEx}_{R}^{\oplus}$ for the full sub-$2$-category on
objects admitting finite biproducts. The normal subcategories
for these $2$-categories will be the full replete
thick subcategories. Their quotients are the exact localizations
of Theorem~\ref{thm:puppe-localization}.

\begin{lemma}\label{lem:puppe-linear-localization}
Let $\C$ be $R$-linear and Puppe exact, and let
$\N\subseteq\C$ be thick. The exact localization
\[
 q:\C\longrightarrow\C/\N=\C[W_{\N}^{-1}]
\]
admits a unique $R$-linear enrichment making $q$ linear.
For every $\D\in\mathbf{PEx}_{R}$, precomposition induces
an equivalence of entire hom-categories
\[
 \mathbf{PEx}_{R}(\C/\N,\D)\simeq
 \{H\in\mathbf{PEx}_{R}(\C,\D)
      \mid H(N)\iso0\text{ for every }N\in\N\},
\]
where the right-hand side is full.
\end{lemma}
\begin{proof}
Use the ternary fractions and common-reduction equality
criterion of Theorem~\ref{thm:puppe-localization}. Any finite
family of parallel arrows can be represented with a common
monic source denominator and a common epic target denominator:
\[
 \alpha_i=q(e)^{-1}q(f_i)q(m)^{-1},\qquad
 X\xleftarrow{m}U\xrightarrow{f_i}V\xleftarrow{e}Y.
\]
Indeed, intersections refine the source denominators, and
pushouts provide common quotients for the target denominators.
For $r_i\in R$, define
\begin{equation}\label{eq:puppe-linear-sum}
 \sum_i r_i\alpha_i
 :=q(e)^{-1}q\!\left(\sum_i r_i f_i\right)q(m)^{-1}.
\end{equation}
The empty sum is the image of the zero morphism. A refinement
changes a numerator by $f\mapsto jfi$, which is $R$-linear.
The common-reduction criterion therefore proves independence
of representatives and of the common denominators. The
semimodule axioms follow by placing the finitely many arrows
involved over common denominators and checking them on their
numerators.

Composition with an original arrow is linear. For
precomposition, pull the fixed source denominator back along
that arrow; for postcomposition, push the fixed target
denominator out along it. These constructions are independent
of the numerator, which is changed by composition with a
fixed arrow. They therefore preserve
\eqref{eq:puppe-linear-sum}. Composition with the inverse of
a denominator is the inverse of a bijective linear map on
the corresponding hom-semimodules, and is consequently linear
as well. Since original arrows and inverse denominators
generate the localization, composition is bilinear.

This enrichment makes $q$ linear. It is unique: in any such
enrichment, composition with the denominator isomorphisms is
linear, forcing~\eqref{eq:puppe-linear-sum}. The exact
descendant of an exact linear functor annihilating $\N$ is
linear by the same formula. Finally, all natural
transformations descend by Lemma~\ref{lem:puppe-nat}.
\end{proof}

\begin{theorem}\label{thm:puppe-linear}
For every commutative rig $R$, the $2$-category
$\mathbf{PEx}_{R}$ is $2$-pointed and $2$-homological.
Its $2$-kernels are, up to equivalence, inclusions of full
replete thick subcategories, and its $2$-cokernels are the
corresponding exact linear localizations. In particular, a
$2$-cokernel of $F:\B\to\C$ is
\[
 \C\longrightarrow\C/\thick_{\C}(F\B).
\]
The full sub-$2$-category $\mathbf{PEx}_{R}^{\oplus}$ is
also $2$-homological.
\end{theorem}
\begin{proof}
Use the forgetful $2$-functor
$U:\mathbf{PEx}_{R}\to\Pt$. It is locally faithful.
The one-object zero category is a strong bizero object, since
zero-valued exact linear functors are zero objects of the
relevant hom-categories. Its image under $U$ is $\mathbf1$.
Every full replete thick subcategory inherits its ambient
$R$-linear Puppe-exact structure, and its inclusion is exact
and linear. These give the required normal-subcategory data.

For (Q1), thick subcategories contain zero objects and are
closed under intersections and inverse images of exact
functors. Normality is transitive: a subquotient or extension
of objects in the smaller subcategory first lies in the
larger thick subcategory, where the smaller one's thickness
applies. Restriction follows because a full thick subcategory
computes kernels and cokernels as in the ambient category.

For (Q2), an exact functor has a thick full annihilator, as in
Proposition~\ref{prop:puppe-kernels}. A functor with image in
that annihilator corestricts exactly and linearly, and all
natural transformations lift through its full inclusion.
Thus the entire hom-category $2$-kernel property holds.

For (Q3), Theorem~\ref{thm:puppe-localization} gives an exact
quotient annihilating precisely the specified thick
subcategory. Lemma~\ref{lem:puppe-linear-localization} equips
it with its enrichment and supplies the exact linear
hom-category universal property.

For (Q4), if $\N\subseteq\M$ are thick in $\C$, the
induced functor $\M/\N\to\C/\N$ is fully faithful by
Theorem~\ref{thm:puppe-homological}, and is exact and linear
by the preceding lemma. An exact linear functor that is an
underlying equivalence has a linear quasi-inverse, obtained
by transporting its inverse hom-semimodule isomorphisms.
That quasi-inverse is exact because equivalences preserve
the kernel and cokernel universal properties. Hence $U$
reflects equivalences. Theorem~\ref{thm:criterion} applies.

If the ambient category has finite biproducts, so does every
thick subcategory: the split short exact sequence
$A\to A\oplus B\to B$ gives closure under binary
biproducts, and zero objects give the empty one. The quotient
preserves biproducts because it is linear. Consequently the
same normal subcategories and quotients satisfy the criterion
within the full sub-$2$-category with finite biproducts.
\end{proof}

\begin{remark}[Relation with abelian categories]\label{rem:puppe-linear-abelian}
The absence of a finite-biproduct hypothesis makes this a
strict extension of the abelian setting. For example, for a
field $k$, the full category of vector spaces with objects
$0$ and $k$ is $k$-linear and Puppe exact: every nonzero arrow
is invertible, and every zero arrow factors through $0$.
It has no biproduct $k\oplus k$, so it is not abelian.

If finite biproducts are required, Puppe exactness forces
additive inverses. To see this directly, let
$\Delta:A\to A\oplus A$ be the diagonal and let
$q=\operatorname{coker}\Delta$. Since $\Delta$ is monic,
Puppe exactness makes it a kernel of $q$. Put
$u=q\iota_1$ and $v=q\iota_2$, so $u+v=0$.
If $ux=0$, then $\iota_1x=\Delta y$ for some $y$;
applying the second and first projections gives $y=0$ and
$x=0$. Thus $\Ker u=0$. If $hu=0$, then $hv=0$ follows
from $u+v=0$, whence $hq=0$ and $h=0$. Thus
$\Coker u=0$. Puppe exactness makes $u$ invertible, and
\[
 1_A+u^{-1}v=0.
\]
Composing with any arrow into $A$ supplies its additive
inverse. The category is therefore additive and Puppe exact,
hence abelian.

In particular, for $R=\mathbb N$ the finite-biproduct version
of Theorem~\ref{thm:puppe-linear} is exactly the result for
abelian categories and exact functors. More generally, the
$R$-action on the resulting hom-groups extends uniquely to
the universal ring $R^{\mathrm{gp}}$ obtained by additive
group completion. The finite-biproduct objects are thus
precisely the $R^{\mathrm{gp}}$-linear abelian categories.
\end{remark}

\section{Triangulated categories: the quotient-criterion proof}\label{sec:tri}

Section~\ref{sec:tri-direct} proved the triangulated case by characterizing functors through roofs and checking the two composition laws and the homology axiom directly. Here we give a second proof: the additive annihilation kernel and localization lemma verify the common criterion of Theorem~\ref{thm:criterion}. In particular, nested full faithfulness replaces the separate roof calculations for the homology axiom.

\subsection{The 2-cells and the additive 2-kernel}

Write $\Sigma_{\T}$ for the shift of a triangulated category $\T$.
An exact functor is an additive functor $F:\T\to\U$ with an
invertible natural transformation
\[
 \mu^F:F\Sigma_{\T}\xRightarrow{\sim}\Sigma_{\U}F
\]
such that it takes exact triangles to exact triangles,
using $\mu^F$ for the final arrow. A transformation
$\eta:(F,\mu^F)\Rightarrow(G,\mu^G)$ is
\emph{shift-compatible} when
\begin{equation}\label{eq:exactNat}
 (\Sigma_{\U}\eta)\,\mu^F
   =\mu^G\,(\eta\Sigma_{\T}).
\end{equation}
The 2-category $\Tri$ uses these transformations; $\TriAll$ uses
all natural transformations between the same exact functors.
Both are closed under vertical and horizontal composition. The standard
triangulated notions used here can be found in \cite{Verdier} and
\cite{StacksTri}. These are precisely the conventions of Section~\ref{sec:tri-direct};
the displayed equation recalls the shift condition for this second proof.
We prove the result simultaneously for both conventions.

A \emph{thick} subcategory of $\T$ means a full replete triangulated
subcategory closed under retracts in $\T$. Thus it contains zero,
is stable under $\Sigma^{\pm1}$, satisfies the two-out-of-three
condition on exact triangles, and is retract-closed.
It is additive: a biproduct is the middle term of a split triangle.
Write $\thick_{\T}(\Scl)$ for the smallest such subcategory containing
$\Scl$. It is obtained by intersection. This is sometimes called a
strictly full saturated triangulated subcategory; see
\cite[Section 13.6]{StacksTri}.

\begin{lemma}\label{lem:trikernel}
For an exact $F:\T\to\U$, the full subcategory
$\K_F=\{X\mid FX\iso0\}$ is thick. Its inclusion is a 2-kernel
in both $\Tri$ and $\TriAll$. The one-object zero triangulated
category is a strong bizero in both 2-categories.
\end{lemma}
\begin{proof}
The bizero proof is again the pointed proof: a zero-valued functor
has a unique shift isomorphism and all its image triangles are zero
triangles. For any exact $G$ and any parallel zero-valued $Z$, the
pointed proof supplies unique natural transformations $G\Rightarrow Z$
and $Z\Rightarrow G$. Equation~\eqref{eq:exactNat} holds for the
first because its two sides have zero codomain, and for the second
because they have zero domain. Thus null functors are zero objects
in the hom-categories for both conventions, proving the strong
bizero assertion in its full two-sided form. For either
$\Lcat=\Tri$ or $\Lcat=\TriAll$, its canonical forgetful
$U:\Lcat\to\Pt$ is locally faithful and sends this bizero
to $\mathbf1$, so the ambient assumptions of
Section~\ref{sec:criterion} hold.

The additive result already gives the full additive retract-closed
annihilator. Shift stability follows from $\mu^F$. If two terms of an
exact triangle are annihilated, apply $F$ and rotate the triangle
if necessary: the third term is also zero. Thus $\K_F$ is thick with
its induced triangulation. This is the usual object-kernel calculation;
compare \cite[Lemma 13.6.2, Tag 05RC]{StacksTriKernel}.

Take $\B'=\K_F$ with this triangulation and $i$ its inclusion.
These give $U(\B')=\mathcal K_U(F)$ and the required underlying
full inclusion. An exact functor whose image lies in $\K_F$
corestricts exactly: its shift isomorphisms and image triangles lie
in the full subcategory with its induced structure. All natural
transformations lift uniquely through the inclusion. If a transformation
is required to satisfy \eqref{eq:exactNat}, its lift does so because
the equation can be tested after the faithful inclusion. These are
precisely the factorization and lifting hypotheses of
Lemma~\ref{lem:annihilation}, which gives both $2$-kernel statements.
\end{proof}

\begin{remark}[The same kernel, viewed additively]
The strong bizero part recovers Proposition~\ref{prop:tri-strongzero}
and Lemma~\ref{lemmanattozeroisexact}: in both proofs, local
pointedness is inherited from the underlying pointed categories,
and shift compatibility is automatic in both directions at a
null functor. The kernel part recovers Theorem~\ref{teorchar2ker1}: the object set and the
natural-transformation lifting are already those of the additive
annihilation kernel. The only additional checks are stability under
shifts and triangles. In particular, no new kernel construction is
needed in passing from \(\Add\) to either triangulated \(2\)-category.
\end{remark}

Thick subcategories satisfy (Q1): intersections, inverse images under
exact functors, and restriction along thick inclusions preserve all the
defining conditions. For transitivity, inherited exact triangles
are the same triangles, and a retraction first lands in the larger full
thick subcategory and then in the smaller one, exactly as in the additive
proof.

\subsection{Verdier localization as the structured additive quotient}

For thick $\N\subseteq\T$, put
\begin{equation}\label{eq:VerdierW}
 W_{\N}=\{s:X\to Y\mid
 X\xrightarrow{s}Y\to C_s\to\Sigma X
 \text{ is exact for some }C_s\in\N\}.
\end{equation}
In the notation of Section~\ref{sec:tri-direct},
$W_{\N}=\mathcal M(\N)$, and $\vq{\T}{\N}$ denotes the
Verdier quotient written $\T/\N$ there.
The condition is independent of the choice of cone, since cones
of a fixed morphism are isomorphic.

The classical Verdier quotient theorem supplies the triangulated category
\begin{equation}\label{eq:Verdierqu}
 \vq{\T}{\N}=\T[W_{\N}^{-1}]
\end{equation}
and an exact quotient $q_{\N}$. Its exact triangles are those
isomorphic to images of exact triangles of $\T$. The class
$W_{\N}$ admits left and right calculi of fractions; every exact
functor annihilating $\N$ descends exactly through the quotient.
See \cite{Verdier} and, for precise modern statements,
\cite[Proposition 13.5.6 and Lemmas 13.6.6--13.6.8]{StacksTri},
especially Tags 05RG and 05RJ. We use this theorem for existence of the
triangulated quotient, not as a substitute for its 2-dimensional
verification below.

Because $\N$ is additive and retract-closed, the additive quotient
of Section~\ref{sec:additive} already gives
\begin{equation}\label{eq:Verdieradditivefactor}
 \T\xrightarrow{a_{\N}}\aq{\T}{\N}
 \xrightarrow{\ell_{\N}}\vq{\T}{\N}.
\end{equation}
The second arrow is the additive localization at $a_{\N}(W_{\N})$,
by the same argument as for \eqref{eq:Serreadditivefactor}.
An arbitrary additive functor killing $\N$ need not invert
$W_{\N}$; an exact functor killing $\N$ does, since its image of a
triangle with cone in $\N$ has zero cone. This is exactly why the
Verdier quotient is the 2-cokernel here rather than the bare additive
quotient.

\begin{lemma}[The full 2-dimensional Verdier quotient property]\label{lem:VerdierUP}
For thick $\N\subseteq\T$, one has $\Ker(q_{\N})=\N$ and
\begin{equation}\label{eq:VerdierUP}
 \Tri(\vq{\T}{\N},\U)\simeq
 \{H\in\Tri(\T,\U)\mid H(\N)\iso0\}.
\end{equation}
The same statement holds with $\TriAll$ in place of $\Tri$.
\end{lemma}
\begin{proof}
If $X\in\N$, the map $0\to X$ belongs to $W_{\N}$ and its
inversion makes $q_{\N}X$ zero. Conversely, if $q_{\N}X$ is zero,
then \eqref{eq:localequality} gives a zero morphism
$s:Z\to X$ in $W_{\N}$. A cone of the zero morphism is
$X\oplus\Sigma Z$, so this object belongs to $\N$. Since $X$ is
its retract, thickness gives $X\in\N$. This recovers the standard
kernel-of-Verdier-quotient result; compare
\cite[Lemma 13.6.9, Tag 05RK]{StacksTriQuotientKernel}.

Let $H:\T\to\U$ be exact and annihilate $\N$. Its underlying
functor inverts $W_{\N}$ and therefore descends. The descendant is
additive, as in \eqref{eq:Verdieradditivefactor}. Choose the usual
localization model whose shift is induced from that of $\T$, so that
$q_{\N}\Sigma_{\T}=\Sigma_{\vq{\T}{\N}}q_{\N}$.
By Lemma~\ref{lem:localizationNat}, the shift isomorphism $\mu^H$
descends uniquely to the required shift isomorphism of the factor.
It is invertible since all its components are invertible. The factor
preserves exact triangles, because these are precisely the
triangles isomorphic to images of exact triangles of $\T$.
Thus it is an exact factor, with the specified shift data.

Now every natural transformation between two composites with $q_{\N}$
descends uniquely by Lemma~\ref{lem:localizationNat}. If the original
transformation is shift-compatible, equation \eqref{eq:exactNat} holds
for its descendant because it is the same equation on each object,
and $q_{\N}$ is the identity on objects. This proves
\eqref{eq:VerdierUP} for $\Tri$. Omitting that last compatibility
condition proves the assertion for $\TriAll$.
\end{proof}

\begin{remark}[Relation with the direct cokernel proof]
The componentwise descent in Theorem~\ref{teorchar2coker1} is exactly
Lemma~\ref{lem:localizationNat} applied to the Verdier denominators.
That general lemma treats noninvertible natural transformations as
well: only the images of denominators must be invertible. Thus the
two proofs establish the same hom-category universal property.
\end{remark}

\begin{corollary}\label{cor:Verdiercoker}
For an exact $F:\U\to\T$, its 2-cokernel in either triangulated
2-category is
\begin{equation}\label{eq:Verdiercoker}
 \T\longrightarrow\vq{\T}{\thick_{\T}(F\U)}.
\end{equation}
The normal subcategories are precisely the thick subcategories, and
normal quotients are their Verdier quotients, up to equivalence.
\end{corollary}
\begin{proof}
An exact functor annihilates $F\U$ if and only if it annihilates its
thick closure, by Lemma~\ref{lem:trikernel}. Apply
Lemma~\ref{lem:VerdierUP}. The kernel calculation and the full
annihilation 2-kernels give the normal forms.
\end{proof}

\begin{remark}[Two necessary saturations]
The image of an exact functor need not be full, and its full replete
image need not be thick;
$\thick_{\T}(F\U)$ is genuinely required in
\eqref{eq:Verdiercoker}. Also, quotienting by a full triangulated
subcategory which is not thick annihilates its thick closure, not just
its original objects. Conversely, none of this requires adding new
objects by idempotent-completing the Verdier quotient. Relative closure
under retracts and idempotent completeness of an entire category are
different requirements.
\end{remark}

\subsection{Nested Verdier subquotients}

The localization results underlying
Proposition~\ref{prop:Verdiernested} below are classical.
For thick $\N\subseteq\M\subseteq\T$, the full faithfulness
of the canonical functor
\[
 j:\vq{\M}{\N}\longrightarrow\vq{\T}{\N}
\]
is established in
\cite[Chap.~II, Proposition~2.3.1(b)]{Verdier}.
The proof of part~(d) of that proposition identifies its
essential image with the full annihilator of
$p:\vq{\T}{\N}\to\vq{\T}{\M}$.
The third-isomorphism equivalence in
Theorem~\ref{thm:triangulated} below is likewise contained in
part~(c) of Verdier's proposition.

The descent of exact functors and morphisms between them
is also included in
\cite[Chap.~II, Corollary~2.2.11(c)]{Verdier}.
The descriptions used here in terms of $2$-kernels and
$2$-cokernels employ Lemmas~\ref{lem:trikernel}
and~\ref{lem:VerdierUP}, which express annihilation and
localization through universal properties on entire
hom-categories in both $\Tri$ and $\TriAll$.

\begin{proposition}\label{prop:Verdiernested}
For thick $\N\subseteq\M\subseteq\T$, the induced exact functor
$j:\vq{\M}{\N}\to\vq{\T}{\N}$ is fully faithful and is the
2-kernel of $p:\vq{\T}{\N}\to\vq{\T}{\M}$, in both $\Tri$
and $\TriAll$.
\end{proposition}
\begin{proof}
For a morphism of $\M$, a cone computed in $\T$ lies in $\M$
and represents its cone in the induced triangulation. Consequently
$W_{\N}^{\M}=W_{\N}^{\T}\cap\Mor(\M)$. If $s:Z\to X$ belongs
to $W_{\N}^{\T}$ and $X\in\M$, a triangle
\[
 Z\xrightarrow{s}X\longrightarrow C_s\longrightarrow\Sigma Z
\]
has $C_s\in\N\subseteq\M$. Two-out-of-three in $\M$ gives
$Z\in\M$. Thus the roof-closure condition \eqref{eq:roofclosed}
holds. Lemma~\ref{lem:roofs} proves full faithfulness of $j$.
This is exactly the same localization argument as in the abelian case,
with the kernel--image short exact sequence replaced by the cone triangle.
It upgrades the fully faithful additive subquotient functor to the
fully faithful Verdier subquotient functor.

The functor $p$ is induced by the Verdier quotient property because
$q_{\M}$ annihilates $\N$. Its full annihilator has exactly the
objects $X\in\M$, since $p(q_{\N}X)=q_{\M}X$.
Thus $j$ has precisely this full thick subcategory as its essential image.
A fully faithful exact functor is an exact equivalence onto its full
triangulated essential image: transport the shift structure to an
ordinary quasi-inverse; a triangle there is exact if and only
if its image is, since a cone on its first morphism in the source
maps to a cone on that morphism in the target, and the triangle
comparison with two identity components has an invertible third
component. This also makes the unit and counit shift-compatible.
Hence $j$ represents the annihilation 2-kernel of $p$ in either
2-category, by Lemma~\ref{lem:trikernel}.
\end{proof}

\begin{remark}[Where the containment hypothesis enters]
In the direct proof of Theorem~\ref{thm:tri-direct}, the hypothesis
\(\Ker(H)\subseteq\operatorname{Im}(G)\) ensures that intermediate
objects of roofs and common refinements remain in that image.
The condition \(\N\subseteq\M\) above has precisely the same role.
It is encoded once in Lemma~\ref{lem:roofs}, after which
Theorem~\ref{thm:criterion} supplies both composition laws and the
homology factorization. Neither proof gives the unconditional
interchange refuted by Example~\ref{ex:tri-not-diexact}.
\end{remark}

\begin{theorem}\label{thm:triangulated}
Both $\Tri$ and $\TriAll$ are 2-pointed and 2-homological.
For thick $\N\subseteq\M\subseteq\T$, the homology factorisation is
\[
 \M\xrightarrow{\text{2-cokernel}}\vq{\M}{\N}
 \xrightarrow{\text{2-kernel}}\vq{\T}{\N}.
\]
The corresponding square is a bipullback and a bipushout, and
\[
 (\vq{\T}{\N})/_{\!\mathrm V}(\vq{\M}{\N})
 \simeq\vq{\T}{\M}.
\]
Here the denominator denotes the full replete essential image of
$\vq{\M}{\N}\to\vq{\T}{\N}$, equivalently the annihilator of
$\vq{\T}{\N}\to\vq{\T}{\M}$.
\end{theorem}
\begin{proof}
For either convention, use the canonical forgetful
$2$-functor $U:\Lcat\to\Pt$, with
$\Lcat=\Tri$ or $\TriAll$. It is locally faithful and
sends the one-object zero triangulated category to
$\mathbf1$. Realize each thick subcategory with its
inherited triangulation and inclusion. By
Proposition~\ref{propequivinTriang}, an exact functor
whose underlying functor is an equivalence admits an exact
quasi-inverse with shift-compatible unit and counit.
Consequently $U$ reflects equivalences for both conventions.

The first subsection establishes (Q1)--(Q2) of
Theorem~\ref{thm:criterion} and strong pointedness.
Lemma~\ref{lem:VerdierUP} establishes (Q3), with its full 2-cell
property in both conventions. Proposition~\ref{prop:Verdiernested}
establishes (Q4). The criterion and Corollary~\ref{cor:third} apply.

In particular, successive Verdier quotients compose to the quotient
by the inverse image of the second thick subcategory:
\[
 \T\xrightarrow{q}\vq{\T}{\N}
 \xrightarrow{r}(\vq{\T}{\N})/_{\!\mathrm V}\Scl
 \quad\simeq\quad
 \T\longrightarrow\vq{\T}{q^{-1}\Scl}.
\]
The inverse image is thick because $q$ is exact. The proof is the same
hom-category argument already used in the pointed and additive cases.
\end{proof}

\section{The di-exact proof revisited through the criterion}\label{sec:di-revisited}

Theorem~\ref{di:thm:main} was proved directly before the
triangulated case and before the general criterion. We now identify
exactly which part of that proof the criterion simplifies. None of
the earlier fraction or quotient results depends on this revisitation.

\begin{proposition}\label{prop:di-criterion-revisited}
The saturated thick subcategories and saturated exact quotients of
Section~\ref{sec:di-homological} satisfy the hypotheses of
Theorem~\ref{thm:criterion}. Consequently that criterion recovers
Theorem~\ref{di:thm:main}, including the composition laws and the
conditional homology factorization. Corollary~\ref{cor:third}
recovers the third isomorphism and the two universal properties of
the subquotient square.
\end{proposition}
\begin{proof}
Use the canonical forgetful $2$-functor $U:\DiHom\to\Pt$.
It is locally faithful and sends the one-object zero category
to $\mathbf1$; realize the normal subcategories with their
inherited structures and inclusions.

The two-sided strong bizero is already constructed in the direct
proof. For (Q1), Lemma~\ref{di:lem:saturated-properties} proves
closure under intersections and exact inverse images, transitivity,
and restriction to a full saturated thick subcategory. Its induced
structure and exact inclusion are established in the same lemma.
Proposition~\ref{di:prop:2kc} provides the full annihilation
$2$-kernel property, hence (Q2). The exact quotient and its precise
annihilation and natural-transformation universal properties are
Theorem~\ref{di:thm:quotient}, giving (Q3). Finally,
Lemma~\ref{di:lem:nested} gives full faithfulness of the induced
nested quotient inclusion. A $1$-cell whose underlying functor
is an equivalence admits an exact quasi-inverse, since equivalences
preserve the kernel and cokernel universal properties. Thus $U$
reflects equivalences, completing (Q4). Applying the criterion and its corollary proves the
assertions without repeating the direct composition or cone arguments.
\end{proof}

\begin{remark}[The extent of the simplification]
The proof above replaces only the final $2$-categorical argument.
The substantive one-dimensional work is still needed: saturation
supplies arbitrary base-change stability; epidd posets supply the
common-reduction equivalence relation; the three-arrow theorem
identifies that relation with all equalities in the localization;
and the saturated quotient theorem proves preservation of kernels
and cokernels of arbitrary, possibly nonnormal, morphisms. None of
these assertions follows merely from the criterion. In contrast,
once those inputs and nested full faithfulness are available, the
two composition laws and the conditional homology axiom have the
same proof in all the applications of the criterion, including the
rig-linear and Puppe exact linear cases.
\end{remark}

\begin{remark}[Puppe exact and abelian specializations]
For Puppe exact categories, saturation is automatic and all
numerators are normal; their saturated quotients are again Puppe
exact by Theorem~\ref{thm:puppe-localization}. The criterion
therefore also recovers the direct specialization
Theorem~\ref{thm:puppe-homological}. For abelian categories the
normal subcategories are Serre subcategories and the same
localization has its abelian quotient structure. Section~\ref{sec:abelian}
verifies the criterion in that structure and relates it to the
additive quotient. In particular, the inclusion $\AbCat\to\pExact$
preserves $2$-cokernels. For an exact $F:\A\to\B$ between
abelian categories, the thick closure of $F\A$ is its Serre
closure. The corresponding Serre quotient is the same ordinary
localization as the Puppe-exact quotient, and
Theorem~\ref{di:thm:quotient}(iv) supplies its full hom-category
universal property against every Puppe exact target. It therefore
remains a $2$-cokernel of $F$ in $\pExact$.
Neither specialization follows just from being a sub-$2$-category
of $\DiHom$: the quotient must stay in the smaller class, and its
structured hom-category universal property must hold there. Those
assertions are supplied by the respective quotient theorems.
\end{remark}

\section{Comparison and the general homological-category question}\label{sec:comparison}

\subsection{Comparison of the quotient constructions}\label{subsec:quotient-comparison}

For a different comparison of abelian and triangulated quotient
constructions see \cite{Krause}. Here the comparison concerns their
universal properties as $2$-cokernels and the precise effect of forgetting
structure.

The preceding proofs give a precise separation between shared facts
and new input. Strong bizero objects, the two-sided local zero-object
property, uniqueness of invertible nullhomotopies when they exist,
and full annihilation kernels have the same componentwise explanation
in each setting. Section~\ref{sec:pointed} isolates that explanation
in $\Pt$. The additive case adds
closure under biproducts and a hom-group quotient. Over a commutative
rig, the hom-group quotient becomes the hom-semimodule congruence
$f+u=g+v$, with $u,v$ in the generated linear ideal. Normality is
therefore the identity-annihilation condition of
Definition~\ref{def:rig-normal}; it is not in general enough to
require closure under retracts and existing finite biproducts.
The abelian case adds Serre closure and localization at
kernel--cokernel equivalences; the triangulated case adds thick
closure and localization at cone equivalences. For Puppe exact
linear categories, the normal subcategories are thick and the
quotient is the exact Puppe localization, whose linear enrichment
is constructed with common ternary denominators in
Lemma~\ref{lem:puppe-linear-localization}. Its finite-biproduct
case recovers the $R^{\mathrm{gp}}$-linear Serre quotient. The criterion-based abelian and triangulated arguments use the
classical existence and structure theorems for their localizations.
Descent of $2$-cells, composition of normal arrows, and the
conditional homology factorization are proved in the text.
Section~\ref{sec:di-homological} established saturated exact
localizations and the corresponding properties directly for di-exact
homological categories, including preservation under
$\pExact\to\DiHom\to\HomCat$.
Section~\ref{sec:di-revisited} explains precisely how the criterion
then compresses that final $2$-categorical argument.

\begin{proposition}\label{prop:forgetful}
The forgetful 2-functors
\[
 \AbCat\longrightarrow\Add\longrightarrow\Pt,
 \qquad
 \Tri\longrightarrow\Add,
 \qquad
 \TriAll\longrightarrow\Add
\]
preserve the strong bizero, the zero objects and zero morphisms of
the hom-categories, and $2$-kernels. They do not in general preserve
$2$-cokernels.
\end{proposition}
\begin{proof}
The one-object zero category is unchanged, and a null functor
remains zero-valued after forgetting structure. Thus local zero
objects are preserved; so are local zero morphisms, since their
factorizations through null functors are preserved. Preservation of
$2$-kernels was proved by restricting the same full annihilation
construction at each stage. Three examples
explain why the quotient statement cannot be transferred in that way.

\emph{From additive to pointed.}
Let $R=\mathbb Z/4\mathbb Z$ and let $\A$ be a small skeleton of the
additive category generated by the $R$-modules $R$ and
$k=\mathbb Z/2\mathbb Z$, under finite sums and retracts. Let
$\N=\add_{\A}(k)$. In $\operatorname{End}_{\A}(R)=\mathbb Z/4$,
the morphisms factoring through $\N$ are precisely $0$ and $2$.
The pointed quotient identifies $0$ with $2$ but leaves $1$ and $3$
distinct. The additive quotient also identifies $1$ with $3$, because
their difference is $2$. Thus the additive 2-cokernel of
$\N\hookrightarrow\A$ is not its pointed 2-cokernel.

\emph{From abelian to additive.}
Let $\A$ be a small skeleton of finitely generated abelian groups and
$\N$ its finite groups, a Serre subcategory. Every morphism
$\mathbb Z\to N\to\mathbb Z$ with $N$ finite is zero. Thus the
additive quotient retains $\operatorname{End}(\mathbb Z)=\mathbb Z$,
and multiplication by $2$ is not invertible there. In the Serre quotient,
multiplication by $2$ is inverted, since its kernel is zero and its
cokernel is $\mathbb Z/2$. Hence the two cokernels differ.

\emph{From triangulated to additive.}
Fix a field $k$, let $\A$ be the abelian category of finite-dimensional
representations of the quiver $1\to2$, and set $\T=D^b(\A)$.
Regard the representations
\[
 S_1=(k\to0),\qquad S_2=(0\to k),\qquad P_1=(k\xrightarrow{1}k)
\]
as complexes concentrated in degree zero. The short exact sequence
\[
 0\longrightarrow S_2\xrightarrow{s}P_1\longrightarrow S_1
 \longrightarrow0
\]
yields a distinguished triangle. Put $\N=\thick_{\T}(S_1)$.
Thus $s$ becomes invertible in $\vq{\T}{\N}$.

On the other hand, a morphism of representations $P_1\to S_2$ is
zero: commutativity with the arrow $k\xrightarrow{1}k$ forces its
component at vertex $2$ to vanish. Since $P_1$ is projective,
\[
 \operatorname{Hom}_{\T}(P_1,S_2)
 =\operatorname{Hom}_{\A}(P_1,S_2)=0.
\]
This hom-group remains zero in the additive quotient $\aq{\T}{\N}$.
Neither $S_2$ nor $P_1$ becomes a zero object there. Indeed, evaluation
at vertex $2$ is an exact functor $\A\to\mathrm{vect}_k$; its
termwise extension induces an exact functor
$D^b(\A)\to D^b(\mathrm{vect}_k)$ which annihilates $\N$ but sends
both these objects to $k$. It factors through the additive quotient,
so it detects that their images are nonzero. Consequently $s$ has no
inverse in the additive quotient. The Verdier quotient is therefore
not the additive 2-cokernel of $\N\hookrightarrow\T$. This example
works for both choices of triangulated 2-cells.
\end{proof}

\begin{proposition}[Forgetting linear structure and exactness]\label{prop:linear-forgetful}
For every commutative rig $R$, the canonical $2$-functors
\[
 \mathbf{Lin}_{R,0}\longrightarrow\Pt,
 \qquad
 \mathbf{PEx}_{R}\longrightarrow\mathbf{Lin}_{R,0}
\]
preserve the strong bizero, local zero objects and zero morphisms,
and $2$-kernels. They need not preserve $2$-cokernels.
The canonical $2$-functors
\[
 \mathbf{Lin}_{R}^{\oplus}\hookrightarrow\mathbf{Lin}_{R,0},
 \qquad
 \mathbf{PEx}_{R}^{\oplus}\hookrightarrow\mathbf{PEx}_{R},
\]
\[
 \mathbf{PEx}_{R}\longrightarrow\pExact,
 \qquad
 \mathbf{PEx}_{R}^{\oplus}\longrightarrow\AbCat
\]
preserve these pointed structures and both $2$-kernels and
$2$-cokernels.
\end{proposition}
\begin{proof}
The strong bizero and local zero structures are preserved for the
same componentwise reasons as in Proposition~\ref{prop:forgetful}.
All the $2$-kernels are the corresponding full annihilation
subcategories, with inherited structure, so are preserved as well.

For the two full inclusions, the normal subcategory and quotient
constructions are the same as in the larger $2$-category. Normal
subcategories of an object with finite biproducts again have finite
biproducts, and linear quotient functors preserve the biproduct
identities. The proofs of Theorems~\ref{thm:rig-linear}
and~\ref{thm:puppe-linear} therefore identify the same $2$-cokernels
on both sides.

Forgetting enrichment from $\mathbf{PEx}_{R}$ preserves thick
closure and the underlying exact localization, by
Lemma~\ref{lem:puppe-linear-localization}. Thus it preserves
$2$-cokernels. With finite biproducts, the same observation identifies
the underlying quotient with the Serre quotient, proving the last
assertion for $\mathbf{PEx}_{R}^{\oplus}\to\AbCat$.

The failures in the first display already occur at $R=\mathbb Z$.
The additive-to-pointed example in Proposition~\ref{prop:forgetful}
is also a linear-to-pointed example. Its abelian-to-additive example
is an exact-linear-to-linear example: the linear ideal quotient
retains the endomorphism ring $\mathbb Z$ of $\mathbb Z$, whereas
the exact quotient inverts multiplication by $2$. These examples
also give the corresponding failures when the source and target
linear categories are required to have finite biproducts.
\end{proof}

\subsection{The general homological-category question}\label{sec:homcat}

Let $\HomCat$ have as objects small homological categories in the
sense of Definition~\ref{def:grandis-homological}, as $1$-cells
functors preserving zero objects, all
kernels and all cokernels, and as $2$-cells arbitrary natural
transformations. Its full sub-$2$-category $\DiHom$ is
$2$-homological by Theorem~\ref{di:thm:main}. The question in this
subsection concerns arbitrary homological objects, without the
unconditional normality of normal-monomorphism--normal-epimorphism
composites. We retain preservation of all kernels and cokernels;
replacing that requirement by preservation only of short exact
sequences would define a different $2$-category.

\begin{proposition}\label{prop:homcat-kernels}
Both $\HomCat$ and $\DiHom$ are $2$-pointed and admit all
$2$-kernels. For an exact functor $F:\C\to\D$, its $2$-kernel is
the exact inclusion of the full annihilation subcategory
\[
 \Ker(F)=\{X\in\C\mid F(X)\cong0\}.
\]
Their forgetful $2$-functors to $\Pt$ preserve these $2$-kernels.
\end{proposition}
\begin{proof}
The one-object zero category is homological and di-exact; the
zero-preserving functors to and from it are exact. The same
componentwise argument as for $\Pt$ gives a unique natural
transformation in each direction between any exact functor and a
parallel null exact functor. Both are allowed $2$-cells. Thus each
null exact functor is a zero object of its hom-category, proving
$2$-pointedness in the sense of Definition~\ref{def:2pointed}.

Let $u:X\to Y$ have both endpoints in $\Ker(F)$. Since $F$
preserves its kernel, the image of that kernel is the kernel of
a morphism between zero objects, hence has zero domain. The dual
argument applies to its cokernel. Thus $\Ker(F)$ is closed under
ambient kernels and cokernels of its morphisms. It contains zero
objects and is full and replete, so Lemma~\ref{lem:inheritance}
proves that it is homological, and di-exact when $\C$ is di-exact.
Its inclusion preserves every kernel and cokernel.

An exact $G:\E\to\C$ with $FG$ null corestricts to $\Ker(F)$.
The corestriction preserves kernels and cokernels because they are
computed in $\C$. Every transformation between two such
corestrictions is uniquely the corresponding transformation between
their composites with the full inclusion. This proves the full
hom-category universal property.
It is the same annihilation kernel as in $\Pt$.
\end{proof}

\begin{proposition}[Necessary saturation of an annihilation class]\label{prop:homcat-saturation}
Every annihilation subcategory of an exact functor from a pointed
homological category is saturated thick in the sense of
Definition~\ref{di:def:saturated}.
\end{proposition}
\begin{proof}
Lemma~\ref{di:lem:saturated-properties}(ii) was proved for arbitrary
homological categories. It used preservation of all kernels and
cokernels, normal subquotients and extensions, but did not use
di-exactness.
\end{proof}

The next example explains why the saturation step in the quotient
construction is necessary even when the source is di-exact.

\begin{example}[A thick subcategory need not be an annihilation kernel]\label{ex:homcat-obstruction}
Let $\C$ be a small skeleton of the category of at most countable
pointed sets, and let $\N$ be its full subcategory of finite pointed
sets. Kernels are pointed subsets. A cokernel collapses a pointed
subset to the basepoint and makes no other identifications.
Consequently a normal map is precisely a pointed map that is
injective away from its zero fibre. Such maps compose: if
$gf(x)=gf(y)\ne0$, injectivity of $g$ away from its zero fibre gives
$f(x)=f(y)\ne0$, and that of $f$ gives $x=y$.
Proposition~\ref{prop:diexact-equivalences} proves that $\C$ is
di-exact homological. All the indicated constructions remain
countable, so the size restriction does not affect this argument.

The subcategory $\N$ is full, replete, closed under retracts,
normal subquotients, and extensions. For the last assertion, in a
short exact sequence of pointed sets the nonbasepoint part of the
quotient is in bijection with the complement of the kernel subset;
if kernel and quotient are finite then the middle set is finite.
The inclusion $I:\N\to\C$ is exact.

Let $J=\{0,1\}$. For every pointed set $X$ there is a map
$u_X:X\to J$ taking every nonbasepoint to $1$, with $\Ker(u_X)=0$.
Suppose an exact $F:\C\to\D$ annihilates $\N$. Then $F(J)=0$
and $F(\Ker(u_X))=0$, so Proposition~\ref{prop:homcat-saturation}
gives $F(X)=0$ for every $X$. Thus $F$ is null. In particular
$\N$ cannot be the annihilation kernel of any such functor.
In fact the $2$-cokernel of $I$ in both $\HomCat$ and $\DiHom$
is $\C\to0$: the functors annihilating $I$ are exactly the null
functors, and their natural-transformation category is contractible.

In contrast, the pointed quotient $q:\C\to\pq{\C}{\N}$ of
Section~\ref{sec:pointed} does not annihilate any infinite pointed
set. Such a set is not a retract of a finite set. In that quotient
$q(u_X)$ has zero codomain, so its kernel is the identity of $qX$,
whereas $q(\Ker(u_X))=0$. For infinite $X$ this proves that $q$
is not exact. Thus the pointed ideal quotient is not the requisite
homological quotient, even when the original category is di-exact
homological and the subcategory is thick.
\end{example}

In the notation of this example,
\[
 \SatTh_{\C}(\N)=\C,
\]
since a saturated thick subcategory containing $J$ contains every
$X$ by the first saturation implication applied to $u_X$.
Theorem~\ref{di:thm:quotient} therefore gives exactly the zero quotient
already detected by the universal property. Saturation, rather than
passage to the exact core alone, resolves this example. General
localization methods for null ideals and arrow categories
\cite{Fritz} remain relevant to the broader question, but the theorem
for di-exact sources does not by itself settle arbitrary homological
sources.

\begin{proposition}[A sufficient route to the general theorem]\label{prop:homcat-conditional}
Suppose that pointed homological categories admit a system of full
replete admissible subcategories with the following properties:
annihilation kernels are admissible; admissibility has the
intersection, inverse-image, transitivity and restriction properties
in \textup{(Q1)} of Theorem~\ref{thm:criterion}; every admissible
subcategory has an exact normal quotient with the precise
annihilation and hom-category universal properties in \textup{(Q3)};
and nested quotient inclusions are fully faithful, as in
\textup{(Q4)}. Then $\HomCat$ is $2$-homological.
\end{proposition}
\begin{proof}
Proposition~\ref{prop:homcat-kernels} supplies the strong bizero,
including both local zero-object properties, and the kernel
assertion (Q2). The assumptions supply exactly the
remaining hypotheses of Theorem~\ref{thm:criterion}. Its proof
therefore constructs $2$-cokernels, proves both composition laws,
and proves the conditional homology axiom. 
\end{proof}

For arbitrary homological sources, it remains to establish the
required exact quotients and fully faithful nested subquotients.
Theorem~\ref{di:thm:quotient} supplies them for di-exact sources,
but its proof uses unconditional normal factorization twice:
in Lemma~\ref{di:lem:denominators}(iv) and when proving that the
quotient sends normal monomorphisms to monomorphisms. Grandis's
conditional homology axiom does not provide these factorizations for
all the pairs occurring there. We therefore retain the following
conjecture for $\HomCat$, separately from the proved di-exact case.

\begin{conjecture}\label{conj:homcat}
The $2$-category $\HomCat$ of small pointed Grandis homological
categories, exact functors, and arbitrary natural transformations
is $2$-homological.
\end{conjecture}

The proved $2$-homological cases are $\Pt$, $\Add$, $\pExact$,
$\AbCat$, $\Triang$, $\TriAll$, and $\DiHom$. In particular,
Corollary~\ref{di:cor:inclusions} shows that the saturated quotients
of di-exact sources already have their full $2$-cokernel universal
property against arbitrary homological targets. What remains is the
quotient problem for sources that are not di-exact.

\end{document}